\documentclass{amsart}
\usepackage{amsthm,amsmath,amsfonts,mathrsfs,xypic, amssymb, eucal,comment, enumerate}
\numberwithin{equation}{section}

\theoremstyle{plain}
\newtheorem{theorem}{Theorem}[section]
\newtheorem{proposition}[theorem]{Proposition}
\newtheorem{lemma}[theorem]{Lemma}
\newtheorem{corollary}[theorem]{Corollary}

\theoremstyle{definition}
\newtheorem{definition}[theorem]{Definition}
\newtheorem{remark}[theorem]{Remark}
\newtheorem{example}[theorem]{Example}

\title{Subtle invariants of $F$-crystals}
\author[X. Xiao]{Xiao Xiao}
\address{Mathematics Department, Utica College, 1600 Burrstone Road, Utica, NY 13502}
\email{xixiao@utica.edu}

\begin{document}
\begin{abstract}
Vasiu proved that the level torsion $\ell_{\mathcal{M}}$ of an $F$-crystal $\mathcal{M}$ over an algebraically closed field of characteristic $p>0$ is a non-negative integer that is an effectively computable upper bound of the isomorphism number $n_{\mathcal{M}}$ of $\mathcal{M}$ and expected that in fact one always has $n_{\mathcal{M}} = \ell_{\mathcal{M}}$. In this paper, we prove that this equality holds.
\end{abstract}

\maketitle

\section{Introduction}

\subsection{Notations}
Let $p$ be a prime number and $k$ an algebraically closed field of characteristic $p$. For every $k$-algebra $R$, let $W(R)$ be the ring of $p$-typical Witt vectors with coefficients in $R$. For every integer $s \geq 1$, let $W_s(R)$ be the ring of truncated $p$-typical Witt vectors of length $s$ with coefficients in $R$. Let $\sigma_{R}$ be the Frobenius of $W(R)$ and $W_s(R)$. Let $\theta_{R}$ be the Verschiebung of $W(R)$ and $W_s(R)$. Recall that $\sigma_R \theta_R = \theta_R \sigma_R = p$. When there is no confusion of the base ring, we also denote $\sigma_R$ by $\sigma$ and $\theta_R$ by $\theta$. Set $B(R) = W(R)[1/p]$. When $R = k$, $B(k)$ is the field of fractions of $W(k)$. An $F$-crystal $\mathcal{M}$ over $k$ is a pair $(M, \varphi)$ where $M$ is a free $W(k)$-module of finite rank and $\varphi : M \to M$ is a $\sigma$-linear monomorphism. Unless mentioned otherwise, all $F$-crystals in this paper are over $k$. We denote by $\mathcal{M}_R$ the pair $(M \otimes_{W(k)} W(R), \varphi \otimes \sigma_R)$. For every $W(k)$-linear automorphism $g$ of $M$, we denote by $\mathcal{M}(g)$ the $F$-crystal $(M, g\varphi)$ over $k$.

\subsection{Aim and scope}
The isomorphism number $n_{\mathcal{M}}$ of an $F$-crystal $\mathcal{M} = (M, \varphi)$ is the smallest non-negative integer such that for every $W(k)$-linear automorphism $g$ of $M$ with the property that $g \equiv 1_M$ modulo $p^{n_{\mathcal{M}}}$, the $F$-crystal $\mathcal{M}(g)$ is isomorphic to $\mathcal{M}$. This is the generalization of the isomorphism number $n_D$ of a $p$-divisible group $D$ over $k$, which is defined to be the smallest non-negative integer such that for every $p$-divisible group $D'$ over $k$ with the same dimension and codimension as $D$, $D'[p^{n_{D}}]$ and $D[p^{n_{D}}]$ are isomorphic if and only if $D'$ is isomorphic to $D$. The isomorphism numbers of $p$-divisible groups are known to exist as early as in  \cite{Manin:formalgroups}, as a consequence of Theorems 3.4 and 3.5 of the loc. cit. Recently, the isomorphism numbers of $F$-crystals are known to exist by \cite[Main Theorem A]{Vasiu:CBP}.

Traverso proved that $n_D \leq cd+1$ in \cite[Theorem 3]{Traverso:pisa}, where $c$ and $d$ are the codimension and the dimension (respectively) of the $p$-divisible group $D$. He later conjectured that $n_D \leq \min\{c, d\}$ in \cite[Section 40, Conjecture 4]{Traverso:specializations}. In search of optimal upper bounds of $n_D$, the following theorem plays an important role:

\begin{theorem}[{\cite[Theorem 1.6]{Vasiu:traversosolved}}] \label{theorem:maintheoremspecial}
If $D$ is a non-ordinary $p$-divisible group over an algebraically closed field $k$, then its isomorphism number $n_D$ is equal to its level torsion $\ell_D$.
\end{theorem}

For the definition of $\ell_D$, see \cite[Subsection 1.4]{Vasiu:reconstructing} and \cite[Definition 8.3]{Vasiu:traversosolved}. We point out that the two definitions are slightly different. In the case when $D$ is a direct sum of two or more isoclinic ordinary $p$-divisible groups of different Newton slopes, we get $\ell_D = 1$ by the definition in \cite[Subsection 1.4]{Vasiu:reconstructing}; on the other hand, we get $\ell_D=0$ by \cite[Definition 8.3]{Vasiu:traversosolved}. If we assume that $D$ is non-ordinary, then the two definitions coincide. 

Vasiu proved that $n_D \leq \ell_D$ in \cite[Main Theorem A]{Vasiu:reconstructing}, and that $n_{D} = \ell_{D}$ provided $D$ is a direct sum of isoclinic $p$-divisible groups, that is, of $p$-divisible groups whose Newton polygons are straight lines. Later Lau, Nicole and Vasiu proved the equality $n_D = \ell_D$ in \cite{Vasiu:traversosolved} for all $p$-divisible groups $D$ over $k$. Theorem \ref{theorem:maintheoremspecial} builds a bridge between the isomorphism number $n_D$ and other invariants of $D$, such as the level torsion $\ell_D$, the endomorphism number $e_D$, and the coarse endomorphism number $f_D$, which turn out to be all equal by \cite[Theorem 8.11]{Vasiu:traversosolved}; see \cite[Definitions 2.2 and 7.2]{Vasiu:traversosolved} for their definitions. Using Theorem \ref{theorem:maintheoremspecial}, Lau, Nicole and Vasiu were able to find the optimal upper bound of $n_D \leq \lfloor 2cd/(c+d) \rfloor$ (see \cite[Theorem 1.4]{Vasiu:traversosolved}), which provides a corrected version of Traverso's conjecture.

The level torsion $\ell_{\mathcal{M}}$ of an $F$-crystal $\mathcal{M}$ is well-defined; see \cite[Section 1.2]{Vasiu:reconstructing} or Subsection \ref{subsection:leveltorsion} for its definition. Therefore it is natural to ask if the similar equality $n_{\mathcal{M}} = \ell_{\mathcal{M}}$ holds or not in general. As mentioned before, Vasiu has already proved that $n_{\mathcal{M}} \leq \ell_{\mathcal{M}}$ and the equality holds when $\mathcal{M}$ is a direct sum of isoclinic $F$-crystals. He expressed the expectation that the equality is true in general; see the paragraph after \cite[1.3 Main Theorem A]{Vasiu:reconstructing}. In this paper, we confirm this expectation.

\begin{theorem}[Main Theorem] \label{theorem:maintheorem}
If $\mathcal{M}$ is a non-ordinary $F$-crystal over an algebraically closed field $k$, then its isomorphism number $n_{\mathcal{M}}$ is equal to its level torsion $\ell_{\mathcal{M}}$.
\end{theorem}

See Theorem \ref{theorem:maintheorem3} for its proof. The definition of the level torsion $\ell_{\mathcal{M}}$ in our paper is slightly different from the definition in \cite[Subsection 1.2]{Vasiu:reconstructing}; see Remark \ref{remark:leveltorsiondiff}. When $\mathcal{M}$ is a non-ordinary $F$-crystal, the two definitions are exactly the same just as in the case of $p$-divisible groups.

\subsection{On the proof of the Main Theorem} The proof of the Main Theorem uses many ideas from \cite{Vasiu:traversosolved}, \cite{Vasiu:levelm}, and \cite{Vasiu:dimensions}. It involves two major steps:
\vskip 0.1in
\noindent {\bf Step (1)}: Generalize the level torsion $\ell_{\mathcal{M}}$, the homomorphism number $e_{\mathcal{M}}$, and the coarse homomorphism number $f_{\mathcal{M}}$ to $F$-crystals $\mathcal{M}$ over $k$. Then prove that they are all equal via a sequence of inequalities $f_{\mathcal{M}} \leq e_{\mathcal{M}} \leq \ell_{\mathcal{M}} \leq f_{\mathcal{M}}$ that are the generalization of the inequalities $f_D \leq e_D \leq \ell_D \leq f_D$ obtained in \cite{Vasiu:traversosolved}.

The main difficulty in Step (1) is to have the right generalizations of $\ell_{\mathcal{M}}$, $e_{\mathcal{M}}$ and $f_{\mathcal{M}}$ so that they remain unchanged under extensions of algebraically closed fields. This requires the constructions of suitable groups schemes $\mathbf{End}_s(\mathcal{M})$ (resp. $\mathbf{Aut}_s(\mathcal{M})$) whose $k$-valued points are the endomorphisms (resp. automorphisms) of $F$-truncations modulo $p^s$ of $\mathcal{M}$ for all $s \geq 1$. The $F$-truncations modulo $p^s$ of $F$-crystals are the generalization of truncated Barsotti--Tate groups of level $s$ associated to $p$-divisible groups. They are first introduced by Vasiu in \cite{Vasiu:CBP} and will be recalled in Section \ref{section:ftruncation}; see Definition \ref{definition:ftruncation}. We will show that $\ell_{\mathcal{M}}$, $e_{\mathcal{M}}$ and $f_{\mathcal{M}}$ are invariant under extensions of algebraically closed fields. This allows us to generalize the proof in \cite[Section 8]{Vasiu:traversosolved} to our case.
\vskip 0.1in
\noindent {\bf Step (2)}: Prove that $f_{\mathcal{M}} = n_{\mathcal{M}}$ by showing that both $f_{\mathcal{M}}$ and $n_{\mathcal{M}}$ are equal to the smallest number $m$ defined by the property that the image of the natural reduction homomorphism $\pi_{s,1}: \mathbf{End}_s(\mathcal{M}) \to \mathbf{End}_1(\mathcal{M})$ has zero dimension if and only if $s-1 \geq m$.

In Step (2), the main result (see Theorem \ref{theorem:monotonicitygamma}) is to show that $n_{\mathcal{M}}$ is the place where the non-decreasing sequence $(\mathrm{dim}(\mathbf{Aut}_s(\mathcal{M})))_{s \geq 1}$ stabilizes, which generalizes a similar result for $p$-divisible groups in \cite{Vasiu:dimensions}. In order to show this, we construct a group action for each $s \geq 1$ whose orbits parametrize isomorphism classes of $F$-truncations modulo $p^s$; see Subsection \ref{subsection:groupaction}. It turns out that the dimension of the stabilizer of the identity element of this action is equal to the dimension of $\mathbf{Aut}_s(\mathcal{M})$ (Lemma \ref{lemma:connectiondimension}). This allows us to use the machinery of group actions to work with the sequence $(\mathrm{dim}(\mathbf{Aut}_s(\mathcal{M})))_{s \geq 1}$ in a way similar to \cite{Vasiu:dimensions} and \cite{Vasiu:levelm}.

We note that the proof of our Main Theorem does not rely on the known fact that $n_{\mathcal{M}} \leq \ell_{\mathcal{M}}$ proved in \cite{Vasiu:reconstructing}.
\vskip 0.1in
\noindent {\it Notes.} After this manuscript was finished, we learned that Sian Nie had a proof of the fact $\ell_{\mathcal{M}} \leq n_{\mathcal{M}}$  where $\mathcal{M}$ is defined over the ring $k[[\epsilon]]$ of formal power series instead of over $W(k)$; see \cite{Nie:Vasiuconj}. He expressed the hope that the same strategy might be used to prove Theorem \ref{theorem:maintheorem}.

\section{$F$-truncations of $F$-crystals} \label{section:ftruncation}
In this section, we recall $F$-truncations modulo $p^s$ of an $F$-crystal $\mathcal{M}$ over $k$ and provide several equivalent descriptions of homomorphisms and isomorphisms between them.

\subsection{Filtrations of $F$-crystals} \label{subsection:filtrationfcry}

Let $r$ be the rank of $\mathcal{M}$. Throughout this paper, the integers $e_1 \leq \dots \leq e_r$ will always be the Hodge slopes of $\mathcal{M}$ and the integers $f_1 < \dots < f_{t}$ will always be all the distinct Hodge slopes of $\mathcal{M}$; thus $\{f_1, \dots, f_t\} = \{e_1, \dots, e_r\}$ as sets. Clearly $f_1 = e_1$ and $f_t=e_r$. For each integer $s \geq 0$, let $h_s$ be the Hodge number of $\mathcal{M}$, that is, $h_s = \#\{e_i \; | \; e_i = s, 1 \leq i \leq r\}$. Clearly, $h_{f_i} \geq 1$ for all $1 \leq i \leq t$. We say that a $W(k)$-basis $\{v_1, v_2, \dots, v_r\}$ of $M$ is an \emph{$F$-basis} of $\mathcal{M}$ if $\{p^{-e_1}\varphi(v_1), p^{-e_2}\varphi(v_2), \dots, p^{-e_r}\varphi(v_r)\}$ is as well a $W(k)$-basis of $M$. Every $F$-basis of $\mathcal{M}$ is also an $F$-basis of $\mathcal{M}(g)$ for all $g \in \mathrm{GL}_M(W(k))$. For each isomorphism of $F$-crystals $h : \mathcal{M}_1 \to \mathcal{M}_2$ and an $F$-basis $\mathcal{B}$ of $\mathcal{M}_1$, it is easy to see that $h(\mathcal{B})$ is an $F$-basis of $\mathcal{M}_2$. 

For each positive integer $1 \leq j\leq t$, we define $I_j = \{i \; | \; e_i = f_j, 1 \leq i \leq r\}$. For an $F$-basis $\mathcal{B}$ of $\mathcal{M}$, let $\widetilde{F}^j_{\mathcal{B}}(M)$ be the free $W(k)$-submodule of $M$ generated by all $v_i$ with $i \in I_j$. We obtain two direct sum decompositions of $M$ that depend on $\mathcal{B}$ (and thus on $\mathcal{M}$):
\[M = \bigoplus_{j=1}^{t} \widetilde{F}^j_\mathcal{B}(M) = \bigoplus_{j=1}^{t} \frac{1}{p^{f_j}} \varphi(\widetilde{F}^j_\mathcal{B}(M)).\]
For each $1 \leq i \leq t$, by letting $F^i_\mathcal{B}(M) := \bigoplus_{j=i}^t \widetilde{F}^j_\mathcal{B}(M)$, we get a decreasing and exhaustive filtration of $M$
\[F^{\bullet}_\mathcal{B}(M)\; : \; \widetilde{F}^t_\mathcal{B}(M) = F^t_\mathcal{B}(M) \subset F^{t-1}_\mathcal{B}(M) \subset \cdots \subset F^1_\mathcal{B}(M) = M.\]
For each $F^i_\mathcal{B}(M)$, let $\varphi_{F^i_\mathcal{B}(M)} : F^i_\mathcal{B}(M) \to M$ be the restriction of $p^{-f_i}\varphi$ to $F^i_{\mathcal{B}}(M)$. For every integer $s>0$, let $F^{\bullet}_\mathcal{B}(M)_s$ be the reduction modulo $p^s$ of the filtration $F^{\bullet}_\mathcal{B}(M)$, namely
\[F^t_\mathcal{B}(M)/p^sF^t_\mathcal{B}(M)  \subset F^{t-1}_\mathcal{B}(M)/p^sF_\mathcal{B}^{t-1}(M) \subset \cdots \subset F^1_\mathcal{B}(M)/p^sF^1_\mathcal{B}(M).\]
For each $1 \leq i \leq t$, we denote by $\varphi_{F^i_\mathcal{B}(M)}[s]$ the $\sigma$-linear monomorphism $\varphi_{F^i_\mathcal{B}(M)}$ modulo $p^s$, and by $\varphi_{F^{\bullet}_\mathcal{B}(M)}[s]$ the sequence of the $\sigma$-linear monomorphisms $\varphi_{F^i_\mathcal{B}(M)}[s]$ with $1 \leq i \leq t$. By a \emph{filtered $F$-crystal modulo $p^s$} of an $F$-crystal $\mathcal{M}$, we mean a  triple of the form
\[(M/p^sM, F^{\bullet}_\mathcal{B}(M)_s, \varphi_{F_\mathcal{B}^{\bullet}(M)}[s]).\]
Let $\mathcal{M}_1$ and $\mathcal{M}_2$ be two $F$-crystals with the same Hodge polygons as $\mathcal{M}$, $\mathcal{B}_1$ and $\mathcal{B}_2$ two $F$-bases of $\mathcal{M}_1$ and $\mathcal{M}_2$ respectively. By an \emph{isomorphism of filtered $F$-crystals modulo $p^s$} from a filtered $F$-crystal modulo $p^s$ of $\mathcal{M}_1$ to a filtered $F$-crystal modulo $p^s$ of $\mathcal{M}_2$, we mean a $W_s(k)$-linear isomorphism $f: M_1/p^sM_1 \to M_2/p^sM_2$ such that for all $1 \leq i \leq t$ we have $f(F^i_{\mathcal{B}_1}(M_1)/p^sF^i_{\mathcal{B}_1}(M_1)) = F^i_{\mathcal{B}_2}(M_2)/p^sF^i_{\mathcal{B}_2}(M_2)$ and $\varphi_{F^i_{\mathcal{B}_2}(M_2)}[s]  f = f  \varphi_{F^i_{\mathcal{B}_1}(M_1)}[s]$.

\subsection{$F$-truncations} In this subsection, we recall the $F$-truncation modulo $p^s$ of an $F$-crystal defined in \cite[Sect. 3.2.9]{Vasiu:CBP}. It is the generalization of the $D$-truncation $(M/p^sM, \varphi[s], \theta[s])$ of a Dieudonn\'e module $(M, \varphi, \theta)$; see \cite[Sect. 3.2.1]{Vasiu:CBP} for the definition of $D$-truncations. 
\begin{definition} \label{definition:ftruncation}
For every integer $s>0$, the \emph{$F$-truncation modulo $p^s$} of an $F$-crystal $\mathcal{M}$ is the set $F_s(\mathcal{M})$ of isomorphism classes of filtered $F$-crystals modulo $p^s$ of $\mathcal{M}$ as $\mathcal{B}$ varies among all possible $F$-bases of $\mathcal{M}$. Let $\mathcal{M}_1$ and $\mathcal{M}_2$ be two $F$-crystals with the same Hodge polygon. A $W_s(k)$-linear isomorphism $f : M_1/p^sM_1 \to M_2/p^sM_2$ is an \emph{isomorphism of $F$-truncations modulo $p^s$} from $F_s(\mathcal{M}_1)$ to $F_s(\mathcal{M}_2)$ if for every $F$-basis $\mathcal{B}_1$ of $\mathcal{M}_1$, there exists an $F$-basis $\mathcal{B}_2$ of $\mathcal{M}_2$ such that
\[f: (M_1/p^sM_1, F^{\bullet}_{\mathcal{B}_1}(M_1)_s, \varphi_{F^{\bullet}_{\mathcal{B}_1}(M_1)}[s]) \to (M_2/p^sM_2, F^{\bullet}_{\mathcal{B}_2}(M_2)_s, \varphi_{F^{\bullet}_{\mathcal{B}_2}(M_2)}[s])\]
is an isomorphism of filtered $F$-crystals modulo $p^s$.
\end{definition}

Suppose $f$ is an isomorphism of $F$-truncations modulo $p^s$ from $F_s(\mathcal{M})$ to $F_s(\mathcal{M}(g))$. Define a set function $\Gamma_{f,s} : F_s(\mathcal{M}) \to F_s(\mathcal{M}(g))$ as follows: the image of the isomorphism class represented by $(M/p^sM, F^{\bullet}_{\mathcal{B}_1}(M)_s, \varphi_{F^{\bullet}_{\mathcal{B}_1}(M)}[s])$ under $\Gamma_{f,s}$ is the isomorphism class represented by $(M/p^sM, F^{\bullet}_{\mathcal{B}_2}(M)_s, (g\varphi)_{F^{\bullet}_{\mathcal{B}_2}(M)}[s])$ if 
\[f: (M/p^sM, F^{\bullet}_{\mathcal{B}_1}(M)_s, \varphi_{F^{\bullet}_{\mathcal{B}_1}(M)}[s]) \to (M/p^sM, F^{\bullet}_{\mathcal{B}_2}(M)_s, (g\varphi)_{F^{\bullet}_{\mathcal{B}_2}(M)}[s])\]
is an isomorphism of filtered $F$-crystals modulo $p^s$. It is easy to see that this function is well-defined and we shall prove that $\Gamma_{f,s}$ is a bijection of sets in Corollary \ref{corollary:h_sbij}.

The following lemma is a generalization of \cite[Lemma 3.2.2]{Vasiu:CBP} to $F$-crystals for $G = \mathbf{GL}_M$.

\begin{lemma} \label{lemma:F-truncations same isonumber}
For each $F$-crystal $\mathcal{M}$ and every $g \in \mathrm{GL}_M(W(k))$, the following two statements are equivalent:
\begin{enumerate}
\item There exist $h \in \mathrm{GL}_M(W(k))$, $F$-bases $\mathcal{B}_1$ and $\mathcal{B}_2$ of $\mathcal{M}$ and $\mathcal{M}(g)$ respectively, such that the reduction $h[s]$ of $h$ modulo $p^s$ induces an isomorphism
\begin{equation} \label{equation:homftruncationdef}
\mbox{\small $h[s] : (M/p^sM, F^{\bullet}_{\mathcal{B}_1}(M)_s, \varphi_{F_{\mathcal{B}_1}^{\bullet}(M)}[s]) \to (M/p^sM, F^{\bullet}_{\mathcal{B}_2}(M)_s, (g\varphi)_{F^{\bullet}_{\mathcal{B}_2}(M)}[s])$}
\end{equation}
of filtered $F$-crystals modulo $p^s$.
\item There exists an element $g_s \in \mathrm{GL}_M(W(k))$ with the property that it is congruent to $1_M$ modulo $p^s$ such that $\mathcal{M}(g_s)$ is isomorphic to $\mathcal{M}(g)$.
\end{enumerate}
\end{lemma}
\begin{proof}
To prove that (2) implies (1), suppose $h \in \mathrm{GL}_M(W(k))$ is an isomorphism from $\mathcal{M}(g_s)$ to $\mathcal{M}(g)$.
For every $F$-basis $\mathcal{B}$ of $\mathcal{M}(g_s)$, there is an $F$-basis $h(\mathcal{B})$ of $\mathcal{M}(g)$, and the reduction of $h$ modulo $p^s$ is an isomorphism of filtered $F$-crystals modulo $p^s$:
\[h[s]: (M/p^sM, F_{\mathcal{B}}^{\bullet}(M)_s, (g_s\varphi)_{F^{\bullet}_{\mathcal{B}}(M)}[s]) \to (M/p^sM, F_{h(\mathcal{B})}^{\bullet}(M)_s, (g\varphi)_{F_{h(\mathcal{B})}^{\bullet}(M)}[s]).\] 
As $g_s \equiv 1_M$ modulo $p^s$ and $\mathcal{B}$ is also an ordered $F$-basis of $\mathcal{M}$, we have a canonical identification of filtered $F$-crystals modulo $p^s$:
\[\mathrm{id}[s]: (M/p^sM, F_\mathcal{B}^{\bullet}(M)_s, \varphi_{F^{\bullet}_\mathcal{B}(M)}[s]) \cong (M/p^sM, F^{\bullet}_\mathcal{B}(M)_s, (g_s\varphi)_{F^{\bullet}_\mathcal{B}(M)}[s]).\]
Composing the two isomorphisms $h[s] \circ \mathrm{id}[s] = h[s]$, we get the desired isomorphism \eqref{equation:homftruncationdef} by taking $\mathcal{B}_1 = \mathcal{B}$ and $\mathcal{B}_2 = h(\mathcal{B})$.

To prove that (1) implies (2), let $g_s = h^{-1}g\varphi h \varphi^{-1}$. We claim that $g_s$ belongs to $\mathrm{GL}_M(W(k))$, which is equivalent to $h(\varphi^{-1}(M)) \subset \varphi^{-1}(M)$. As $M = \bigoplus_{j=1}^t p^{-f_j} \varphi(\widetilde{F}_{\mathcal{B}_1}^j(M))$, it is enough to show that
\[h (\widetilde{F}_{\mathcal{B}_1}^j(M)) \subset \bigoplus_{i=1}^t p^{\mathrm{max}(0,f_j-f_i)}\widetilde{F}_{\mathcal{B}_1}^i(M) = \varphi^{-1}(p^{f_j}M) \cap M.\] 
Indeed, for each $v \in \widetilde{F}_{\mathcal{B}_1}^j(M) \subset F_{\mathcal{B}_1}^j(M)$, we have $h\varphi_{F_{\mathcal{B}_1}^j(M)}(v) - g\varphi_{F_{\mathcal{B}_2}^j(M)} h(v) \in p^sM$, therefore $h \varphi(v) - g \varphi h (v) \in p^{s+f_j}M$. As $v \in \widetilde{F}^j_{\mathcal{B}_1}(M)$, we know that $\varphi(v) \in p^{f_j}M$ and thus $h \varphi(v) \in p^{f_j}M$. By the last two sentences, we know that $g\varphi h(v) \in p^{f_j}M$, whence $\varphi h(v) \in p^{f_j}M$. This implies that $h(v) \in \varphi^{-1}(p^{f_j}M) \cap M$. As 
\[h^{-1} : (M, g\varphi) \cong (M, h^{-1}g\varphi h) = (M, g_s \varphi),\]
it remains to prove that $g_s$ is congruent to $1_M$ modulo $p^s$. As $\mathcal{B}_2$ is an $F$-basis of $\mathcal{M}(g)$, $h^{-1}(\mathcal{B}_2)$ is an $F$-basis of $\mathcal{M}(g_s)$. We have an isomorphism of filtered $F$-crystals modulo $p^s$ as follows:
\[\mbox{\footnotesize $h^{-1}[s] : (M/p^sM, F^{\bullet}_{\mathcal{B}_2}(M)_s, (g\varphi)_{F^{\bullet}_{\mathcal{B}_2}(M)}[s]) \longrightarrow (M/p^sM, F^{\bullet}_{h^{-1}(\mathcal{B}_2)}(M)_s, (g_s\varphi)_{F^{\bullet}_{h^{-1}(\mathcal{B}_2)}(M)}[s]).$}\]
Composing the isomorphism \eqref{equation:homftruncationdef} with the last isomorphism, we have an isomorphism
\[\mbox{\footnotesize $\mathrm{id}[s] : (M/p^sM, F^{\bullet}_{\mathcal{B}_1}(M)_s, \varphi_{F^{\bullet}_{\mathcal{B}_1}(M)}[s]) \longrightarrow (M/p^sM, F^{\bullet}_{h^{-1}(\mathcal{B}_2)}(M)_s, (g_s\varphi)_{F^{\bullet}_{h^{-1}(\mathcal{B}_2)}(M)}[s]).$}\]
For every $1 \leq j \leq t$, and for each $v \in \widetilde{F}^j_{\mathcal{B}_1}(M)$, we have
\[(g_s \varphi)_{F^j_{h^{-1}(\mathcal{B}_2)}(M)}(v) - \varphi_{F_{\mathcal{B}_1}^j(M)}(v) \in p^sM.\] 
This means that $g_s(p^{-f_j}\varphi(v)) - p^{-f_j}\varphi(v) \in p^{s}M$, that is, $g_s$ fixes every element of $p^{-j}\varphi(\widetilde{F}^j_{\mathcal{B}_1}(M))$ modulo $p^s$. Because $M = \bigoplus_{j=1}^t p^{-f_j} \varphi(\widetilde{F}^j_{\mathcal{B}_1}(M))$, we know that $g_s$ fixes every element of $M$ modulo $p^s$, whence $g_s \equiv 1_M$ modulo $p^s$.
\end{proof}

\begin{proposition} \label{proposition:homftruncation}
For all $g,h \in \mathrm{GL}_M(W(k))$, the reduction of $h$ modulo $p^s$ is an isomorphism from $F_s(\mathcal{M})$ to $F_s(\mathcal{M}(g))$ if and only if $h^{-1}g\varphi h \varphi^{-1} \equiv 1_M$ modulo $p^s$.
\end{proposition}
\begin{proof}
For every $h \in \mathrm{GL}_M(W(k))$, if the reduction of $h$ modulo $p^s$ is an isomorphism from $F_s(\mathcal{M})$ to $F_s(\mathcal{M}(g))$, then $h: \mathcal{M}(g_s) \to \mathcal{M}(g)$ is an isomorphism of $F$-crystals where $g_s \equiv h^{-1}g\varphi h\varphi^{-1} \equiv 1_M$ modulo $p^s$  by Lemma \ref{lemma:F-truncations same isonumber}. 

If $h^{-1}g\varphi h \varphi^{-1} \equiv 1_M$ modulo $p^s$, then there exists $g_s \equiv h^{-1}g\varphi h \varphi^{-1}$ congruent to $1_M$ modulo $p^s$ such that $h$ induces an isomorphism from  $\mathcal{M}(g_s)$ to $\mathcal{M}(g)$. For every $F$-basis $\mathcal{B}$ of $\mathcal{M}$, which is also an $F$-basis of $\mathcal{M}(g_s)$, we get an isomorphism of filtered $F$-crystals modulo $p^s$:
\[h[s] : (M/p^sM, F^{\bullet}_\mathcal{B}(M)_s, \varphi_{F_\mathcal{B}^{\bullet}(M)}[s]) \to (M/p^sM, F^{\bullet}_{h(\mathcal{B})}(M)_s, (g\varphi)_{F_{h(\mathcal{B})}^{\bullet}(M)}[s]). \qedhere\]
\end{proof}

\begin{corollary} \label{corollary:h_sbij}
Let $s$ be a positive integer. We recall that $\Gamma_{f,s} : F_s(\mathcal{M}) \to F_s(\mathcal{M}(g))$ is the function defined by an isomorphism $f$ of $F$-truncations modulo $p^s$ from $F_s(\mathcal{M})$ to $F_s(\mathcal{M}(g))$ (see the paragraph after Definition \ref{definition:ftruncation} for its definition). Then the function $\Gamma_{f,s}$ is a bijection.
\end{corollary}
\begin{proof}
Let $h \in \mathrm{GL}_M(W(k))$ be a preimage of $f \in \mathrm{GL}_M(W_s(k))$ via the canonical surjection $\mathrm{GL}_M(W(k)) \to \mathrm{GL}_M(W_s(k))$. By Proposition \ref{proposition:homftruncation}, we have $h^{-1}g\varphi h \varphi^{-1} \equiv 1_M$ modulo $p^s$. Taking inverses on both hand sides, we have $\varphi h^{-1} \varphi^{-1} g^{-1} h \equiv1_M$ modulo $p^s$. After multiplying $h$ on the left and $h^{-1}$ on the right on both hand sides, we get $h\varphi h^{-1} \varphi^{-1} g^{-1} \equiv 1$ modulo $p^s$, that is, $h\varphi h^{-1} (g\varphi)^{-1} \equiv 1_M$ modulo $p^s$. Hence $h^{-1}$ defines an isomorphism of $F$-truncations modulo $p^s$ from $F_s(\mathcal{M}(g))$ to $F_s(\mathcal{M})$. This implies that $\Gamma_{f,s}$ is a bijection.
\end{proof}

The next corollary justifies that the isomorphism number of $F$-crystals is the right generalization of the isomorphism number of $p$-divisible groups.

\begin{corollary} \label{corollary:correctgeneralization}
Let $t_{\mathcal{M}}$ be the smallest integer such that for all $g \in \mathrm{GL}_M(W(k))$, if $F_{t_{\mathcal{M}}}(\mathcal{M})$ is isomorphic to $F_{t_{\mathcal{M}}}(\mathcal{M}(g))$, then $\mathcal{M}$ is isomorphic to $\mathcal{M}(g)$. We have $t_{\mathcal{M}} = n_{\mathcal{M}}$.
\end{corollary}
\begin{proof}
If $F_{n_{\mathcal{M}}}(\mathcal{M})$ is isomorphic to $F_{n_{\mathcal{M}}}(\mathcal{M}(g))$, then by Lemma \ref{lemma:F-truncations same isonumber}, there exists $g_{n_{\mathcal{M}}} \in \mathrm{GL}_M(W(k))$ with the property that $g_{n_{\mathcal{M}}} \equiv 1_M$ modulo $p^{n_{\mathcal{M}}}$ such that $\mathcal{M}(g_{n_{\mathcal{M}}})$, which is isomorphic to $\mathcal{M}$ by the definition of isomorphism numbers, is isomorphic to $\mathcal{M}(g)$. Thus $t_{\mathcal{M}} \leq n_{\mathcal{M}}$.

Let $g_{t_{\mathcal{M}}} \equiv 1_M$ modulo $p^{t_{\mathcal{M}}}$. By Proposition \ref{proposition:homftruncation}, $1_M[t_{\mathcal{M}}] \in \mathrm{GL}_M(W_{t_{\mathcal{M}}}(k))$ is an isomorphism from $F_{t_{\mathcal{M}}}(\mathcal{M})$ to $F_{t_{\mathcal{M}}}(\mathcal{M}({g_{t_{\mathcal{M}}}}))$. By definition of $t_{\mathcal{M}}$, $\mathcal{M}$ is isomorphic to $\mathcal{M}(g_{t_{\mathcal{M}}})$. Thus $n_{\mathcal{M}} \leq t_{\mathcal{M}}$.
\end{proof}

Proposition \ref{proposition:homftruncation} motivates the following definition of a homomorphism modulo $p^s$ between two $F$-crystals.

\begin{definition} \label{definition:homftruncation}
A $W_s(k)$-linear map $h[s]: M_1/p^sM_1 \to M_2/p^sM_2$ is a homomorphism from $F_s(\mathcal{M}_1)$ to $F_s(\mathcal{M}_2)$ if a preimage $h \in \mathrm{Hom}_{W(k)}(M_1, M_2)$ of $h[s]$ under the canonical surjection $\mathrm{Hom}_{W(k)}(M_1, M_2) \to \mathrm{Hom}_{W_s(k)}(M_1/p^sM_1, M_2/p^sM_2)$ satisfies $\varphi_2 h \varphi_1^{-1} \equiv h$ modulo $p^s$. We call $h$ a lift of $h[s]$ and $h[s]$ a homomorphism modulo $p^s$ from $\mathcal{M}_1$ to $\mathcal{M}_2$.
\end{definition}

\begin{remark}
A homomorphism $h[s]$ modulo $p^s$ between $\mathcal{M}_1$ and $\mathcal{M}_2$ implicitly implies that there exists a lift $h$ of $h[s]$ in $\mathrm{Hom}_{W(k)}(M_1, M_2)$ such that $\varphi_2 h \varphi^{-1}_1$ is also an element in $\mathrm{Hom}_{W(k)}(M_1, M_2)$. Note that $h[s]$ is \emph{not} just a $W_s(k)$-linear homomorphism $h[s] : M_1/p^sM_1 \to M_2/p^sM_2$ such that $h\varphi_1 \equiv \varphi_2 h$ modulo $p^s$, although this is a consequence of the definition but it is not equivalent to the definition.
\end{remark}

\begin{remark}
Note that the definition of an isomorphism between two filtered $F$-crystals modulo $p^s$ requires that the two $F$-crystals have the same Hodge polygon described in Subsection \ref{subsection:filtrationfcry}. In Proposition \ref{proposition:homftruncation} we also require that the two $F$-crystals have the same Hodge polygon. On the other hand, in Definition \ref{definition:homftruncation}, we do not require that the two $F$-crystals have the same Hodge polygon. It is reasonable to ask if $h[s] \in \mathrm{GL}_M(W_s(k))$ and there exists a lift $h \in \mathrm{GL}_M(W(k))$ of $h[s]$ such that $\varphi_2 h \varphi_1^{-1} \equiv h$ modulo $p^s$, do $(M, \varphi_1)$ and $(M, \varphi_2)$ have the same Hodge polygon so that $h[s]$ induces an isomorphism between $F_s(\mathcal{M}_1)$ and $F_s(\mathcal{M}_2)$? The answer is yes because if $\varphi_2 h \varphi_1^{-1} \equiv h$ modulo $p^s$, then we know that $\varphi_2 h \varphi_1^{-1} \in \mathrm{GL}_M(W(k))$. Thus $\varphi_2 h \varphi_1^{-1}(\varphi_1(M)) = \varphi_2 h(M) = \varphi_2(M)$. As a result, $\varphi_2 h \varphi_1^{-1}$ induces an isomorphism from $M/\varphi_1(M)$ to $M/\varphi_2(M)$ and thus $(M, \varphi_1)$ and $(M, \varphi_2)$ have the same Hodge polygon. Therefore if $\mathcal{M}_1$ and $\mathcal{M}_2$ have different Hodge polygons, then $F_s(\mathcal{M}_1)$ and $F_s(\mathcal{M}_2)$ are not isomorphic modulo $p^s$.
\end{remark}

\begin{proposition} \label{proposition:anotherdefinitionofftruncation}
Let $s \geq 1$ be an integer. A homomorphism $h[s] : M_1/p^sM_1 \to M_2/p^sM_2$ is a homomorphism from $F_s(\mathcal{M}_1)$ to $F_s(\mathcal{M}_2)$ if and only if there exists a lift $h$ of $h[s]$ in $\mathrm{Hom}_{W(k)}(M_1, M_2)$ such that for every $x \in M_1\; \backslash\; pM_1$, if $\varphi_1(x) \in p^iM_1\; \backslash\; p^{i+1}M_1$, then $h\varphi_1(x) \equiv \varphi_2h(x)$ modulo $p^{s+i}$. Moreover, if we fix an $F$-basis $\mathcal{B}_1 = \{v_1, v_2, \dots, v_r\}$ of $\mathcal{M}_1$, then the condition ``for every $x\in M_1 \; \backslash \; pM_1$" in the prior sentence can be strengthen to ``for all $x \in \mathcal{B}_1$".
\end{proposition}
\begin{proof}
Let $h[s]$ be a homomorphism from $F_s(\mathcal{M}_1)$ to $F_s(\mathcal{M}_2)$, then there exists $h \in \mathrm{Hom}_{W(k)}(M_1, M_2)$ such that $\varphi_2h\varphi_1^{-1} \equiv h$ modulo $p^s$. Let $x \in M_1 \; \backslash \; pM_1$ be such that $\varphi_1(x) \in p^iM_1 \; \backslash \; p^{i+1}M_1$, whence $\frac{1}{p^i}\varphi_1(x) \in M_1 \; \backslash \; pM_1$. Plugging $\frac{1}{p^i}\varphi_1(x)$ into $\varphi_2 h \varphi_1^{-1} \equiv h$ modulo $p^s$ gives the desired congruence $h\varphi_1(x) = \varphi_2h(x)$ modulo $p^{s+i}$.

Suppose $h[s] \in \mathrm{Hom}(M_1/p^sM_1, M_2/p^sM_2)$ satisfies that  for every $x \in M_1\; \backslash\; pM_1$, if $\varphi_1(x) \in p^iM_1\; \backslash\; p^{i+1}M_1$, then $h\varphi_1(x) \equiv \varphi_2h(x)$ modulo $p^{s+i}$. For every $x \in M_1 \; \backslash \; \{0\}$, there exists $l \geq 0$ such that $x \in p^lM_1 \; \backslash \; p^{l+1}M_1$. We write $\varphi_1^{-1}(x) = p^jx'$ for some $j \in \mathbb{Z}$ and $x' \in M_1 \; \backslash \; pM_1$. Therefore $\varphi_1(x') \in p^{l-j}M_1 \backslash p^{l-j+1}M_1$. Plugging $x' = p^{-j}\varphi_1^{-1}(x)$ into the congruence $\varphi_2 h \equiv h \varphi_1$ modulo $p^{s+l-j}$, we get $\varphi_2 h \varphi_1^{-1}(x) \equiv h (x)$ modulo $p^s$ as $l \geq 0$.

To prove the strengthening part, for all $x \in M_1 \; \backslash \; pM_1$, $x = \sum_{i=1}^r x_iv_i$ for some $x_i \in W(k)$, we have $\varphi_1(x) = \sum_{i=1}^r p^{e_i} \sigma(x_i)w_i$ for some $F$-basis $\{w_1, w_2, \dots, w_r\}$ of $\mathcal{M}_2$. Let $i = \mathrm{min}\{e_i+\mathrm{ord}_p(x_i) \; | \; 1 \leq i \leq r\}$. Then $\varphi_1(x) \in p^iM_1 \; \backslash \; p^{i+1}M_1$. Suppose for every $1 \leq i \leq r$, $h\varphi_1(v_i) - \varphi_1h(v_i) = p^{s+e_i}v_i'$ for some $v_i' \in  M_2 \; \backslash \; pM_2$, we conclude the proof by considering the difference
\begin{align*}
h\varphi_1(x) - \varphi_1h(x) &= \sum_{i=1}^r\sigma(x_i)h(\varphi_1(v_i)) - \sum_{i=1}^r \sigma(x_i)\varphi_1(h(v_i)) \\
&= \sum_{i=1}^r p^{s+e_i}\sigma(x_i)v_i' \in p^{s+i} M_2. \qedhere
\end{align*}
\end{proof}

\begin{corollary} \label{corollary:anotherdefinitionofftruncation}
Let $\mathcal{M}$ be an $F$-crystal over $k$ and let $\mathcal{B} = \{v_1, v_2, \dots, v_r\}$ be an $F$-basis of $\mathcal{M}$. For all $g,h \in \mathrm{GL}_M(W(k))$, the reduction of $h$ modulo $p^s$ is an isomorphism between $F_s(\mathcal{M})$ and $F_s(\mathcal{M}(g))$ if and only if for all $v_i \in \mathcal{B}$ we have $h\varphi_1(v_i) \equiv \varphi_2h(v_i)$ modulo $p^{s+e_i}$ where $e_1 \leq e_2 \leq \cdots \leq e_r$ are the Hodge slopes of $\mathcal{M}$.
\end{corollary}

We denote by $\mathrm{Hom}_s(\mathcal{M}_1, \mathcal{M}_2)$ the (additive) group of all homomorphisms modulo $p^s$ from $\mathcal{M}_1$ to $\mathcal{M}_2$, that is, all homomorphisms from $F_s(\mathcal{M}_1)$ to $F_s(\mathcal{M}_2)$. For $i = 1, 2$, if $h_i[s] \in \mathrm{GL}_M(W_s(k))$ is an automorphism of $F_s(\mathcal{M})$, and $h_i \in \mathrm{GL}_M(W(k))$ is a lift of $h_i[s]$ such that $\varphi h_i \varphi^{-1} \equiv h_i$ modulo $p^s$, then $(h_1h_2)[s]$ is also an automorphism of $F_s(\mathcal{M})$ as $h_1h_2 \in \mathrm{GL}_M(W(k))$ is a lift of $(h_1h_2)[s]$ that satisfies
\[(h_1h_2)^{-1} \varphi (h_1h_2) \varphi^{-1} \equiv h_2^{-1} (h_1^{-1} \varphi h_1 \varphi^{-1}) \varphi h_2 \varphi^{-1} \equiv h_2^{-1} \varphi h_2 \varphi^{-1} \equiv 1 \; \mathrm{modulo} \; p^s.\]
Thus all automorphisms of $F_s(\mathcal{M})$ form an abstract group $\mathrm{Aut}_s(\mathcal{M})$ under composition.

\subsection{$\mathbb{W}_s$ functor} \label{section:wsfunctor} For every affine scheme ${\bf X}$ over $\mathrm{Spec}\,W(k)$, there is a functor $\mathbb{W}_s({\bf X})$ from the category of affine schemes over $k$ to the category of sets defined as follows: For every affine scheme $\mathrm{Spec}\,R$,
\[\mathbb{W}_s({\bf X})(\mathrm{Spec}\,R) := {\bf X}(W_s(R)).\]
If ${\bf X}$ is of finite type over $W(k)$, it is known that this functor is representable by an affine $k$-scheme of finite type (see \cite[p. 639 Corollary 1]{Greenberg:Schemata}), which will be denoted by $\mathbb{W}_s({\bf X})$. If in addition ${\bf X}$ is smooth over $\mathrm{Spec}\,W(k)$, then $\mathbb{W}_s({\bf X})$ is smooth. Indeed, for every $k$-algebra $R$ and an ideal $I$ of $R$ such that $I^2=0$, the kernel of $W_s(R) \to W_s(R/I)$ is of square zero. As ${\bf X}$ is smooth, we get that 
\[\mathbb{W}_s({\bf X})(R) = {\bf X}(W_s(R)) \to {\bf X}(W_s(R/I)) = \mathbb{W}_s({\bf X})(R/I)\]
is surjective by \cite[Ch. 2, Sec. 2, Prop. 6]{Neronmodels}, whence $\mathbb{W}_s({\bf X})$ is smooth by the loc. cit. Suppose $\mathbf{X}$ is a smooth affine group scheme over $\mathrm{Spec}\,W(k)$, then $\mathbb{W}_s(\mathbf{X})$ is a smooth affine group scheme over $k$. The reduction epimorphism $W_{s+1}(R) \to W_s(R)$ naturally induces a smooth epimorphism of affine group schemes over $k$
\[\mathrm{Red}_{s+1,{\bf X}}: \mathbb{W}_{m+1}({\bf X}) \to \mathbb{W}_m({\bf X}).\]
The kernel of $\mathrm{Red}_{s+1,{\bf X}}$ is a unipotent commutative group isomorphic to $\mathbb{G}_a^{\mathrm{dim}({\bf X}_k)}$. Identifying $\mathbb{W}_1({\bf X}) = {\bf X}_k$, an inductive argument shows that $\mathrm{dim}(\mathbb{W}_s({\bf X})) = s \cdot \mathrm{dim}({\bf X}_k)$ and $\mathbb{W}_s({\bf X})$ is connected if and only if ${\bf X}_k$ is connected.

\subsection{Group schemes pertaining to $F$-truncations modulo $p^s$} \label{subsection:groupscheme}
 In this subsection, we construct a smooth (additive) group scheme $\mathbf{Hom}_s(\mathcal{M}_1, \mathcal{M}_2)$ of finite type over $k$ such that its group of $k$-valued points is $\mathrm{Hom}_s(\mathcal{M}_1, \mathcal{M}_2)$, and a smooth (multiplicative) group scheme $\mathbf{Aut}_s(\mathcal{M})$ of finite type over $k$ such that its group of $k$-valued points is $\mathrm{Aut}_s(\mathcal{M})$.

Fix $s \geq 1$. Let $\mathcal{M}_1$ and $\mathcal{M}_2$ be two $F$-crystals over $k$. Let $r_1$ and $r_2$ be the ranks of $M_1$ and $M_2$ respectively. We fix $W(k)$-bases $\mathcal{B}_1$ of $M_1$ and $\mathcal{B}_2$ of $M_2$ (they are not necessarily $F$-bases.) Thus a $W(k)$-linear homomorphism $h: M_1 \to M_2$ corresponds to an $r_2 \times r_1$ matrix $X =  [h]_{\mathcal{B}_1}^{\mathcal{B}_2} = (x_{ij})_{1 \leq i \leq r_2, 1 \leq j \leq r_1}$ with respect to $\mathcal{B}_1$ and $\mathcal{B}_2$. Here and in all that follows we adopt the following convention: for any $v \in M_1$, $[h(v)]_{\mathcal{B}_2} = X[v]_{\mathcal{B}_1}$. The Frobenius of $\mathcal{M}_1$ corresponds to an $r_1 \times r_1$ matrix $U= [\varphi_1]_{\mathcal{B}_1}^{\mathcal{B}_1} = (u_{ij})_{1 \leq i, j \leq r_1}$ with respect to $\mathcal{B}_1$, and the Frobenius of $\mathcal{M}_2$ corresponds to an $r_2 \times r_2$ matrix $V=[\varphi_2]_{\mathcal{B}_2}^{\mathcal{B}_2} = (v_{ij})_{1 \leq i, j \leq r_2}$ with respect to $\mathcal{B}_2$. 

Let $W = (w_{ij})_{1 \leq i, j \leq r_1}$ be the transpose of the cofactor matrix of $U$. We have $w_{ij} \in W(k)$. The matrix representation of $\varphi_2 h \varphi_1^{-1}$ with respect to $\mathcal{B}_1$ and $\mathcal{B}_2$ is $V \sigma(X) \sigma(W/\mathrm{det}(U))$. We would like to find conditions on $X$ so that the reduction of $h$ modulo $p^s$, denoted by $h[s]$, is a homomorphism from $F_s(\mathcal{M}_1)$ to $F_s(\mathcal{M}_2)$. By definition, the condition $\varphi_2h\varphi_1^{-1} \equiv h$ modulo $p^s$ is equivalent to the system of equations
\begin{equation} \label{equation:solvehom}
\frac{1}{\sigma(\mathrm{det}(U))}\sum_{m=1}^{r_1}\sum_{n=1}^{r_2} v_{in} \sigma(x_{nm})\sigma(w_{mj}) \equiv x_{ij} \quad \mathrm{modulo} \; \; p^s,
\end{equation}
for all $1 \leq i \leq r_1$ and $1 \leq j \leq r_2$. Let $l := \mathrm{ord}_p(\mathrm{det}(U))$, and $\mathrm{det}(U)^{-1} = p^{-l}d$ where $d \in W(k) \; \backslash \; pW(k)$. Then the system of equations \eqref{equation:solvehom} is equivalent to 
\begin{equation} \label{equation:solvehom2}
\sum_{m=1}^{r_1}\sum_{n=1}^{r_2} \sigma(d)v_{in} \sigma(x_{nm})\sigma(w_{mj}) \equiv p^l x_{ij} \equiv \theta^l (\sigma^l (x_{ij})) \quad \mathrm{modulo} \; p^{s+l}.
\end{equation}

If $R$ is a perfect ring, two elements $u = (u^{(0)}, u^{(1)}, \dots)$ and $w = (w^{(0)}, w^{(1)}, \dots)$ of $W(R)$ are congruent modulo $p^s$ if and only if $u^{(i)} \equiv w^{(i)}$ for all $0 \leq i \leq s-1$. This is true because $p^s = (\sigma_R \theta_R)^s = \sigma^s_R \theta^s_R$, and $\sigma_R$ is an automorphism of $W(R)$ when $R$ is perfect. Thus over perfect rings, the system of equations \eqref{equation:solvehom2} is equivalent to
\begin{equation} \label{equation:solvehom3}
\sum_{m=1}^{r_1}\sum_{n=1}^{r_2} \sigma(d)v_{in} \sigma(x_{nm})\sigma(w_{mj}) \equiv \theta^l (\sigma^l (x_{ij})) \quad \mathrm{modulo} \; \; \theta^{s+l}(W(R)).
\end{equation}

Let $x_{nm} = (x_{nm}^{(0)}, x_{nm}^{(1)}, \dots)$ and $P_{r,q}$ the polynomial with integral coefficients that computes the $q$-th coordinate of the $p$-typical Witt vector which is a product of $r$ $p$-typical Witt vectors. Then the system of equations \eqref{equation:solvehom3} is equivalent to
\begin{equation} \label{equation:solvehomwitt1}
\sum_{m=1}^{r_1} \sum_{n=1}^{r_2} P_{4, q+l}(\sigma(d), v_{in}, \sigma(x_{nm}), \sigma(w_{mj})) - (x_{ij}^{(q)})^{p^l} = 0
\end{equation}
for all $1 \leq i \leq r_1$, $1 \leq j \leq r_2$, and $0 \leq q \leq s-1$, and the equations
\begin{equation} \label{equation:solvehomwitt2}
\sum_{m=1}^{r_1} \sum_{n=1}^{r_2} P_{4, q}(\sigma(d), v_{in}, \sigma(x_{nm}), \sigma(w_{mj}))=0
\end{equation}
for all $1 \leq i \leq r_1$, $1 \leq j \leq r_2$, and $0 \leq q \leq l-1$. 

For any three non-negative integers $n_1, n_2$ and $n_3$, let $R_{n_1, n_2, n_3}$ be the polynomial $k$-algebra with variables $x_{ij}^{(q)}$ where $1 \leq i \leq n_1$, $1 \leq j \leq n_2$ and $0 \leq q \leq n_3$. Let $\mathfrak{I}$ be the ideal of $R_{r_1,r_2,s+l-1}$ generated by equations \eqref{equation:solvehomwitt1} and \eqref{equation:solvehomwitt2}. Let $\mathbf{Y}_s$ be the scheme theoretic closure of $\mathbf{X}_s = \mathrm{Spec}\,R_{r_1,r_2,s+l-1}/\mathfrak{I}$ under the canonical morphism $\mathrm{Spec}\,R_{r_1,r_2,s+l-1} \to \mathrm{Spec}\,R_{r_1,r_2,s-1}$ induced by the natural inclusion $i:R_{r_1,r_2,s-1} \hookrightarrow R_{r_1,r_2,s+l-1}$. Thus $\mathbf{Y}_s$ is affine and is isomorphic to $\mathrm{Spec}\,R_{r_1,r_2,s-1}/i^{-1}(\mathfrak{I}) =: \mathrm{Spec}\,R_s$.
\begin{itemize}
\item If $s \leq l$, then $i^{-1}(\mathfrak{I})$ is generated by equations \eqref{equation:solvehomwitt2} for all $1 \leq i \leq r_1$, $1 \leq j \leq r_2$ and $0 \leq q \leq s-1$.
\item If $s > l$, then $i^{-1}(\mathfrak{I})$ is generated by equations \eqref{equation:solvehomwitt2} for all $1 \leq i \leq r_1$, $1 \leq j \leq r_2$ and $0 \leq q \leq l-1$, and also equations \eqref{equation:solvehomwitt1} for all $1 \leq i \leq r_1$, $1 \leq j \leq r_2$ and $0 \leq q \leq s-l-1$.
\end{itemize}
For each $k$-algebra $R$ (not necessarily perfect), the set of $R$-valued points $\mathbf{Y}_s(R)$ is set of all $W_s(R)$-linear maps 
\[h[s] :M_1 \otimes_{W_s(k)} W_s(R) \to M_2 \otimes_{W_s(k)} W_s(R)\] 
with the property that there exists a lift 
\[h: M_1 \otimes_{W(k)} W(R) \to M_2 \otimes_{W(k)} W(R))\]
such that for each $x \in M$, if $\varphi_1(x) \in p^iM \; \backslash \; p^{i+1}M$, then we have
\[h \circ (\varphi_1 \otimes_{W(k)} \sigma_{R}) (x \otimes 1_{W(R)}) \equiv (\varphi_2 \otimes_{W(k)} \sigma_{R}) \circ h(x \otimes 1_{W(R)})\]
modulo $M \otimes_{W(k)} \theta^{i+s}(W(R))$. It is clear that $\mathbf{Y}_s(R)$ has a functorial group structure under addition, and thus $\mathbf{Y}_s$ is a group scheme. Let $\mathbf{Hom}_s(\mathcal{M}_1, \mathcal{M}_2) := (\mathbf{Y}_s)_{\mathrm{red}}$. If no confusions can occur, we denote $\mathbf{Hom}_s(\mathcal{M}_1, \mathcal{M}_2)$ by $\mathbf{H}_s$. From the construction of $(R_s)_{\mathrm{red}}$, it is clear that $\mathbf{H}_s$ is a smooth group scheme of finite type over $k$, and $\mathbf{H}_s(k) = \mathrm{Hom}_s(\mathcal{M}_1, \mathcal{M}_2)$.

The definition of $\mathbf{H}_s$ would not be very useful if it would depend on the choices of $\mathcal{B}_1$ and $\mathcal{B}_2$. We now show that $\mathbf{H}_s$ does not depend on the choices of $\mathcal{B}_1$ and $\mathcal{B}_2$. Let $\mathcal{B}_1'$ and $\mathcal{B}_2'$ be other $W(k)$-bases of $M_1$ and $M_2$ respectively. Let $T = (t_{ij})$ be the change of basis matrix from $\mathcal{B}_1$ to $\mathcal{B}_1'$ and $T^{-1} = (t'_{ij})$ be its inverse. Let $S = (s_{ij})$ be the change of basis matrix from $\mathcal{B}_2$ to $\mathcal{B}_2'$ and $S^{-1} = (s'_{ij})$ be its inverse. Let $U' = [\varphi_1]_{\mathcal{B}'_1}^{\mathcal{B}'_1}$ and $V' = [\varphi_2]_{\mathcal{B}'_2}^{\mathcal{B}'_2}$ be the matrix representations of $\varphi_1$ and $\varphi_2$ with respect to $\mathcal{B}_1'$ and $\mathcal{B}_2'$ respectively. We get that $T^{-1}U'\sigma(T) = U$, $S^{-1}V'\sigma(S) = V$ and $TU^{-1}\sigma^{-1}(T^{-1}) = U'^{-1}$. Let $W'$ be the transpose of the cofactor matrix of $U'$, then $W'/\mathrm{det}(U') = U'^{-1}$. Let $Y$ be the $r_2 \times r_1$ matrix $[h]_{\mathcal{B}'_1}^{\mathcal{B}'_2} = (y_{ij})_{1 \leq i \leq r_2, 1 \leq j \leq r_1}$ representing $h$ with respect to $\mathcal{B}'_1$ and $\mathcal{B}'_2$. Therefore we have $X = S^{-1} Y T$. By solving $V'\sigma(Y)\sigma(W'/\mathrm{det}(U')) \equiv Y$ modulo $p^s$, we get a similar system of equations like \eqref{equation:solvehomwitt1} and \eqref{equation:solvehomwitt2}, with $d$ replaced by $d'$, $v_{in}$ replaced by $v'_{in}$, $x_{nm}$ replaced by $y_{nm}$, and $w_{mj}$ replaced by $w'_{mj}$. They generate an ideal $\mathfrak{I}'$ of a polynomial algebra $R'_{r_1,r_2,s+l-1}$ with variables $y_{ij}^{(q)}$. We now construct an isomorphism $\iota : R_{r_1,r_2,s+l-1} \to R'_{r_1,r_2,s+l-1}$ induced by the equality $X = S^{-1}YT$. More precisely, as the $(i,j)$-entry of $S^{-1}YT$ is $\sum_{l,m} s_{il}'y_{lm}t_{mj}$, we define $\iota(x_{ij}^{(q)}) = \sum_{l,m} P_{3,q}(s'_{il},y_{lm},t_{mj})$. It is easy to see that $\iota$ is an isomorphism as its inverse $\eta$ can be constructed by the equality $Y = SXT^{-1}$ in a similar way. Now we show that $\iota$ induces a well-defined homomorphism $\iota: R_{r_1,r_2,s+l-1}/\mathfrak{I} \to R'_{r_1,r_2,s+l-1}/\mathfrak{I}'$. Suppose that $f \in \mathfrak{I}$, then we want to show that $\iota(f) \in \mathfrak{I}'$. This is equivalent to show that if $V\sigma(X)\sigma(U^{-1}) = X$, then $V'\sigma(Y)\sigma(U'^{-1}) = Y$, assuming that $X = S^{-1}YT$. Indeed, we have $TU^{-1}\sigma^{-1}(T^{-1}) = U'^{-1}$, and $S^{-1}V'\sigma(S) = V$, we get
\begin{gather*}
Y = SXT^{-1} = SV\sigma(X)\sigma(U^{-1})T^{-1}= S(S^{-1}V'\sigma(S))\sigma(X)\sigma(U^{-1})T^{-1} =\\
V'\sigma(S)\sigma(X)\sigma(T^{-1}U'^{-1}\sigma^{-1}(T)) T^{-1}= 
V'\sigma(Y)\sigma(U'^{-1})TT^{-1}= V'\sigma(Y)\sigma(U'^{-1}).
\end{gather*}
Thus the induced $\iota$ is well-defined. By the same token, we can show that the inverse $\eta$ also induces a well-defined homomorphism at the level of quotient $k$-algebras. As $\iota$ and $\eta$ are inverses of each other, we know that $R_{r_1,r_2,s+l-1}/\mathfrak{I} \cong R'_{r_1,r_2,s+l-1}/\mathfrak{I}'$. Let $\mathbf{Y}'_s$ be the scheme theoretic closure of $\mathrm{Spec}\,R'_{r_1,r_2,s+l-1}/\mathfrak{I}'$ under the canonical morphism $\mathrm{Spec}\,R'_{r_1,r_2,s+l-1} \to \mathrm{Spec}\,R'_{r_1,r_2,s-1}$ induced by the natural inclusion $i':R'_{r_1,r_2,s-1} \hookrightarrow R'_{r_1,r_2,s+l-1}$. It is clear that $\mathbf{Y}_s$ is isomorphic to $\mathbf{Y}'_s$ as $k$-schemes. To see that they are also isomorphic as $k$-group schemes under addition, it is enough to see that the definition $\iota$ and $\eta$ respect addition because if $X_1 =S^{-1}Y_1T$ and $X_2 = S^{-1}Y_2T$, then $X_1+X_2 = S^{-1}(Y_1+Y_2)T$. Thus the definition of $\mathbf{H}_s$ does not depend on the choice of basis.

If $\mathcal{M}_1 = \mathcal{M}_2 = \mathcal{M}$, then $r_1=r_2=r$. In this case, we denote $\mathbf{Hom}_s(\mathcal{M}_1, \mathcal{M}_2)$ by $\mathbf{End}_s(\mathcal{M})$ or for simplicity $\mathbf{E}_s$ if no confusions can occur.

Now we assume that $\mathcal{M}_1 = \mathcal{M}_2 = \mathcal{M}$ (thus $r_1=r_2=r$) and construct a group scheme $\mathbf{Aut}_s(\mathcal{M})$ whose $k$-valued points is $\mathrm{Aut}_s(\mathcal{M})$. Here we make use of a simple fact of Witt vectors: for any $k$-algebra $R$, an element $x \in W(R)$ is invertible if and only if $x^{(0)}$ is a unit in $R$. Put 
\[T_s := R_{r,r,s-1}[\frac{1}{\mathrm{det}(x_{ij}^{(0)})}]/i^{-1}(\mathfrak{I}).\]
Then $\mathrm{Spec}\,T_s(R)$ contains all the multiplicative invertible elements in $\mathbf{Y}_s(R)$. It is the set of all $W_s(R)$-linear automorphisms 
\[h[s] :M \otimes_{W_s(k)} W_s(R) \to M \otimes_{W_s(k)} W_s(R)\] 
with the property that there exists a lift $h \in \mathrm{GL}_{M \otimes_{W(k)} W(R)}(W(R))$ of $h[s]$ such that for each $x \in M$, if $\varphi(x) \in p^iM \; \backslash \; p^{i+1}M$, then we have
\[h \circ (\varphi \otimes_{W(k)} \sigma_{R}) (x \otimes 1_{W(R)}) \equiv (\varphi \otimes_{W(k)} \sigma_{R}) \circ h(x \otimes 1_{W(R)})\]
modulo $M \otimes_{W(k)} \theta^{i+s}(W(R))$. If $h_1[s], h_2[s] \in \mathrm{Spec}\, T_s(R)$, then $(h_1h_2)[s]$ is in $\mathrm{Spec}\, T_s(R)$. Here $(h_1h_2)[s]$ is $h_1h_2$ modulo $\theta^s$. It coincides with the notation that $h[s]$ is $h$ modulo $p^s$ when $R$ is perfect. Hence $\mathrm{Spec}\, T_s(R)$ has a functorial group structure under composition and thus $\mathrm{Spec}\, T_s$ is a group scheme. Let $\mathbf{A}_s = \mathbf{Aut}_s(\mathcal{M}) := \mathrm{Spec}\,(T_s)_{\mathrm{red}}$. Then $\mathbf{A}_s(k) = \mathrm{Aut}_s(\mathcal{M})$ is the group under composition of automorphisms of $F$-truncations modulo $p^s$ of $\mathcal{M}$. From the construction of $(T_s)_{\mathrm{red}}$, it is clear that $\mathbf{A}_s$ is a smooth group scheme of finite type over $k$ and, as a scheme, it is an open subscheme of $\mathbf{E}_s$. We now study an important invariant  $\gamma_{\mathcal{M}}(s) := \mathrm{dim}(\mathbf{Aut}_s(\mathcal{M}))$ associated to $\mathcal{M}$. As $\mathbf{E}_s$ is smooth, all connected components of $\mathbf{E}_s$ have the same dimension. Therefore $\gamma_{\mathcal{M}}(s) = \mathrm{dim}(\mathbf{E}_s)$.

\begin{proposition} \label{proposition:nonincreasingdelta}
For every $l \geq 0$, the sequence $(\gamma_{\mathcal{M}}(s+l) - \gamma_{\mathcal{M}}(s))_{s \geq 1}$ is a non-increasing sequence of non-negative integers. Therefore, we have a chain of inequalities $0 \leq \gamma_{\mathcal{M}}(1) \leq \gamma_{\mathcal{M}}(2) \leq \cdots.$
\end{proposition}
\begin{proof}
For each pair of integers $t \geq s$, there is a canonical reduction homomorphism $\pi_{t,s}: \mathbf{E}_t \to \mathbf{E}_s$. For every perfect $k$-algebra $R$, and every $h[s] \in \mathbf{E}_s(R)$, there is a lift $h$ of $h[s]$ such that $\varphi h \varphi^{-1} \equiv h$ modulo $p^s$, then $\varphi p^{t-s}h \varphi^{-1} \equiv p^{t-s}h$ modulo $p^t$. Hence we get a monomorphism $p^{t-s} : \mathbf{E}_s \to \mathbf{E}_t$ that sends $h[s]$ to $p^{t-s}h[s]$ at the level of $R$-valued points. For every perfect $k$-algebra $R$ and every $h[t] \in \mathbf{E}_t(R)$, $h[t] = p^{t-s}(h'[s])$ for some $h'[s] \in \mathbf{E}_s(R)$ if and only if $h[s]$ belongs to the kernel of $\pi_{t,t-s}$. Hence we have an exact sequence on the level of $R$-valued points
\[0  \xrightarrow{\; \; \; \; \; \; \;} \mathbf{E}_s(R) \xrightarrow{\; \; \; p^{t-s} \; \; \;} \mathbf{E}_t(R) \xrightarrow{\pi_{t,t-s}} \mathbf{E}_{t-s}(R).\]
The dimension of $\mathrm{Im}(\pi_{t,t-s})$ is equal to $\gamma_{\mathcal{M}}(t) - \gamma_{\mathcal{M}}(s) \geq 0$. Because $\pi_{s+1+l, l} = \pi_{s+l, l } \circ \pi_{s+1+l, s+l}$, $\mathrm{Im}(\pi_{s+1+l,l})$ is a subgroup scheme of $\mathrm{Im}(\pi_{s+1,l})$. Hence the dimension of $\mathrm{Im}(\pi_{s+1+l,l})$, which is $\gamma_{\mathcal{M}}(s+1+l) - \gamma_{\mathcal{M}}(s+1)$, is less than or equal to the dimension of $\mathrm{Im}(\pi_{s+l,l})$, which is $\gamma_{\mathcal{M}}(s+l) - \gamma_{\mathcal{M}}(s)$.
\end{proof}

Recall an $F$-crystal $\mathcal{M}$ is ordinary if its Hodge polygon and Newton polygon coincide. It is well known that the ordinary $F$-crystals over $k$ are precisely those $F$-crystals over $k$ which are direct sums of $F$-crystals of rank $1$. 

\begin{proposition} \label{proposition:ordinaryimpliesdimensionzero}
Let $\mathcal{M}$ be an ordinary $F$-crystal, then $\gamma_{\mathcal{M}}(s) = 0$ for all $s \geq 1$.
\end{proposition}
\begin{proof}
If $\mathcal{M}$ is ordinary, then $\mathcal{M} = \oplus_{i=1}^t \mathcal{M}_i$ where $\mathcal{M}_i$ are isoclinic ordinary $F$-crystals. Thus there exists an $F$-basis $\mathcal{B} = \{v_1, v_2, \dots, v_r\}$ of $M$ such that $\varphi(v_i) = p^{e_i}v_i$ for $1 \leq i \leq r$. The ideal $\mathfrak{I}$ that defines the representing $k$-algebra of $\mathbf{E}_s(\mathcal{M})$ is generated by equations of the following two types:
\begin{enumerate}
\item $\sigma(x_{ij}^{(r)}) - x_{ij}^{(r)}$ for all $r$ and $i, j \in I_l$ for all $1 \leq l \leq t$;
\item $x^{(r)}_{ij}$ for all $r$ and $i, j$ that don't belong to the same $I_l$.
\end{enumerate}
It is clear now that representing $k$-algebra is finite dimensional over $k$. Thus $\mathbf{E}_s$ is of dimension zero, so is $\mathbf{A}_s$.
\end{proof}

\section{Isomorphism classes of $F$-truncations}

In this section, we follow the ideas of \cite{Vasiu:dimensions} and \cite{Vasiu:levelm} to define a group action for each $s \geq 1$ whose orbits parametrize the isomorphism classes of $F_s(\mathcal{M}(g))$ for all $g \in \mathrm{GL}_M(W(k))$. We show that the stabilizer of the identity element of this action has the same dimension as $\mathbf{Aut}_s(\mathcal{M})$, which allows us to study the non-decreasing sequence $(\gamma_{\mathcal{M}}(i))_{i\geq 1}$ via the orbits and the stabilizers of the action. The main result of this section is Theorem \ref{theorem:monotonicitygamma}, which is a partial generalization of \cite[Theorem 1]{Vasiu:dimensions}. It will play an important role in the proof of the Main Theorem in Section \ref{section:proofmaintheorem}.

\subsection{Group schemes}
In this subsection, we will introduce some affine group schemes that are necessary to define the group actions in order to study isomorphism classes of $F$-truncations.

Let $\mathcal{M} = (M, \varphi)$ be an $F$-crystal over $k$. Recall $\mathrm{\bf GL}_M$ is the group scheme over $\mathrm{Spec}\,W(k)$ with the property that for every $W(k)$-algebra $S$, $\mathrm{\bf GL}_M(S)$ is the group of $S$-linear automorphism of $M \otimes_{W(k)} S$. Put $V = M \otimes_{W(k)} B(k)$, then we have canonical identifications
\[\mathrm{\bf GL}_{V} = \mathrm{\bf GL}_M \times_{W(k)} \mathrm{Spec}\,B(k) = \mathrm{\bf GL}_{\varphi^{-1}(M)} \times_{W(k)}\mathrm{Spec}\,B(k).\] 
Let ${\bf G}$ be the scheme theoretic closure of $\mathrm{\bf GL}_V$ in $\mathrm{\bf GL}_M \times \mathrm{\bf GL}_{\varphi^{-1}(M)}$ embedded via the composite homomorphism
\[\mathrm{\bf GL}_V \xrightarrow{\Delta} \mathrm{\bf GL}_V \times \mathrm{\bf GL}_V \to \mathrm{\bf GL}_M \times \mathrm{\bf GL}_{\varphi^{-1}(M)}.\]
For any flat $W(k)$-algebra $S$, $\mathbf{G}(S)$ contains all $h \in \mathbf{GL}_{M \otimes_{W(k)} S}(S)$ with the property that $h(\varphi^{-1}(M) \otimes_{W(k)} S) = \varphi^{-1}(M) \otimes_{W(k)} S$. Let $P_{\bf G}: {\bf G} \to \mathrm{\bf GL}_M$ be the composition of the inclusion and the first projection $\mathbf{GL}_M \times \mathbf{GL}_{\varphi^{-1}(M)} \to \mathbf{\bf GL}_M$. 

Let $\mathcal{B} = \{v_1, v_2, \dots, v_r\}$ be an $F$-basis of $\mathcal{M}$. There are two direct sum decompositions of $M = \bigoplus_{j=1}^{t} \widetilde{F}^j_{\mathcal{B}}(M) =  \bigoplus_{j=1}^{t} p^{-f_j} \varphi(\widetilde{F}^j_{\mathcal{B}}(M))$, which implies that $\varphi^{-1}(M) = \bigoplus_{j=1}^tp^{-f_i}\widetilde{F}^j_{\mathcal{B}}(M)$. With respect to $\mathcal{B}$, the representing $k$-algebras of the following affine group schemes are clear: 
\begin{itemize}
\item $\displaystyle \mathrm{\bf GL}_V = \mathrm{Spec}\,B(k)[x_{ij} \; | \; 1 \leq i, j \leq r][\frac{1}{\mathrm{det}(x_{ij})}]$; 
\item $\displaystyle \mathrm{\bf GL}_M = \mathrm{Spec}\,W(k)[x_{ij} \; | \; 1 \leq i , j \leq r][\frac{1}{\mathrm{det}(x_{ij})}]$;
\item $\displaystyle \mathrm{\bf GL}_{\varphi^{-1}(M)} = \mathrm{Spec}\,W(k)[p^{\delta_{ij}}x_{ij} \; | \; 1 \leq i, j \leq r][\frac{1}{\mathrm{det}(x_{ij})}],$ where $\delta_{ij} = f_l-f_m$ if $i \in I_l$ and $j \in I_m$; see Subsection \ref{subsection:filtrationfcry} for the definition of $I_l$ and $I_m$. Note that $\mathrm{det}(p^{\delta_{ij}}x_{ij}) = \mathrm{det} (x_{ij})$ as for each permutation $\pi$ of $\{1, 2, \dots, r\}$, we have $\prod_{i=1}^r p^{\delta_{i\pi(i)}}x_{i\pi(i)} = \prod_{i=1}^r x_{i\pi(i)}$.
\item $\displaystyle {\bf G} = \mathrm{Spec}\,W(k)[p^{\epsilon_{ij}}x_{ij} \; | \; 1 \leq i, j \leq r][\frac{1}{\mathrm{det}(x_{ij})}],$ where $\epsilon_{ij} = \mathrm{min}(0, \delta_{ij})$. For any affine scheme $\mathbf{H}$, let $R_{\mathbf{H}}$ be the ring such that $\mathbf{H} = \mathrm{Spec} \, R_{\mathbf{H}}$. Let $\mathcal{K}$ be the kernel of the composition
\[R_{\mathbf{GL}_M} \otimes R_{\mathbf{GL}_{\varphi^{-1}(M)}} \to R_{\mathbf{GL}_V} \otimes R_{\mathbf{GL}_V} \to R_{\mathbf{GL}_V}.\]
Then $R_{\mathbf{G}} \cong R_{\mathbf{GL}_M} \otimes R_{\mathbf{GL}_{\varphi^{-1}(M)}} / \mathcal{K}$. It is easy to see that the natural homomorphism
\[R_{\mathbf{GL}_M} \otimes R_{\mathbf{GL}_{\varphi^{-1}(M)}} / \mathcal{K} \to W(k)[p^{\epsilon_{ij}}x_{ij} \; | \; 1 \leq i, j \leq r][\frac{1}{\mathrm{det}(x_{ij})}]\]
is an isomorphism of $W(k)$-algebras.
\end{itemize}

\begin{proposition}
The scheme ${\bf G}$ is a connected smooth, affine group scheme over $\mathrm{Spec}\,W(k)$ of relative dimension $r^2$.
\end{proposition}
\begin{proof}
As $G$ is a principal open subscheme of the affine space $\mathrm{Spec} \, W(k)[p^{\epsilon_{ij}}x_{ij} \; | \; 1 \leq i, j \leq r]$ over $W(k)$, it is affine, smooth, integral and of relative dimension $r^2$. From this the lemma follows.
\end{proof}

Fix an $F$-basis $\mathcal{B}$ of $\mathcal{M}$. If $l \neq m$, let ${\bf G}_{l,m}$ be the maximal subgroup scheme of $\mathbf{GL}_M$ that fixes both 
\[\widetilde{F}^1_\mathcal{B}(M) \oplus \cdots \oplus \widetilde{F}^{m-1}_\mathcal{B}(M)\oplus \widetilde{F}^{m+1}_\mathcal{B}(M) \oplus \cdots \oplus \widetilde{F}^t_\mathcal{B}(M) \; \, \textrm{and} \; \, \widetilde{F}^l_\mathcal{B}(M) \oplus \widetilde{F}^m_\mathcal{B}(M)/\widetilde{F}^l_\mathcal{B}(M).\]
With respect to the $F$-basis $\mathcal{B}$, the (multiplicative) group scheme $\mathbf{G}_{l,m}$ is isomorphic to $\mathrm{Spec}\,W(k)[x_{ij} \; | \; i \in I_l, j \in I_m]$.
If $R$ is a $W(k)$-algebra, then
\[{\bf G}_{l,m}(R) = 1_{M \otimes_{W(k)} R} + \mathrm{Hom}(\widetilde{F}^m_{\mathcal{B}}(M),\widetilde{F}^l_{\mathcal{B}}(M)) \otimes_{W(k)} R,\]
and thus ${\bf G}_{l,m} \cong \mathbb{G}_a^{h_{f_l}h_{f_m}}$. If $l = m$, let ${\bf G}_{l,l}$ be $\mathbf{GL}_{\widetilde{F}^l_{\mathcal{B}}(M)}$. With respect to the $F$-basis $\mathcal{B}$, ${\bf G}_{l,l}$ is isomorphic to $ \mathrm{Spec}\,W(k)[x_{ij} \; | \; i, j \in I_l][\frac{1}{\mathrm{det}(x_{ij})}].$
Put
\[\mathbf{G}_+ = \prod_{1 \leq m < l \leq t} \mathbf{G}_{l,m} = \mathbf{G}_{t,t-1} \times \mathbf{G}_{t,t-2} \times \mathbf{G}_{t-1, t-2} \cdots \times \mathbf{G}_{3,2}   \times \mathbf{G}_{t,1} \times \cdots \times \mathbf{G}_{3,1} \times \mathbf{G}_{2,1},\]
\[\mathbf{G}_- = \prod_{1 \leq l < m \leq t} \mathbf{G}_{l,m} = \mathbf{G}_{1,2} \times \mathbf{G}_{1,3} \times \cdots \times \mathbf{G}_{1,t} \times  \mathbf{G}_{2,3} \times \cdots \mathbf{G}_{t-2,t-1} \times \mathbf{G}_{t-2,t} \times \mathbf{G}_{t-1,t},\]
\[{\bf G}_0 := \prod_{l=1}^t \mathbf{G}_{l,l}, \qquad \mathrm{and} \qquad \widetilde{\bf G} := \mathbf{G}_+ \times \mathbf{G}_0 \times \mathbf{G}_-.\]
With respect to the $F$-basis $\mathcal{B}$,
\[\widetilde{\bf G} = \mathrm{Spec}\,W(k)[x_{ij} \; | \; 1 \leq i,j \leq r][\frac{1}{\prod_{l=1}^t\mathrm{det}(x_{ij})_{i,j\in I_l}}].\]
Let $P_m : \widetilde{\bf G} \to \mathbf{GL}_M$ be the natural product morphism, and let $P_{{\widetilde{\bf G}}}$ be the composition
\[P_m \circ \big (1_{\bf G_+} \times 1_{\bf G_0}\times \prod_{1 \leq l < m \leq t} (\bullet )^{p^{f_m-f_l}} \big ): \widetilde{\bf G} \to \widetilde{\bf G} \to \mathrm{\bf GL}_M.\]

For any morphism $Q: \mathbf{H}_1 \to \mathbf{H}_2$ of affine schemes, let $Q' : R_{\mathbf{H}_2} \to R_{\mathbf{H}_1}$ be the natural homomorphism induced by $Q$.
\begin{lemma}
There is a unique morphism $P: \widetilde{\bf G} \to {\bf G}$ such that $P_{\bf G} \circ P = P_{\widetilde{\bf G}}$.
\end{lemma}

\begin{proof}
The morphism $P_{\bf G}: {\bf G} \to \mathbf{GL}_M$ at the level of $W(k)$-algebras
\[P_{\bf G}' : W(k)[x_{ij} \; | \; 1 \leq i, j \leq r][\frac{1}{\mathrm{det}(x_{ij})}] \longrightarrow W(k)[p^{\epsilon_{ij}}x_{ij} \; | \; 1 \leq i, j \leq r][\frac{1}{\mathrm{det}(x_{ij})}]\]
is such that $P_{\bf G}'(x_{ij}) = x_{ij}$; see the coordinate description of $\mathbf{G}$ for the definition of $\epsilon_{ij}$. Note that $\epsilon_{ij} \leq 0$. The morphism $P_{\widetilde{\bf G}} : \widetilde{\bf G} \to \mathbf{GL}_M$ at the level of $W(k)$-algebras
\[P_{\widetilde{\bf G}}': W(k)[x_{ij} | 1 \leq i,j \leq r][\frac{1}{\mathrm{det}(x_{ij})}] \longrightarrow W(k)[x_{ij} | 1 \leq i,j \leq r][\frac{1}{\prod_{l=1}^t\mathrm{det}(x_{ij})_{i,j \in I_l}}]\]
is such that $P_{\widetilde{\bf G}}'(x_{ij}) = P_ m'(p^{-\epsilon_{ij}}x_{ij})$. It is easy to check (at the level of $R$-valued points) that $P'_{\widetilde{\bf G}}(\mathrm{det}(x_{ij})) = \prod_{l=1}^t\mathrm{det}(x_{ij})_{i,j \in I_l}$. This forces $P:\widetilde{\bf G} \to {\bf G}$ to satisfy $P'(p^{\epsilon_{ij}}x_{ij}) = P'_m(x_{ij})$ and it indeed defines a $W(k)$-algebra homomorphism
\[\mbox{\small $\displaystyle P': W(k)[p^{\epsilon_{ij}}x_{ij} \, | \, 1 \leq i, j \leq r][\frac{1}{\mathrm{det}(x_{ij})}] \to W(k)[x_{ij} \, | \, 1 \leq i, j \leq r][\frac{1}{\prod_{l=1}^t \mathrm{det}(x_{ij})_{i,j \in I_l}}],$}\]
as 
\[P'(\mathrm{det}(x_{ij})) = \mathrm{det}(P'(x_{ij})) = \mathrm{det}(P'_m(p^{-\epsilon_{ij}}x_{ij})) = \mathrm{det}(P'_{\widetilde{\bf G}}(x_{ij})) = \prod_{l=1}^t \mathrm{det}(x_{ij})_{i,j \in I_l}. \qedhere\]
\end{proof}

\begin{lemma} \label{lemma:matrixmanipulation}
For every $k$-algebra $R$, the morphism $P$ induces a bijection on $W_s(R)$-valued points for all positive integer $s$.
\end{lemma}
\begin{proof}
We first show that $P$ induces a bijection on $W(R)$-valued points.

We start by showing that the image of $P_{\widetilde{\bf G}}(W(R))$ is the same as the image $P_{\bf G}(W(R))$ in $\mathbf{GL}_M(W(R))$, which is
\[S := \{(p^{-\epsilon_{ij}}r_{ij})_{1 \leq i,j \leq r} \; | \; r_{ij} \in W(R),\, \mathrm{det}(r_{ij}) \in W(R)^*\}.\]
As $t \times t$ block matrices, these are matrices of the type
\[N = \begin{pmatrix}
N_{11} & p^{f_2-f_1}N_{12} & p^{f_3-f_1}N_{13} & \cdots & p^{f_t-f_1}N_{1t} \\
N_{21} & N_{22} & p^{f_3-f_2}N_{23} & \cdots & p^{f_t-f_2}N_{2t} \\
\vdots & \vdots & \vdots & \ddots & \vdots \\
N_{t1} & N_{t2} & N_{t3} & \cdots & N_{tt}
\end{pmatrix}\]
where $N_{lm}$ is an arbitrary $h_{f_l} \times h_{f_m}$ matrix for $1 \leq l, m \leq t$ with entries in $W(R)$, and $\mathrm{det}(N) \in W(R)^*$. We claim that $N_{ii}$ are invertible for $1 \leq i \leq t$. After reduction modulo $\theta(W(R))$, the matrix $N$ is a lower triangular block matrix. The determinant of $N$ modulo $\theta(W(R))$ is $\prod_{i=1}^t \mathrm{det}(N_{ii})$ modulo $\theta(W(R))$, which is a unit in $R$, this implies that $\mathrm{det}(N_{ii})$ modulo $\theta(W(R))$ is a unit in $R$ and hence $\mathrm{det}(N_{ii})$ is a unit in $W(R)$.

Let $X$ be an arbitrary $t \times t$ block matrix in $\mathbf{G}_0(W(R))$ so that the diagonal blocks are denoted by $X_{ii}$. If $l>m$, let $Y_{lm}$ be an arbitrary $t \times t$ block matrix in $\mathbf{G}_{l,m}(W(R))$ with $\widetilde{Y}_{lm}$ at $(l,m)$ block entry and $0$ at everywhere else. If $l<m$, let $Z_{lm}$ be an arbitrary $t \times t$ block matrix in $\mathbf{G}_{l,m}(W(R))$ with $p^{f_m-f_l}\widetilde{Z}_{lm}$ at $(l,m)$ block entry and $0$ at everywhere else. We need to show that the set
\[\{\prod_{1 \leq m < l \leq t} Y_{lm} X \prod_{1 \leq l <m \leq t} Z_{lm} \; | \; X, Y_{lm}, Z_{lm} \; \textrm{satisfy the conditions stated above}\}\]
is equal to the set $S$ of all $t \times t$ matrices $N$ as described above. Here the order of the product $\prod_{1 \leq m < l \leq t} Y_{lm}$ is the same as the order in the definition of $\mathbf{G}_+$. The order of the product $\prod_{1 \leq l <m \leq t} Z_{lm}$ is the same as the order in the definition of $\mathbf{G}_-$.

We use induction on $t$. The base case when $t=1$ is trivial. Suppose it is true for $t-1$. Then 
\[\prod_{1 \leq m < l \leq t-1} Y_{lm} X \prod_{1 \leq l <m \leq t-1} Z_{lm}
= \begin{pmatrix}
X_{11} & p^{f_2-f_1}X_{12} & \cdots & p^{f_{t-1}-f_1}A_{1,t-1} & 0 \\
X_{21} & X_{22} & \cdots & p^{f_{t-1}-f_2}X_{2,t-1} & 0 \\
\vdots & \vdots & \ddots & \vdots & \vdots \\
X_{t-1,1} & X_{t-1,2} & \cdots & X_{t-1,t-1} & 0 \\
0 & 0  & \cdots & 0 & X_{tt} 
\end{pmatrix}\]
satisfies $\mathrm{det}(X_{ii}) \in W(R)^*$, and each $X_{lm}$ is an arbitrary $h_{f_l} \times h_{f_m}$ matrix if $l \neq m$. We abbreviate this matrix by $\begin{pmatrix} \widetilde{X} & 0 \\ 0 & X_{tt} \end{pmatrix}$, then
\[\mbox{\small $\displaystyle \prod_{1 \leq m \leq t-1} Y_{tm} \begin{pmatrix} \widetilde{X} & 0 \\ 0 & X_{tt} \end{pmatrix} \prod_{1 \leq l \leq t-1} Z_{lt}= \begin{pmatrix} 1 & 0 \\ \widetilde{Y} & 1 \end{pmatrix} \begin{pmatrix} \widetilde{X} & 0 \\ 0 & X_{tt} \end{pmatrix} \begin{pmatrix} 1 & \widetilde{Z} \\ 0 & 1 \end{pmatrix} = \begin{pmatrix} \widetilde{X} & \widetilde{X} \widetilde{Z} \\ \widetilde{Y} \widetilde{X} & X_{tt}+\widetilde{Y}\widetilde{X}\widetilde{Z} \end{pmatrix}.$}\]
Here the matrix $\widetilde{Y} = (\widetilde{Y}_{t1}, \widetilde{Y}_{t2}, \dots, \widetilde{Y}_{t,t-1})$ has size $h_{f_t} \times (r-h_{f_t})$ and the matrix $\widetilde{Z} = (p^{f_t-f_1}\widetilde{Z}_{1t},\, p^{f_t-f_2}\widetilde{Z}_{2t},\, \dots,\, p^{f_t-f_{t-1}}\widetilde{Z}_{t-1,t})^T$ has size $(r-h_{f_t}) \times h_{f_t}$. As $\widetilde{X}$ is invertible, the right multiplication of $\widetilde{X}$ induces a bijection from the set of all $h_{f_t} \times (r-h_{f_t})$ matrices to itself. Thus $\widetilde{Y}\widetilde{X}$ can be any $h_{f_t} \times (r-h_{f_t})$ matrix with $\widetilde{X}$ fixed and $\widetilde{Y}$ varied. Multiplying $\widetilde{X}$ and $\widetilde{Z}$, we get
\[\begin{pmatrix}
X_{11} & p^{f_2-f_1}X_{12} & \cdots & p^{f_{t-1}-f_1}X_{1,t-1} \\
X_{21} & X_{22} & \cdots & p^{f_{t-1}-f_2}X_{2,t-1}\\
\vdots & \vdots & \ddots & \vdots \\
X_{t-1,1} & X_{t-1,2} & \cdots & X_{t-1,t-1}
\end{pmatrix} \begin{pmatrix} p^{f_t-f_1}\widetilde{Z}_{1t} \\ p^{f_t-f_2}\widetilde{Z}_{2t} \\ \vdots \\  p^{f_t-f_{t-1}}\widetilde{Z}_{t-1,t} \end{pmatrix} = \]
\[\begin{pmatrix}p^{f_t-f_1}(X_{11}\widetilde{Z}_{1t} + \cdots + X_{1,t-1}\widetilde{Z}_{t-1,t}) \\ p^{f_t-f_2}(p^{f_2-f_1}X_{21}\widetilde{Z}_{1t} + \cdots + X_{2,t-1}\widetilde{Z}_{t-1,t}) \\ \vdots \\ p^{f_t-f_{t-1}}(p^{f_{t-1}-f_1}X_{t-1,1}\widetilde{Z}_{1t} + \cdots + X_{t-1,t-1}\widetilde{Z}_{t-1,t})\end{pmatrix}.\]
To show that $\widetilde{X}\widetilde{Z}$ can be any matrix of the type 
\[(p^{f_t-f_1}N_{1t}, p^{f_t-f_2}N_{2t}, \dots, p^{f_t-f_{t-1}}N_{t-1,t})^T\]
with $\widetilde{X}$ fixed and $\widetilde{Z}$ varied, it is enough to show that the matrix 
\[\begin{pmatrix}
X_{11} & X_{12} & \cdots & X_{1,t-1}\\
p^{f_2-f_1}X_{21} & X_{22} & \cdots & X_{2,t-1}\\
\vdots & \vdots & \ddots & \vdots \\
p^{f_{t-1}-f_1}X_{t-1,1} & p^{f_{t-1}-f_2}X_{t-1,2} & \cdots & X_{t-1,t-1} \\
\end{pmatrix}
\]
is invertible. But this is so because $X_{ii}$ are invertible. When $\widetilde{X}$, $\widetilde{Y}$ and $\widetilde{Z}$ are fixed, $X_{tt}+\widetilde{Y}\widetilde{X}\widetilde{Z}$ can be an arbitrary invertible $h_t \times h_t$ matrix with $X_{tt}$ varied because $\widetilde{Y}\widetilde{X}\widetilde{Z}$ modulo $p$ is zero. Thus we have shown that $P_{\widetilde{\bf G}}(W(R))$ and $P_{\bf G}(W(R))$ are the same in $\mathbf{GL}_M(W(R))$.

To show that $P(W(R))$ is a bijection, it is enough to show that $P_{\widetilde{\bf G}}(W(R))$ is an injection. If this is true, as $P_{\bf G}(W(R))$ is an injection and the image of $P_{\widetilde{\bf G}}(W(R))$  and $P_{\bf G}(W(R))$ are the same, then $P(W(R))$ is a bijection. Suppose  
\begin{equation} \label{equation:productmatrixequal}
\prod_{1 \leq m < l \leq t} Y_{lm} X \prod_{1 \leq l < m \leq t} Z_{lm} = \prod_{1 \leq m < l \leq t} Y'_{lm} X' \prod_{1 \leq l < m \leq t} Z'_{lm},
\end{equation}
and we want to show that $Y_{lm} = Y'_{lm}$ for all $1 \leq m < l \leq t$, $X=X'$ and $Z_{lm} = Z'_{lm}$ for all $1 \leq l < m \leq t$. By the definition of $Y_{lm}$ and $Z_{lm}$ it suffices to show that $\prod_{1 \leq m < l \leq t} Y_{lm} = \prod_{1 \leq m < l \leq t} Y'_{lm}$ and $\prod_{1 \leq l < m \leq t} Z_{lm} =  \prod_{1 \leq l < m \leq t} Z'_{lm}$. Equality \eqref{equation:productmatrixequal} is equivalent to
\begin{equation} \label{equation:productmatrixequalreduced}
(\prod_{1 \leq m < l \leq t} Y'_{lm})^{-1}\prod_{1 \leq m < l \leq t} Y_{lm}X = X \prod_{1 \leq l < m \leq t} Z'_{lm} ( \prod_{1 \leq l < m \leq t} Z_{lm})^{-1}.
\end{equation}
Let $(\prod_{1 \leq m < l \leq t} Y'_{lm})^{-1}\prod_{1 \leq m < l \leq t} Y_{lm} = I + Y$ where $Y$ is strictly lower triangular and $\prod_{1 \leq l < m \leq t} Z'_{lm} ( \prod_{1 \leq l < m \leq t} Z_{lm})^{-1} = I+Z$ where $Z$ is strictly upper triangular. The equality \eqref{equation:productmatrixequalreduced} is equivalent to $YX = XZ$. It is easy to see that $Y = Z = 0$ as $X$ is a diagonal block matrix with invertible blocks $X_{ii}$. This completes the proof that $P(W(R))$ is a bijection.

To show that $P(W_s(R))$ is injective, let $\bar{f}_1, \bar{f}_2 \in \widetilde{\mathbf{G}}(W_s(R))$ with lifts $f_1, f_2 \in \widetilde{\mathbf{G}}(W(R))$ respectively such that $P(W_s(R))(\bar{f}_1) = P(W_s(R))(\bar{f}_2)$. The images of $P(W(R))(f_1)$ and $P(W(R))(f_2)$ under the reduction epimorphism $\mathbf{G}(W(R)) \to \mathbf{G}(W_s(R))$ are the same. Hence $P(W(R))(f_1)$ and $P(W(R))(f_2)$ are congruent modulo $\theta^s$. As $P(W(R))$ is a bijection, $f_1$ and $f_2$ are also congruent modulo $\theta^s$ as well. Hence $\bar{f}_1 = \bar{f}_2$.

 To show that $P(W_s(R))$ is surjective, let $\bar{f} \in {\bf G}(W_s(R))$, a lift $f \in {\bf G}(W(R))$ of $\bar{f}$ has a preimage $g \in \widetilde{\bf G}(W(R))$ such that $P(W(R))(g) = f$ because $P(W(R))$ is surjective. Thus the image of $g$ in $\widetilde{\bf G}(W_s(R))$ is a preimage of $\bar{f}$. This shows that $P(W_s(R))$ is bijective and thus completes the proof the lemma.
\end{proof}

\begin{corollary} \label{corollary:natidg}
The morphism $P: \widetilde{\bf G} \to {\bf G}$ induces an isomorphism $P_{W_s(k)} : \widetilde{\bf G}_{W_s(k)} \to {\bf G}_{W_s(k)}$ for each $s \geq 1$.
\end{corollary}
\begin{proof}
If $s = 1$, then $P_{W_1(k)} = P_k$. It is an isomorphism by Lemma \ref{lemma:matrixmanipulation}. Suppose that $s>1$. As $R_{\bf G}$ and $R_{\widetilde{\bf G}}$ are $W_s(k)$-flat algebras, we get that $p^{s-1}R_{\bf G} / p^{s}R_{\bf G} \cong R_{\bf G}/pR_{\bf G}$ and $p^{s-1}R_{\widetilde{{\bf G}}} / p^{s}R_{\widetilde{{\bf G}}} \cong R_{\widetilde{\bf G}}/pR_{\widetilde{\bf G}}$ by the local criteria on flatness. As a result, $p^{s-1}R_{\bf G} / p^{s}R_{\bf G} \cong p^{s-1}R_{\widetilde{{\bf G}}} / p^{s}R_{\widetilde{{\bf G}}}$. We have the following commutative diagram:
\[\xymatrix{
0 \ar[r] & p^{s-1}R_{\bf G} / p^{s}R_{\bf G} \ar[r] \ar[d]^{\cong} & R_{\bf G}/p^{s}R_{\bf G} \ar[r] \ar[d]^{P'_{W_{s}(k)}} & R_{\bf G}/p^{s-1}R_{\bf G} \ar[r] \ar[d]^{P'_{W_{s-1}(k)}} & 0 \\
0 \ar[r] &  p^{s-1}R_{\widetilde{{\bf G}}} / p^{s}R_{\widetilde{{\bf G}}} \ar[r] & R_{\widetilde{\bf G}}/p^{s}R_{\widetilde{\bf G}} \ar[r] & R_{\widetilde{\bf G}}/p^{s-1}R_{\widetilde{\bf G}} \ar[r] & 0}.\]
An easy induction on $s$ using the five lemma concludes the proof of the lemma. 
\end{proof}

Let $\mathcal{B}$ be an $F$-basis of $\mathcal{M}$. The $\mathrm{Spec}\,W(k)$-scheme $\mathbf{GL}_M$ is represented by the $W(k)$-algebra $W(k)[x_{ij} \; | \; 1 \leq i, j \leq r][1/\mathrm{det}(x_{ij})]$. We construct the cocharacter $\mu : \mathbb{G}_m \to \mathbf{GL}_M$ (that depends on $\mathcal{B}$) defined by the $k$-algebra homomorphism 
\[\mu' : W(k)[x_{ij} \; | \; 1 \leq i, j \leq r][1/\mathrm{det}(x_{ij})] \to W(k)[x,1/x]\]
with the property that $\mu'(x_{ij}) = 0$ if $i \neq j$ and $\mu'(x_{ii}) = (1/x)^{e_{i}}$ for $1 \leq i \leq r$ where $e_1, e_2, \dots, e_r$ are the Hodge slopes of $(M, \varphi)$. Put $\sigma_{\mathcal{M}} := \varphi \mu(B(k))(p)$. It is a $\sigma$-linear isomorphism of $M$ defined by the rule $\sigma_{\mathcal{M}}(x) = p^{-f_j}\varphi(x)$ for $x \in \widetilde{F}^j_{\mathcal{B}}(M)$. It is well known \cite[A.1.2.6]{Fontaine:padicrep1} that there is a $\mathbb{Z}_p$-submodule $M_0 = \{ x \in M \; | \; \sigma_{\mathcal{M}}(x) = x\}$ of $M$, whose rank is the same as the rank of $M$ and such that $M = M_0 \otimes_{\mathbb{Z}_p} W(k)$. Note that the construction of $M_0$ also depends on $\mathcal{B}$. We fix a $\mathbb{Z}_p$-basis of $\mathcal{B}_0 = \{w_1, w_2, \dots, w_r\}$ of $M_0$. It induces a $\mathbb{Z}_p$-basis $\mathcal{B}_0^* = \{e_{ij} \; | \; 1 \leq i,j \leq r\}$ of  $\mathrm{End}_{\mathbb{Z}_p}(M_0)$ such that $e_{ij}(w_j) = w_i$. Note that 
\[\mathrm{End}_{W(R)}(M \otimes_{W(k)} W(R)) = \mathrm{End}_{\mathbb{Z}_p}(M_0) \otimes_{\mathbb{Z}_p} W(R).\]
Let $h = \sum_{i,j}a_{ij}e_{ij} \in \mathrm{End}_{W(R)} (M \otimes_{W(k)} W(R))$, $a_{ij} \in W(R)$ for all $1 \leq i, j \leq r$. Define 
\[\bar{\sigma}_{\mathcal{M}} : \mathrm{End}_{W(R)}(M \otimes_{W(k)} W(R)) \to \mathrm{End}_{W(R)}(M \otimes_{W(k)} W(R))\]
by the formula $\bar{\sigma}_{\mathcal{M}}(h) = \sum_{i,j} \sigma_R(a_{ij})e_{ij}$ and similarly, define
\[\bar{\sigma}_{\mathcal{M}} : \mathrm{End}_{W_s(R)}(M \otimes_{W(k)} W_s(R)) \to \mathrm{End}_{W_s(R)}(M \otimes_{W(k)} W_s(R))\]
by the formula $\bar{\sigma}_{\mathcal{M}}(h[s]) = (\sum_{i,j} \sigma_R(a_{ij})e_{ij})[s]$, where $h[s]$ is $h$ modulo $\theta^s$ (again this does not contradict to the previous convention that $h[s]$ is $h$ modulo $p^s$ when $R$ is perfect.) One can easily show that the definition of $\bar{\sigma}_{\mathcal{M}}$ does not depend on $\mathcal{B}_0$ and $\mathcal{B}^*_0$ but does depend on $\mathcal{B}$. If $R$ is a perfect field, then $\overline{\sigma}_{\mathcal{M}}$ satisfies $\bar{\sigma}_{\mathcal{M}}(h) = \sigma_{\mathcal{M}} h \sigma_{\mathcal{M}}^{-1}$, which is a formula that does not depend on the choice of $\mathcal{B}_0$ or $\mathcal{B}_0^*$ but does depend on $\mathcal{B}$ since $\sigma_{\mathcal{M}}$ does.

For every $ h \in \mathrm{End}_{W(R)}(M \otimes_{W(k)} W(R))$, define 
\[\varphi(h) = \bar{\sigma}_{\mathcal{M}} (\mu(B(R))(1/p) \circ h \circ \mu(B(R))(p)).\]
A priori, the definition of $\varphi(h)$ depends on the choice of the $F$-basis $\mathcal{B}$ of $\mathcal{M}$ as $\bar{\sigma}_{\mathcal{M}}$ and $\mu$ do. As $h = \sum_{i,j} a_{ij}e_{ij}$, $a_{ij} \in W(R)$, we get
\begin{align*}
\varphi(h) = & \; \bar{\sigma}_{\mathcal{M}} (\mu(B(R))(1/p) \circ h \circ \mu(B(R))(p)) \\
= &\;  \bar{\sigma}_{\mathcal{M}} (\sum_{i,j} (\mu(B(k))(1/p) \circ e_{ij} \circ \mu(B(k))(p)) \otimes a_{ij}) \\
= &\; \sum_{i,j} \bar{\sigma}_{\mathcal{M}} ((\mu(B(k))(1/p) \circ e_{ij} \circ \mu(B(k))(p)) \otimes \sigma_R(a_{ij}) \\
= &\; \sum_{i,j} \sigma_{\mathcal{M}} \mu(B(k))(1/p) \circ e_{ij} \circ \mu(B(k))(p) \sigma^{-1}_{\mathcal{M}} \otimes \sigma_R(a_{ij}) \\
= &\; \sum_{i,j} (\varphi \circ e_{ij} \circ \varphi^{-1}) \otimes \sigma_R(a_{ij}).
\end{align*}
Thus $\varphi(h)$ is a $B(R)$-linear endomorphism of $M \otimes_{W(k)} B(R)$ defined by the following rule: let $h = \sum_i h_i \otimes c_i$ under the natural identification (basis free)
\[\mathrm{End}_{W(R)}(M \otimes_{W(k)} W(R)) = \mathrm{End}_{W(k)}(M) \otimes_{W(k)} W(R).\]
For any $m \otimes b \in M \otimes_{W(k)} W(R)$, we have $\varphi(h)(m \otimes b) = \sum_i (\varphi \circ h_i \circ \varphi^{-1})(m) \otimes \sigma_R(c_i)b \in (M \otimes_{W(k)} B(k)) \otimes_{B(k)} B(R) = M \otimes_{W(k)} B(R)$. Thus the definition $\varphi(h)$ does not depend on the choice of $\mathcal{B}$. Note that $\varphi(h)$ might not be an element in $\mathrm{End}_{W(R)}(M \otimes_{W(k)} W(R))$ in general, but it is always an element in $\mathrm{End}_{W(R)}(M \otimes_{W(k)} B(R))$.

\begin{lemma}
For simplicity, set $\mu(B(R))(p) = \mu(p)$ and $\mu(B(R))(1/p) = \mu(1/p)$. For every $g \in {\bf G}_{l,m}(W(R))$, the following three formulae hold.
\begin{enumerate}
\item If $m<l$, then $\mu(p)g^{p^{f_l-f_m}} \mu(1/p) = g$.
\item If $m=l$, then $\mu(p)g \mu(1/p) = g$.
\item If $m>l$, then $\mu(p)g\mu(1/p) = g^{p^{f_m-f_l}}$.
\end{enumerate}
\end{lemma}
\begin{proof}
We first prove (1) when $m<l$. By definition, $g \in {\bf G}_{l,m}(W(R))$ if and only if $g = 1_{M \otimes W(R)}+e$ where $e \in \mathrm{Hom}(\widetilde{F}^m_{\mathcal{B}}(M),\widetilde{F}^l_{\mathcal{B}}(M)) \otimes_{W(k)} W(R)$. If $m<l$, then 
\begin{align*}
\mu(p)g^{p^{f_l-f_m}}\mu(1/p) &= \mu(p)(1_{M \otimes_{W(k)} W(R)}+p^{f_l-f_m}e)\mu(1/p)  \\ &= 1_{M \otimes_{W(k)} W(R)} + p^{f_l-f_m}\mu(p)e\mu(1/p).
\end{align*}
Thus it suffices to show that $p^{f_l-f_m}\mu(p)e\mu(1/p) = e$. 

As $e \in \mathrm{Hom}(\widetilde{F}^m_{\mathcal{B}}(M),\widetilde{F}^l_{\mathcal{B}}(M)) \otimes_{W(k)}W(R)$, $\mu(1/p)$ acts on $\widetilde{F}^m_{\mathcal{B}}(M) \otimes_{W(k)} W(R)$ as $p^{f_m}$, and $\mu(p)$ acts on $\widetilde{F}^l_{\mathcal{B}}(M) \otimes_{W(k)} W(R)$ as $p^{-f_l}$, we get the desired equality.

The cases when $m=l$ and $m>l$ are similar and are left to the reader.
\end{proof}

\begin{corollary} For every $g \in {\bf G}_{l,m}(W(R))$, the following three formulae hold.
\begin{enumerate}
\item If $m<l$, then $\overline{\sigma}_{\mathcal{M}}(g^{p^{f_l-f_m}}) = \varphi(g)$.
\item If $m=l$, then $\overline{\sigma}_{\mathcal{M}}(g) = \varphi(g)$. 
\item If $m>l$, then $\overline{\sigma}_{\mathcal{M}}(g) = \varphi(g^{p^{f_m-f_l}})$.
\end{enumerate}
\end{corollary}

\subsection{The group action $\mathbb{T}_s$} \label{subsection:groupaction}
Set ${\bf G}_s = \mathbb{W}_s({\bf G})$ and ${\bf D}_s = \mathbb{W}_s(\mathrm{\bf GL}_M)$. As ${\bf G}_{W_s(k)} = \widetilde{\bf G}_{W_s(k)}$, we have  $\widetilde{\bf G}_s :=\mathbb{W}_s(\widetilde{\bf G}) = {\bf G}_s$. The group action 
\[\mathbb{T}_s : {\bf G}_s \times_k {\bf D}_s \to {\bf D}_s\]
is defined on $R$-valued points as follows: For every $h[s] \in {\bf G}_s(R)$, $g[s] \in {\bf D}_s(R)$, let $h \in \mathbf{G}(W(R))$ be a lift of $h[s]$ under the reduction epimorphism $\mathbf{G}(W(R)) \to \mathbf{G}(W_s(R))$ and $g \in \mathbf{GL}_M(W(R))$ be a lift of $g[s]$ under the reduction epimorphism $\mathbf{GL}_M(W(R)) \to \mathbf{GL}_M(W_s(R))$. Define
\[\mathbb{T}_s(R)(h[s], g[s]) := (hg\varphi(h^{-1}))[s].\]
It is clear that the definition does not depend on the choices of lifts of $h[s]$ and $g[s]$ and does not depend on choice of basis.

To see that $(hg\varphi(h^{-1}))[s] \in {\bf D}_s(R)$, let us first recall the identification ${\bf G}_{W_s(k)} = \widetilde{\bf G}_{W_s(k)}$ from Corollary \ref{corollary:natidg}, thus $h[s] \in \mathbf{G}_s(R) = \mathbf{G}(W_s(R)) = \mathbf{G}_{W_s(k)}(W_s(R)$ is an element of $\widetilde{\mathbf{G}}_{W_s(k)}(W_s(R))$. We can g (non-uniquely) $h[s]$ as a product
\[\prod_{1 \leq m < l \leq t} h_{lm}[s]  h_0[s] \prod_{1 \leq l < m \leq t} h^{p^{f_m-f_l}}_{lm}[s],\]
where $\prod_{1 \leq m < l \leq t} h_{lm}[s] \in ({\bf G}_-)_{W_s(k)}(W_s(R))$, $h_0[s] \in (\mathbf{G}_0)_{W_s(k)}(W_s(R))$, and $\prod_{1 \leq l < m \leq t} h_{lm}[s] \in ({\bf G}_+)_{W_s(k)}(W_s(R))$. Therefore $(hg\varphi(h^{-1}))[s]$ is equal to
\begin{align*}
 & (\prod_{1 \leq m < l \leq t} h_{lm}  h_0 \prod_{1 \leq l < m \leq t} h^{p^{f_m-f_l}}_{lm} g\varphi( \prod_{1 \leq l < m \leq t} h^{p^{f_l-f_m}}_{lm} h^{-1}_{0} \prod_{1 \leq m < l \leq t} h^{-1}_{lm} ))[s]\\
= &(\prod_{1 \leq m < l \leq t} h_{lm} h_{0} \prod_{1 \leq l < m \leq t} h^{p^{f_m-f_l}}_{lm} g \prod_{1 \leq l < m \leq t} \varphi(h^{p^{f_l-f_m}}_{lm})\varphi(h^{-1}_{0}) \prod_{1 \leq m < l \leq t} \varphi(h^{-1}_{lm}))[s]\\
=&(\prod_{1 \leq m < l \leq t} h_{lm} h_{0} \prod_{1 \leq l < m \leq t} h^{p^{f_m-f_l}}_{lm} g \prod_{1 \leq l < m \leq t} \bar{\sigma}_{\mathcal{M}}(h^{-1}_{lm})\bar{\sigma}_{\mathcal{M}}(h^{-1}_{0}) \prod_{1 \leq m < l \leq t} \bar{\sigma}_{\mathcal{M}} ((h^{-1}_{lm})^{p^{f_l-f_m}}))[s]\\
=&\mbox{\small $\displaystyle \prod_{1 \leq m < l \leq t} h_{lm}[s] h_{0}[s] \prod_{1 \leq l < m \leq t} h^{p^{f_m-f_l}}_{lm}[s] g[s] \prod_{1 \leq l < m \leq t} \bar{\sigma}_{\mathcal{M}}(h^{-1}_{lm}[s])\bar{\sigma}_{\mathcal{M}}(h^{-1}_{0}[s]) \prod_{1 \leq m < l \leq t} \bar{\sigma}_{\mathcal{M}} ((h^{-1}_{lm}[s])^{p^{f_l-f_m}})$}
\end{align*}
which is in $\mathbf{D}_s(R)$. The above formula proves that $\mathbb{T}_s$ is a morphism.

For later use, we record the following formula when $R=k$ and $s=1$.
\begin{equation} \label{equation:actionlevel1}
\mathbb{T}_1(k)(h[1],g[1]) = \prod_{1 \leq m < l \leq t} h_{lm}[1] h_{0}[1] g[1] \prod_{1 \leq l < m \leq t} (\bar{\sigma}_{\mathcal{M}}(h^{-1}_{lm}[1]))(\bar{\sigma}_{\mathcal{M}}(h^{-1}_{0}[1])) 
\end{equation}

\subsection{Orbits and Stabilizers of $\mathbb{T}_s$} Let $1_M[s] \in \mathbf{D}_s(k)$. The image of the morphism
\[\Psi := \mathbb{T}_s \circ (1_{\mathbf{G}_s} \times_k 1_M[s]) : \mathbf{G}_s \cong \mathbf{G}_s \times_k \mathrm{Spec}\,k \to \mathbf{G}_s \times_k \mathbf{D}_s \to \mathbf{D}_s .\]
is the orbit of $1_M[s]$, which we denoted by $\mathbf{O}_s$. Its Zariski closure $\overline{\mathbf{O}}_s$ is a closed integral subscheme of $\mathbf{D}_s$. The orbit $\mathbf{O}_s$ is a smooth connected  open subscheme of $\overline{\mathbf{O}}_s$. 

\begin{proposition} \label{proposition:descriptionoforbits}
Let $g_1, g_2 \in \mathrm{\bf GL}_M(W(k))$. Then $g_1[s], g_2[s] \in \mathrm{\bf GL}_M(W_s(k)) = {\bf D}_s(k)$ belong to the same orbit of the action $\mathbb{T}_s$ if and only if $F_s(\mathcal{M}(g_1))$ is isomorphic to $F_s(\mathcal{M}(g_2))$.
\end{proposition}
\begin{proof}
We know that $g_1[s]$ and $g_2[s]$ belong to the same orbit of the action $\mathbb{T}_s$ if and only if there exists $h[s] \in {\bf G}_s(k)$ such that $\mathbb{T}_s(h[s],g_1[s]) = (hg_1\varphi h^{-1}\varphi^{-1})[s] = g_2[s]$. This implies that $h[s]$ is an isomorphism from $F_s(\mathcal{M}( g_1))$ to $F_s(\mathcal{M}(g_2))$.

If $h[s]$ is an isomorphism from $F_s(\mathcal{M}(g_1))$ to $F_s(\mathcal{M}(g_2))$, then $hg_1\varphi h^{-1} \varphi^{-1} \equiv g_2$ modulo $p^s$. To conclude the proof, it is enough to show that $h \in \mathbf{G}_s(k)$, but this is clear from the facts that $h(M) = M$ and $h(\varphi^{-1}(M)) \subset \varphi^{-1}(M)$.
\end{proof}

\begin{corollary}
The set of orbits of the action $\mathbb{T}_s$ is in natural bijection to the set of isomorphism classes of $F$-truncations modulo $p^s$ of $\mathcal{M}(g)$ for all $g \in \mathbf{GL}_M(W(k))$.
\end{corollary}

Let $\mathbf{S}_s$ be the fibre product defined by the following commutative diagram:
\[\xymatrix{
\mathbf{S}_s \ar[r] \ar[d] & \mathrm{Spec}\,k \ar[d]^{1_M[s]} \\
\mathbf{G}_s  \ar[r]^{\Psi} & \mathbf{D}_s
}\]
It is the stabilizer of $1_M[s]$ and is a subgroup scheme of $\mathbf{G}_s$. We denote by $\mathbf{C}_s$ the reduced scheme $(\mathbf{S}_s)_{\mathrm{red}}$, and $\mathbf{C}^0_s$ the identity component of $\mathbf{C}_s$. Clearly, 
\begin{equation} \label{equation:gammases}
\mathrm{dim}(\mathbf{S}_s) = \mathrm{dim}(\mathbf{C}_s) = \mathrm{dim}(\mathbf{C}^0_s) = \mathrm{dim}(\mathbf{G}_s) - \mathrm{dim}(\mathbf{O}_s).
\end{equation}

\begin{example} \label{example:stabilizerunipotent}
In this example, we follow the ideas of \cite[Section 2.3]{Vasiu:levelm} to discuss $\mathbb{T}_1(k)$. As a result, we will see that $\mathbf{C}_1^0$ is a unipotent group scheme over $k$. Let $(M, \varphi)$ be an $F$-crystal over $k$ such that $e_1 = 0$. By \cite[Section 1.8]{Vasiu:modp} or \cite[Theorem 1.1]{Viehmann:truncation1}, there exist an element $g \in \mathbf{GL}_M(W(k))$ with the property that $g \equiv 1_M$ modulo $p$, an $F$-basis $\mathcal{B} = \{v_1, v_2, \dots, v_r\}$ of $\mathcal{M}$, and a permutation $\pi$ on the set $I = \{1, 2, \dots, r\}$ that defines a $\sigma$-linear monomorphism $\varphi_{\pi} : M \to M$ with the property that $\varphi_{\pi}(v_i)  = p^{e_i} v_{\pi(i)}$ for all $i \in I$, such that $\mathcal{M}$ is isomorphic to $(M, g\varphi_{\pi})$. Let $\mu$ be the cocharacter defined with respect to $\mathcal{B}$ and let $\bar{\sigma}_{\mathcal{M}}$ be the $\sigma$-linear endomorphism of $\mathrm{End}_{W(R)}(M \otimes_{W(k)} W(R))$ defined with respect to $\mu$. Set
\begin{align*}
I_+ &= \{(i,j) \in I \times I \; | \; i \in I_l,\, j \in I_m, \; \textrm{where $m > l$} \}; \\
I_0 &= \{(i,j) \in I \times I  \; | \; i, j \in I_l \; \textrm{for some $l$}\}; \\
I_- &= \{(i,j) \in I \times I \; | \; i \in I_l,\, j \in I_m, \; \textrm{where $m < l$}\}.
\end{align*}
See Subsection \ref{subsection:filtrationfcry} for the definition of $I_l$ and $I_m$. For $1 \leq i, j \leq r$, let $f_{i,j} \in \mathrm{End}(M)$ be such that $f_{i,j} (v_j) = v_i$ and $f_{i,j}(v_l) = 0$ for $l \neq j$. For every $1+ \bar{f}_{i,j} \in \mathbf{GL}_M(k)$, where $\bar{f}_{i,j}$ is $f_{i,j}$ modulo $p$,  $\bar{\sigma}_{\mathcal{M}}(1+\bar{f}_{i,j}) = \varphi_{\pi}(1+\bar{f}_{i,j})\varphi_{\pi^{-1}}  = 1+\bar{f}_{\pi(i),\pi(j)}$.  For every $h = h_+h_0h_- \in {\bf G}(k) = \widetilde{\bf G}^*(k)$, where $h_{\dagger} \in \mathbf{G}_{\dagger} $ for $\dagger \in \{+, 0, -\}$, we know that $h[1] \in \mathbf{S}_1(k) = \mathbf{C}_1(k)$ if and only if
\begin{equation*}
h_+[1]h_0[1] = \bar{\sigma}_{\mathcal{M}}(h_0[1])\bar{\sigma}_{\mathcal{M}}(h_-[1])
\end{equation*}
by \eqref{equation:actionlevel1}. This is equivalent to
\begin{equation} \label{equation:examplet1stabilizer}
(h_+h_0)[1] = \bar{\sigma}_{\mathcal{M}}((h_0h_-)[1])
\end{equation}
Let
\[(h_+h_0)[1] = 1_M[1] + \sum_{(i,j) \in I_+ \cup I_0} x_{i,j}\bar{e}_{i,j}\]
and 
\[(h_0h_-)[1] = 1_M[1] + \sum_{(i,j) \in I_0 \cup I_-} x_{i,j}\bar{e}_{i,j}.\]
Then \eqref{equation:examplet1stabilizer} can be rewritten as 
\begin{equation} \label{equation:example1stabilizerneed}
\sum_{(i,j) \in I_+ \cup I_0} x_{i,j}\bar{e}_{i,j} = \sum_{(i,j) \in I_0 \cup I_-} x^p_{i,j} \bar{e}_{\pi(i),\pi(j)}
\end{equation}
This is equivalent to three types of equalities:
\begin{gather} 
x_{\pi(i),\pi(j)} = x^p_{i,j} \; \qquad \textrm{if} \; (i,j) \in I_- \cup I_0 \; \textrm{and} \; (\pi(i),\pi(j)) \in I_+ \cup I_0,  \label{equation:example1important1} \\
x_{\pi(i),\pi(j)} = 0 \; \qquad \textrm{if} \; (i,j) \in I_+ \; \textrm{and} \; (\pi(i),\pi(j)) \in I_+ \cup I_0, \label{equation:example1important2} \\
x^p_{i,j} = 0 \; \qquad \textrm{if} \; (i,j) \in I_- \cup I_0 \; \textrm{and} \; (\pi(i),\pi(j)) \in I_-. \label{equation:example1important3}
\end{gather}
We decompose the permutation $\pi \times \pi$ on $I \times I$ into a product of disjoint cycles $\prod_u (\pi \times \pi)_u$. To ease language, we say that a pair $(i,j) \in I \times I$ is \emph{in} $(\pi \times \pi)_u$ if $(\pi \times \pi)_u(i,j) \neq (i,j)$. To study the system of equations defined by \eqref{equation:example1important1} to \eqref{equation:example1important3} we consider the following three cases:
\begin{enumerate}
\item Consider $(\pi \times \pi)_u$ such that all $(i,j)$ in $(\pi \times \pi)_u$ are in $I_0$. By \eqref{equation:example1important1}, $x_{i,j} = x^{p^{\mathrm{ord}((\pi \times \pi)_u)}}_{i,j}$. Thus there are finitely many solutions for $x_{i,j}$.
\item Consider $(\pi \times \pi)_u$ such that all $(i,j)$ in $(\pi \times \pi)_u$ are in $I_0 \cup I_+$ and at least one $(i,j)$ is in $I_+$. By \eqref{equation:example1important2}, $x_{i,j} = 0$ for all $(i,j)$ in $(\pi \times \pi)_u$.
\item Consider $(\pi \times \pi)_u$ such that at least one $(i,j)$ in $(\pi \times \pi)_u$ is in $I_-$. Let $\nu_{\pi}(i,j)$ be the smallest positive integer such that 
\[(\pi^{\nu_{\pi}(i,j)}(i), \pi^{\nu_{\pi}(i,j)}(j)) \in I_+ \cup I_-.\] 
By \eqref{equation:example1important1}, $x_{\pi^m(i),\pi^m(j)} = x^{p^m}_{i,j}$ for all $1 \leq m < \nu_{\pi}(i,j)$. 
\begin{itemize}
\item If $(\pi^{\nu_{\pi}(i,j)}(i), \pi^{\nu_{\pi}(i,j)}(j)) \in I_-$, then $x_{\pi^m(i),\pi^m(j)} = 0$ for all $0 \leq m \leq \nu_{\pi}(i,j)$.
\item If $(\pi^{\nu_{\pi}(i,j)}(i), \pi^{\nu_{\pi}(i,j)}(j)) \in I_+$, then $x_{\pi^m(i),\pi^m(j)} = x^{p^m}_{i,j}$ for all $1 \leq m \leq \nu_{\pi}(i,j)$. 
\end{itemize}
Thus $x_{\pi^m(i),\pi^m(j)}$ for all $1 \leq m < \nu_{\pi}(i,j)$ has finitely many solutions.
\begin{itemize}
\item If $(\pi^{\nu_{\pi}(i,j)+1}(i),\pi^{\nu_{\pi}(i,j)+1}(j)) \in I_+ \cup I_0$, then $x_{\pi^{\nu_{\pi}(i,j)+1}(i),\pi^{\nu_{\pi}(i,j)+1}(j)}$ equals $0$ by \eqref{equation:example1important2}. 
\item If $(\pi^{\nu_{\pi}(i,j)+1}(i),\pi^{\nu_{\pi}(i,j)+1}(j)) \in I_-$, then $x_{\pi^{\nu_{\pi}(i,j)+1}(i),\pi^{\nu_{\pi}(i,j)+1}(j)}$ is not related to $x_{i,j}$.
\end{itemize}
\end{enumerate}
 Let $I_-^{\pi}$ be a subset of $I_-$ that contains pairs $(i,j)$ such that $(\pi^{\nu_{\pi}(i,j)}(i), \pi^{\nu_{\pi}(i,j)}(j)) \in I_+$. We conclude that $h[1] \in \mathbf{C}_1^0(k)$ if and only if the following equations hold:
\begin{align*}
h_+[1]h_0[1] &= 1_M[1] + \sum_{(i,j) \in I_-^{\pi}} \sum_{l=1}^{\nu_{\pi}(i,j)} x_{i,j}^{p^l}\bar{e}_{\pi^l(i),\pi^l(j)},\\
h_0[1] &=  1_M[1] + \sum_{(i,j) \in I_-^{\pi}} \sum_{l=1}^{\nu_{\pi}(i,j)-1} x_{i,j}^{p^l}\bar{e}_{\pi^l(i),\pi^l(j)},\\
h_0[1]h_-[1]  &= 1_M[1] + \sum_{(i,j) \in I_-^{\pi}} \sum_{l=0}^{\nu_{\pi}(i,j)-1} x_{i,j}^{p^l}\bar{e}_{\pi^l(i),\pi^l(j)},
\end{align*}
where $x_{i,j} \in I_-^{\pi}$ can take independently all values in $k$ such that $h_0[1] \in \mathbf{G}_1(k)$. This shows that $\mathrm{Lie}(\mathbf{C}^0_1) = \bigoplus_{(i,j) \in I_-^{\pi}} k \bar{e}_{i,j}$, which contains no non-zero semi-simple elements. Thus $\mathbf{C}_1^0$ has no subgroup isomorphic to $\mathbb{G}_m$ and hence it is unipotent. We also get that the dimension of $\mathbf{C}_1^0$ is equal to the cardinality of $I_-^{\pi}$. Therefore the dimension of $\mathbf{O}_1$ is equal to the cardinality of the set $I^2 - I_-^{\pi}$.
\end{example}

\begin{proposition}
For every $s \geq 1$, the connected smooth group scheme $\mathbf{C}^0_s$ is unipotent.
\end{proposition}
\begin{proof}
We proceed by induction. The base case $s=1$ is checked in Example \ref{example:stabilizerunipotent}. Suppose $\mathbf{C}^0_{s-1}$ is unipotent. The image of $\mathbf{C}^0_{s}$ under the reduction map $\mathrm{Red}_{s,\mathbf{G}} : \mathbf{G}_s \to \mathbf{G}_{s-1}$ is in $\mathbf{C}^0_{s-1}$, and thus is unipotent. The kernel of $\mathbf{C}^0_s \to \mathbf{C}^0_{s-1}$ is in the kernel of $\mathrm{Red}_{s,\mathbf{G}}$, and thus is unipotent. Therefore $\mathbf{C}^0_s$ is an extension of unipotent group schemes, and thus is unipotent; see \cite[Exp. XVII, Prop 2.2]{SGA3II}.
\end{proof}

We construct a homomorphism $\Lambda_s : \mathbf{C}_s \to \mathbf{A}_s$ as follows. For every $k$-algebra $R$, let $h[s] \in \mathbf{C}_s(R)$. Thus $\varphi(h[s]) = h[s]$. Fix a $\mathbb{Z}_p$-basis $\mathcal{B}_0 = \{v_1, \dots, v_r\}$ of $\mathcal{M}_0$. Let $B_0^* = \{e_{ij}\}$ be the standard $\mathbb{Z}_p$-basis of $\mathrm{End}_{\mathbb{Z}_p}(M_0)$ induced by $B_0$. If $h = \sum_{i,j}e_{ij} \otimes a_{ij} \in \mathrm{End}_{\mathbb{Z}_p}(M_0) \otimes W(R) = \mathrm{End}_{W(R)}(M \otimes_{W(k)} W(R))$, where $a_{ij} \in W(R)$, then $\varphi(h[s]) = h[s]$ is equivalent to 
\begin{equation} \label{equation:hinstabilizer}
\sum_{i,j} \varphi e_{ij} \varphi^{-1} \otimes \sigma_R(a_{ij}) \equiv \sum_{i,j} e_{ij} \otimes a_{ij}  \quad \textrm{modulo}  \; \; \theta^s(W(R)). 
\end{equation}
Let $C = (c_{ij})$ be the matrix representation of $\varphi$ with respect to $\mathcal{B}_0$ and $C^{-1} = (c'_{ij})$ be its inverse. Using the matrix notation, \eqref{equation:hinstabilizer} can be restated as
\begin{equation} \label{equation:hinstablizer2}
(c_{ij})(\sigma_R(a_{ij}))(\sigma_R(c'_{ij})) \equiv (a_{ij}) \quad \textrm{modulo}  \; \;  \theta^s(W(R)). 
\end{equation}
This implies that a lift $h$ of $h[s]$ satisfies the equation that defines $\mathbf{A}_s$. Thus we can define $\Lambda_s(R)(h[s]) = h[s]$.

\begin{lemma} \label{lemma:connectiondimension}
The homomorphism $\Lambda_s(k) : \mathbf{C}_s(k) \to \mathbf{A}_s(k)$ is an isomorphism. Therefore, $\Lambda_s$ is a finite epimorphism and thus $\mathrm{dim}(\mathbf{C}_s)  = \gamma_{\mathcal{M}}(s)$.
\end{lemma}
\begin{proof}
The group $\mathbf{C}_s(k)$ consists of all $h \in \mathbf{G}_s(k)$ such that $h \equiv \varphi h \varphi^{-1}$ modulo $p^s$, which are exactly all automorphisms of $F_s(\mathcal{M})$. As $\mathbf{A}_s(k)$ is also the group of automorphisms of $F_s(\mathcal{M})$ and $\Lambda_s(k)$ is the identity map, we know that they are isomorphic.

As $\Lambda_s(k)$ is an isomorphism, $\Lambda_s$ is a finite epimorphism. Therefore $\mathrm{dim}(\mathbf{C}_s) =\mathrm{dim}(\mathbf{C}^0_s) = \mathrm{dim}(\mathbf{A}^0_s) = \mathrm{dim}(\mathbf{A}_s)$, which by definition is $\gamma_{\mathcal{M}}(s)$.
\end{proof}

Let $\mathbf{T}_{s+1}$ be the reduced group of the group subscheme $\mathrm{Red}^{-1}_{s+1, \mathbf{G}}(\mathbf{C}_s)$ of $\mathbf{G}_{s+1}$, and let $\mathbf{T}_{s+1}^0$ be its identity component. We have a short exact sequence 
\begin{equation} \label{equation:sestc}
1 \to \mathrm{Ker}(\mathrm{Red}_{s+1,\mathbf{G}}) \to \mathbf{T}_{s+1}^0 \to \mathbf{C}_s^0 \to 1.
\end{equation} 
Thus $\mathbf{T}^0_{s+1}$ is unipotent as $\mathrm{Ker}(\mathrm{Red}_{s+1,\mathbf{G}})$ and $ \mathbf{C}_s^0$ are. We have the following equality
\begin{equation} \label{equation:dimtrc}
\mathrm{dim}(\mathbf{T}_{s+1}^0) = \mathrm{dim}(\mathrm{Ker}(\mathrm{Red}_{s+1,\mathbf{G}})) + \mathrm{dim}(\mathbf{C}_s^0) = r^2 + \mathrm{dim}(\mathbf{C}_s^0).
\end{equation}
By Lemma \ref{lemma:connectiondimension} and \eqref{equation:dimtrc}, we know that 
\begin{equation}
\mathrm{dim}(\mathbf{T}_{s+1}^0) = r^2 + \gamma_{\mathcal{M}}(s).
\end{equation}
By \eqref{equation:gammases} and the fact that $\mathrm{Red}_{s+1,\mathbf{G}}$ is an epimorphism whose kernel has dimension $r^2$, we know that 
\begin{equation} \label{equation:gammadiff}
\gamma_{\mathcal{M}}(s+1) - \gamma_{\mathcal{M}}(s)  = r^2 - \mathrm{dim}(\mathbf{O}_{s+1}) + \mathrm{dim}(\mathbf{O}_s).
\end{equation}

Let $\mathbf{V}_{s+1}$ be the inverse image of the point $1_M[s] \in \mathbf{D}_s(k)$. It is isomorphic to the kernel of $\mathrm{Red}_{s+1, \mathbf{D}}$ and thus isomorphic to $\mathbb{A}_{k}^{r^2}$. The inverse image of $\mathbf{O}_s$ under $\mathrm{Red}^{-1}_{s+1, \mathbf{D}}$ in $\mathbf{D}_{s+1}$ is a union of orbits and $\mathbf{O}_{s+1}$ is one of them. Let $\mathscr{O}_{s+1, s}$ be the set of orbits of the action $\mathbb{T}_{s+1}$ that is contained in $\mathrm{Red}^{-1}_{s+1,\mathbf{D}}(\mathbf{O_{s}})$. Every orbit in $\mathscr{O}_{s+1,s}$ intersects $\mathbf{V}_{s+1}$ nontrivially. 

We now give another description of $\mathscr{O}_{s+1, s}$ in terms of $F$-truncations modulo $p^s$ of $F$-crystals. Let $\mathscr{I}_{s}$ be the set of all $F$-crystals $\mathcal{M}(g)$ with $g \in \mathbf{GL}_M(W(k))$ up to $F$-truncations modulo $p^{s}$ isomorphisms. In other words, if $F_{s}(\mathcal{M}(g_1))$ is isomorphic to $F_{s}(\mathcal{M}(g_2))$, then we identify $\mathcal{M}(g_1)$ and $\mathcal{M}(g_2)$ in $\mathscr{I}_s$. By Proposition \ref{proposition:descriptionoforbits}, we know that there is a bijection between the set of orbits of $\mathbb{T}_s$ and $\mathscr{I}_{s}$.

\begin{proposition} \label{proposition:bijectionbetweenlifting}
There is a bijection between $\mathscr{O}_{s+1,s}$ and the subset of $\mathscr{I}_{s+1}$ that contains all $\mathcal{M}(g)$ (up to $F$-truncation modulo $p^{s+1}$ isomorphism) such that $F_s(\mathcal{M}(g))$ is isomorphic to $F_s(\mathcal{M})$. Therefore, $\mathscr{O}_{s+1,s}$ has only one orbit if $s \geq n_{\mathcal{M}}$.
\end{proposition}
\begin{proof}
The first part of the proposition follows from the following fact: for every $g \in \mathbf{GL}_M(W(k)$, an orbit of $\mathbb{T}_{s+1}$ that contains the $F$-truncation modulo $p^{s+1}$ of the $F$-crystal $\mathcal{M}(g)$ is in $\mathrm{Red}^{-1}_{s+1,\mathbf{D}}(\mathbf{O_{s}})$ if and only if $F_s(\mathcal{M}(g))$ is isomorphic to $F_s(\mathcal{M})$.

If $s \geq n_{\mathcal{M}}$, let $\mathcal{M}(g)$ be an $F$-crystal such that $F_s(\mathcal{M}(g))$ is isomorphic to $F_s(\mathcal{M})$. By Corollary \ref{corollary:correctgeneralization}, $\mathcal{M}(g)$ is isomorphic to $\mathcal{M}$. Thus $\mathscr{O}_{s+1,s}$ contains only one element by the first part of the proposition.
\end{proof}

\subsection{Monotonicity of $\gamma_{\mathcal{M}}(i)$} \label{subsection:monotonicityofgamma}

\begin{lemma} \label{lemma:keylemmastrictseq}
The following two statements are equivalent:
\begin{enumerate}[(i)]
\item $\mathrm{dim}(\mathbf{O}_{s+1}) = \mathrm{dim}(\mathbf{O}_s) + r^2$;
\item $\mathbf{V}_{s+1} \subset \mathbf{O}_{s+1}$.
\end{enumerate}
\end{lemma}
\begin{proof}
As $\mathrm{Red}_{s+1, \mathbf{D}}: \mathbf{O}_{s+1} \to \mathbf{O}_s$ is faithfully flat, the fibers of this morphism are equidimensional. Hence we have
\begin{equation} \label{equation:dimensionorbits}
\mathrm{dim}(\mathbf{O}_{s+1}) = \mathrm{dim}(\mathbf{O}_s) + \mathrm{dim}(\mathbf{V}_{s+1} \cap \mathbf{O}_{s+1}).
\end{equation}

If (i) holds, as $\mathrm{dim}(\mathbf{V}_{s+1} \cap \mathbf{O}_{s+1}) = \mathrm{dim}(\mathbf{O}_{s+1}) - \mathrm{dim}(\mathbf{O}_s) = r^2 = \mathrm{dim}(\mathbf{V}_{s+1})$, $\mathbf{V}_{s+1} \cap \mathbf{O}_{s+1}$ is open in $\mathbf{V}_{s+1}$. 

Consider the action $\mathbb{T}^0_{s+1} : \mathbf{T}^0_{s+1} \times_k \mathbf{V}_{s+1} \to \mathbf{V}_{s+1}$. By \cite[Proposition 2.4.14]{Springer:linearalgebraicgroups}, we know that all the orbits of $\mathbb{T}^0_{s+1}$ are closed. As the orbits of the action $\mathbb{T}_{s+1} : \mathbf{T}_{s+1} \times_k \mathbf{V}_{s+1} \to \mathbf{V}_{s+1}$ is a finite union of the orbits of the action $\mathbb{T}^0_{s+1}$, we know that  the orbits of the action $\mathbb{T}_{s+1}$ is also closed. The orbit of $1_M[s+1]$ under the action of $\mathbb{T}_{s+1}$ is $(\mathbf{V}_{s+1} \cap \mathbf{O}_{s+1})_{\mathrm{red}}$. Because it is  an open, closed and dense orbit of $\mathbb{T}_{s+1}$, we know that $\mathbf{V}_{s+1} \cap \mathbf{O}_{s+1} = \mathbf{V}_{s+1}$. Hence $\mathbf{V}_{s+1} \subset \mathbf{O}_{s+1}$.

If (ii) holds, as $\mathrm{dim}(\mathbf{V}_{s+1} \cap \mathbf{O}_{s+1}) = \mathrm{dim}(\mathbf{V}_{s+1}) = r^2$, (i) follows from \eqref{equation:dimensionorbits}.
\end{proof}

\begin{corollary} \label{corollary:keycorollarystrictseq}
$\gamma_{\mathcal{M}}(s+1) = \gamma_{\mathcal{M}}(s)$ if and only if $\mathscr{O}_{s+1,s}$ has only one element.
\end{corollary}
\begin{proof}
The first part of the Lemma \ref{lemma:keylemmastrictseq} is equivalent to $\gamma_{\mathcal{M}}(s+1) = \gamma_{\mathcal{M}}(s)$ and the second part of the Lemma \ref{lemma:keylemmastrictseq} is equivalent to $\mathscr{O}_{s+1,s}$ has only one element.
\end{proof}

\begin{theorem} \label{theorem:monotonicitygamma}
For every $F$-crystal $\mathcal{M}$, we have
\[0 \leq \gamma_{\mathcal{M}}(1) < \gamma_{\mathcal{M}}(2) < \cdots < \gamma_{\mathcal{M}}(n_{\mathcal{M}}) = \gamma_{\mathcal{M}}(n_{\mathcal{M}}+1) = \gamma_{\mathcal{M}}(n_{\mathcal{M}}+2) = \cdots\]
\end{theorem}
\begin{proof}
We first show that for every $1 \leq i \leq n_{\mathcal{M}}-1$, $\gamma_{\mathcal{M}}(i) \neq \gamma_{\mathcal{M}}(i+1)$. Suppose the contrary, then by Proposition \ref{proposition:nonincreasingdelta}, $\gamma_{\mathcal{M}}(i) = \gamma_{\mathcal{M}}(j)$ for all $j \geq i$. In particular, we have $\gamma_{\mathcal{M}}(n_{\mathcal{M}}) = \gamma_{\mathcal{M}}(n_{\mathcal{M}}-1)$. By Corollary \ref{corollary:keycorollarystrictseq}, $\mathscr{O}_{n_{\mathcal{M}}, n_{\mathcal{M}}-1}$ contains one element. Let $\mathcal{M}(g)$ be an $F$-crystal such that $F_{n_{\mathcal{M}}-1}(\mathcal{M}(g))$ is isomorphic to $F_{n_{\mathcal{M}}-1}(\mathcal{M})$, by Proposition \ref{proposition:bijectionbetweenlifting}, there is a unique $\mathcal{M}(g)$ up to $F$-truncation modulo $p^{n_{\mathcal{M}}}$ such that $F_{n_{\mathcal{M}}-1}(\mathcal{M}(g))$ is isomorphic to $F_{n_{\mathcal{M}}-1}(\mathcal{M})$, thus $F_{n_{\mathcal{M}}}(\mathcal{M}(g))$ is isomorphic to $F_{n_{\mathcal{M}}}(\mathcal{M})$. By Lemma \ref{lemma:F-truncations same isonumber}, $\mathcal{M}(g)$ is isomorphic to $\mathcal{M}$. Hence we conclude that $n_{\mathcal{M}}-1$ is the isomorphism number of $\mathcal{M}$, which is a contradiction.

If $s \geq n_{\mathcal{M}}$, then every $F$-crystal $\mathcal{M}(g)$ such that $F_s(\mathcal{M}(g))$ is isomorphic to $F_s(\mathcal{M})$, is isomorphic to $\mathcal{M}$. Therefore $F_{s+1}(\mathcal{M}(g))$ is isomorphic to $F_{s+1}(\mathcal{M})$, whence $\mathscr{O}_{s+1,s}$ has only one element. By Corollary \ref{corollary:keycorollarystrictseq}, $\gamma_{\mathcal{M}}(s+1) = \gamma_{\mathcal{M}}(s)$ for all $s \geq n_{\mathcal{M}}$.
\end{proof}

We have a converse of Proposition \ref{proposition:ordinaryimpliesdimensionzero}.
\begin{proposition} \label{proposition:dimensionzeroimpliesordinary}
If there exists an $s \geq 1$ such that $\gamma_{\mathcal{M}}(s) = 0$, then $\mathcal{M}$ is ordinary.
\end{proposition}
\begin{proof}
For some $s \geq 1$, if $\gamma_{\mathcal{M}}(s) = 0$, we know that $\gamma_{\mathcal{M}}(1) = 0$ by Theorem \ref{theorem:monotonicitygamma}. By Lemma \ref{lemma:connectiondimension}, we can assume that $\mathrm{dim}(\mathbf{C}_1^0) = 0$. Hence $|I_-^{\pi}| = 0$; see Example \ref{example:stabilizerunipotent} for the definition of $I_-^{\pi}$.

As $(M, \varphi)$ is isomorphic to $(M, g\varphi_{\pi})$ for some $g \equiv 1$ modulo $p$ and the isomorphism number of an ordinary $F$-crystal is less than or equal to $1$ (for example, see \cite[Section 2.3]{Xiao:computing}), in order to show that $\mathcal{M}$ is ordinary, it is enough to show that $(M, \varphi_{\pi})$ is ordinary. Write $\pi$ as a product of disjoint cycle $\pi_u$, it is clear that $(M, \varphi_{\pi}) = \bigoplus_u (M, \varphi_{\pi_u})$. To show that $\mathcal{M}$ is ordinary we can assume that $\pi$ is a cycle and show that $(M, \varphi_{\pi})$ is isoclinic ordinary.  Let $r$ be the rank of $M$.

As $\pi$ is an $r$-cycle, we know that $(\pi \times \pi) = \prod_{u=1}^r (\pi \times \pi)_u$ and each $(\pi \times \pi)_u$ is an $r$-cycle (as a permutation of $I \times I$). Recall a pair $(i,j) \in I \times I$ is said to be in $(\pi \times \pi)_u$ if and only if $(\pi \times \pi)_u(i,j) \neq (i,j)$. It is easy to see that if there is a pair $(i,j) \in I_+$ in $(\pi \times \pi)_u$, then there is also a pair $(i,j) \in I_-$ in $(\pi \times \pi)_u$, and vice versa. Since $I_-^{\pi}$ is an empty set, we know that there is no $(\pi \times \pi)_u$ such that $(\pi \times \pi)_u$ sends an element in $I_-$ to an element in $I_+$ by an argument used in Example \ref{example:stabilizerunipotent}. This means that if there is an element in $I_-$ (or $I_+$ respectively) that is also in $(\pi \times \pi)_u$, then there are elements also elements in $I_0$ and in $I_+$ (or $I_-$ respectively) that are in $(\pi \times \pi)_u$.

The fact that $I_-^{\pi}$ is empty means that for all $(i,j) \in I_-$, if $\nu_{\pi}(i,j)$ is the smallest positive integer such that $(\pi^{\nu_{\pi}(i,j)}(i), \pi^{\nu_{\pi}(i,j)}(j)) \in I_+ \cup I_-$, then it is in $I_-$. Start with an element $(i,j) \in I_-$, and apply this fact recursively. We can see that for every integer $n$, $(\pi^n(i), \pi^n(j)) \notin I_+$. This is a contradiction as we know that there must be some element in $I_+$ that is in $(\pi \times \pi)_u$. Therefore we conclude that every element in $(\pi \times \pi)_u$ is in $I_0$. This means that all the Hodge slopes of $(M, \varphi_{\pi})$ are equal and hence $(M, \varphi_{\pi})$ is isoclinic ordinary.
\end{proof}

\begin{corollary}
The inequality $\gamma_{\mathcal{M}}(1) \geq 0$ is an equality if and only if $\mathcal{M}$ is ordinary. When the equality holds, we have $\gamma_{\mathcal{M}}(s)  = 0$ for all $s \geq 1$.
\end{corollary}

\section{Invariants}
In this section, we introduce several invariants of $F$-crystals over $k$. They are the generalizations of the $p$-divisible groups case introduced in \cite{Vasiu:traversosolved}. It will turn out that these invariants are all equal to the isomorphism number. They provide a good source of computing the isomorphism number from different points of view. All the proofs of this section follow closely the ones of \cite{Vasiu:traversosolved}.

\subsection{Notations} 
Recall that for every $F$-crystal $\mathcal{M} = (M, \varphi)$ and every field extension $k \subset k'$ with $k'$ perfect, we have an $F$-crystal over $k'$ \[\mathcal{M}_{k'} = (M_{k'},\varphi_{k'}) := (M \otimes_{W(k)} W(k'), \varphi \otimes \sigma_{k'}).\]
We denote by $\mathcal{M}^* = (M^*, \varphi)$ the dual of $\mathcal{M}$, where $M^* = \mathrm{Hom}_{W(k)}(M,W(k))$ and $\varphi (f) = \varphi f \varphi^{-1}$ for $f \in M^*$. Note that the pair $(M^*, \varphi)$ is not an $F$-crystal in general, it is just a latticed $F$-isocrystal (meaning that $\varphi$ is an isomorphism after tensored with $B(k)$ but $\varphi(M^*) \not\subset M^*$ in general). We denote by $H_{\infty} = \mathrm{Hom}(\mathcal{M}_1, \mathcal{M}_2)$ the (additive) group of all homomorphisms of $F$-crystals from $\mathcal{M}_1$ to $\mathcal{M}_2$. It is a finitely generated $\mathbb{Z}_p$-module. For every integer $s \geq 1$, let $H_s = \mathrm{Hom}_s(\mathcal{M}_1, \mathcal{M}_2) = \mathbf{Hom}_s(\mathcal{M}_1, \mathcal{M}_2)(k)$ be the (additive) group of all homomorphisms from $F_s(\mathcal{M}_1)$ to $F_s(\mathcal{M}_2)$. It is a $\mathbb{Z}_p/p^s\mathbb{Z}_p$-module but not necessarily finitely generated in general. We denote by $\pi_{\infty, s} : H_{\infty} \to H_s$ and $\pi_{t, s}: H_t \to H_s, t \geq s$ the natural projections. We have two exact sequences:
\[0 \xrightarrow{\; \; \; \; \; \; \;} H_{\infty} \xrightarrow{\; \; \; p^s \; \;\;} H_{\infty} \xrightarrow{\pi_{\infty,s}} H_s,\]
and
\[0  \xrightarrow{\; \; \; \; \; \; \;} H_s \xrightarrow{\; \; \; p \; \; \;} H_{s+1} \xrightarrow{\pi_{s+1,1}} H_1.\]
Let $r_1$ and $r_2$ be the ranks of $M_1$ and $M_2$ respectively. 
\subsection{The endomorphism number}

In this subsection, we generalize the endomorphism number defined in \cite[Section 2]{Vasiu:traversosolved} for $p$-divisible groups. The following proposition is a generalization of \cite[Lemma 2.1]{Vasiu:traversosolved}. For the sake of generality, we will work with the homomorphism version.

\begin{proposition} \label{proposition:eexists}
There exists a non-negative integer $e_{\mathcal{M}_1,\mathcal{M}_2}$ which depends only on $\mathcal{M}_1$ and $\mathcal{M}_2$ with the following property: For every positive integer $n$ and every non-negative integer $e$, we have $\mathrm{Im}(\pi_{\infty, n}) = \mathrm{Im}(\pi_{n+e, n})$ if and only if $e \geq e_{\mathcal{M}_1, \mathcal{M}_2}$.
\end{proposition}
\begin{proof}
We first prove that $e_{\mathcal{M}_1,\mathcal{M}_2}$ exists for each $n$ and then prove that it does not depend on $n$. Note that $\pi_{\infty, n} = \pi_{n+1,n} \circ \pi_{\infty, n+1}$ and $\pi_{n+e, n} = \pi_{n+1,n} \circ \pi_{n+e,n+1}$. If $\mathrm{Im}(\pi_{\infty,n+1}) = \mathrm{Im}(\pi_{n+e,n+1})$ for all $e-1 \geq e_{\mathcal{M}_1, \mathcal{M}_2}(n+1)$, then $\mathrm{Im}(\pi_{\infty,n}) = \mathrm{Im}(\pi_{n+e,n})$. Thus $e_{\mathcal{M}_1, \mathcal{M}_2}(n) \leq e_{\mathcal{M}_1, \mathcal{M}_2}(n+1)+1$. Hence to show that $e_{\mathcal{M}_1, \mathcal{M}_2}(n)$ exists for all positive integer $n$, it is enough to show that $e_{\mathcal{M}_1, \mathcal{M}_2}(n)$ exists for sufficient large $n$.

Let
$H'_n := \mathrm{Hom}_{W_n(k)}((M_1/p^nM_1, \varphi_1), (M_2/p^nM_2, \varphi_2))$. It is the (additive) group of all $W_n(k)$-linear homomorphisms $h:M_1/p^nM_1 \to M_2/p^nM_2$ such that $\varphi_2 h \equiv h \varphi_1$ modulo $p^n$. Thus $H_n$ is a subgroup of $H_n'$.

The existence of $e_{\mathcal{M}_1,\mathcal{M}_2}$ for each $n$ relies on the following commutative diagram:
\[\xymatrixcolsep{4pc} \xymatrix{
H_{\infty} \ar[r]^-{\pi_{\infty, n}} \ar[rd]_-{\pi'_{\infty, n}} &  H_{n} \ar@{^{(}->}[d]& H_{n+e} \ar[l]_-{\pi_{n+e, n}} \ar@{^{(}->}[d]\\
& H'_{n} & H'_{n+e}\ar[l]_-{\pi'_{n+e,n}} }\]
where $\pi'_{\infty,n}$ and $\pi'_{n+e,n}$ are the natural projections.

By \cite[Theorem 5.1.1(a)]{Vasiu:CBP}, we know that for any sufficient large $n$ (in fact $n \geq n_{12}$), there exists a positive integer $e_{\mathcal{M}_1,\mathcal{M}_2}(n)$ such that for all $e\geq e_{\mathcal{M}_1,\mathcal{M}_2}(n)$, $\mathrm{Im}(\pi'_{\infty, n}) = \mathrm{Im}(\pi'_{n+e, n})$. Therefore the images of $\mathrm{Im}(\pi_{\infty,n})$ and $\mathrm{Im}(\pi_{n+e,n})$ in $H'_n$ are the same. Thus $\mathrm{Im}(\pi_{n+e, n}) = \mathrm{Im}(\pi_{\infty, n})$ for all $e \geq e_{\mathcal{M}_1,\mathcal{M}_2}(n)$. This proves that $e_{\mathcal{M}_1,\mathcal{M}_2}(n)$ exists for each $n$.

Now we show that $e_{\mathcal{M}_1,\mathcal{M}_2}(n)$ does not depend on $n$. The proof relies on the following commutative diagram:
\[
\xymatrix{
0 \ar[r] & \mathrm{Im}(\pi_{\infty,n}) \ar[r]^-p \ar[d]^{i_1} & \mathrm{Im}(\pi_{\infty,n+1}) \ar[r] \ar[d]^{i_2} & \mathrm{Im}(\pi_{\infty,1}) \ar[r] \ar[d]^{i_3} & 0 \\
0 \ar[r] & \mathrm{Im}(\pi_{n+e,n}) \ar[r]^-p & \mathrm{Im}(\pi_{n+1+e,n+1}) \ar[r] & \mathrm{Im}(\pi_{1+e,1})
}\]
with horizontal exact sequences and with all vertical maps injective. By the snake lemma, we have an exact sequence
\[0  \to \mathrm{Coker}(i_1) \to \mathrm{Coker}(i_2) \to \mathrm{Coker}(i_3).\]
If we take $e \geq e_{\mathcal{M}_1,\mathcal{M}_2}(n+1)$, then $\mathrm{Coker}(i_2) = 0$. Thus $\mathrm{Coker}(i_1)=0$ and $e_{\mathcal{M}_1,\mathcal{M}_2}(n+1) \geq e_{\mathcal{M}_1,\mathcal{M}_2}(n)$. If we take $e \geq \mathrm{max}(e_{\mathcal{M}_1,\mathcal{M}_2}(n),e_{\mathcal{M}_1,\mathcal{M}_2}(1))$, then $\mathrm{Coker}(i_2)=0$. Thus $e_{\mathcal{M}_1,\mathcal{M}_2}(n+1) \leq \mathrm{max}(e_{\mathcal{M}_1,\mathcal{M}_2}(n),e_{\mathcal{M}_1,\mathcal{M}_2}(1))$. An easy induction on $n \geq 1$ using the sequence of inequalities
\[e_{\mathcal{M}_1,\mathcal{M}_2}(n) \leq e_{\mathcal{M}_1,\mathcal{M}_2}(n+1) \leq \mathrm{max}(e_{\mathcal{M}_1,\mathcal{M}_2}(n),e_{\mathcal{M}_1,\mathcal{M}_2}(1))\]
gives that $e_{\mathcal{M}_1,\mathcal{M}_2}(n) = e_{\mathcal{M}_1,\mathcal{M}_2}(1)$ for all $n$.
\end{proof}

\begin{definition}
The non-negative integer $e_{\mathcal{M}_1,\mathcal{M}_2}$ of Proposition \ref{proposition:eexists} is called the \emph{homomorphism number} of $\mathcal{M}_1$ and $\mathcal{M}_2$. If $\mathcal{M}_1 = \mathcal{M}_2 = \mathcal{M}$, we denote $e_{\mathcal{M},\mathcal{M}}$ by $e_{\mathcal{M}}$ and call it the \emph{endomorphism number} of $\mathcal{M}$.
\end{definition}

The following lemma is a generalization of \cite[Lemma 2.8(c)]{Vasiu:traversosolved} and is proved in a similar way.

\begin{lemma}
Let $k \subset k'$ be an extension of algebraically closed fields. We have $e_{\mathcal{M}_1,\mathcal{M}_2} = e_{\mathcal{M}_{1,k'},\mathcal{M}_{2,k'}}$.
\end{lemma}
\begin{proof}
When $m \geq n$, let $\pi_{m,n} : \mathbf{H}_m \to \mathbf{H}_n$ be the canonical reduction homomorphism and let $\mathbf{H}_{m,n}$ be the scheme theoretic image of $\pi_{m,n}$, which is of finite type over $k$, and whose definition is compatible with base change $k \subset k'$. If $l \geq m$, then $\mathbf{H}_{l,n}$ is a subgroup scheme of $\mathbf{H}_{m,n}$. By the definition of $e_{\mathcal{M}_1,\mathcal{M}_2}$, we have $m-n \geq e_{\mathcal{M}_1,\mathcal{M}_2}$ if and only if $\mathbf{H}_{m,n}(k) = \mathbf{H}_{l,n}(k)$. As $k$ and $k'$ are algebraically closed, we have $\mathbf{H}_{m,n}(k) = \mathbf{H}_{l,n}(k)$ if and only if $\mathbf{H}_{m,n}(k') = \mathbf{H}_{l,n}(k')$. This is further equivalent to $(\mathbf{H}_{m,n})_{k'}(k') = (\mathbf{H}_{l,n})_{k'}(k')$, thus $e_{\mathcal{M}_1,\mathcal{M}_2} = e_{\mathcal{M}_{1,k'},\mathcal{M}_{2,k'}}$.
\end{proof}

\subsection{Coarse endomorphism number} \label{subsection:cendnum}

In this subsection, we generalize the coarse endomorphism number defined in \cite[Section 7]{Vasiu:traversosolved} for $p$-divisible groups. The following proposition is a generalization of \cite[Lemma 7.1]{Vasiu:traversosolved}. Again for the sake of generality, we will work with the homomorphism version.

\begin{lemma} \label{lemma:cendnumexists}
There exists a non-negative integer $f_{\mathcal{M}_1,\mathcal{M}_2}$ that depends on $\mathcal{M}_1$ and $\mathcal{M}_2$ such that for positive integers $m \geq n$, the restriction homomorphism $\pi_{m,n}: H_m \to H_n$ has finite image if and only if $m \geq n+f_{\mathcal{M}_1,\mathcal{M}_2}$.
\end{lemma}
\begin{proof}
As $H_{\infty}$ is a finitely generated $\mathbb{Z}_p$-module, $\mathrm{Im}(\pi_{\infty,n})$ inside the $p^n$-torsion $\mathbb{Z}_p$-module $H_n$ is finite. By Proposition \ref{proposition:eexists}, there exists $f_{\mathcal{M}_1,\mathcal{M}_2}(n)$ such that for all $m\geq n+f_{\mathcal{M}_1,\mathcal{M}_2}(n)$, $\mathrm{Im}(\pi_{m,n}) = \mathrm{Im}(\pi_{\infty,n})$ is finite. 

To show that $f_{\mathcal{M}_1, \mathcal{M}_2}(n)$ is independent of $n$, we consider the exact sequence
\[
\xymatrix{
0 \ar[r] & \mathrm{Im}(\pi_{m,n}) \ar[r]^-p & \mathrm{Im}(\pi_{m,n+1}) \ar[r] & \mathrm{Im}(\pi_{m,1}) 
}.\]
It implies that 
\[f_{\mathcal{M}_1,\mathcal{M}_2}(n) \leq f_{\mathcal{M}_1,\mathcal{M}_2}(n+1) \leq \mathrm{max}(f_{\mathcal{M}_1,\mathcal{M}_2}(n),f_{\mathcal{M}_1,\mathcal{M}_2}(1))\]
An easy induction on $n \geq 1$ shows that $f_{\mathcal{M}_1,\mathcal{M}_2}(n) = f_{\mathcal{M}_1,\mathcal{M}_2}(1)$ for all $n \geq 1$.
\end{proof}

\begin{definition}
The non-negative integer $f_{\mathcal{M}_1,\mathcal{M}_2}$ of Lemma \ref{lemma:cendnumexists} is called the \emph{coarse homomorphism number} of $\mathcal{M}_1$ and $\mathcal{M}_2$. If $\mathcal{M}_1 = \mathcal{M}_2 = \mathcal{M}$, we denote $f_{\mathcal{M},\mathcal{M}}$ by $f_{\mathcal{M}}$ and call it the \emph{coarse endomorphism number} of $\mathcal{M}$.
\end{definition}

\begin{proposition} \label{proposition:f<e}
We have an inequality $f_{\mathcal{M}_1,\mathcal{M}_2} \leq e_{\mathcal{M}_1,\mathcal{M}_2}$.
\end{proposition}
\begin{proof}
It is clear as $\mathrm{Im}(\pi_{\infty,n})$ is finite.
\end{proof}

\begin{lemma} \label{lemma:fextension}
Let $k \subset k'$ be an extension of algebraically closed fields. We have $f_{\mathcal{M}_1,\mathcal{M}_2} = f_{\mathcal{M}_{1,k'},\mathcal{M}_{2,k'}}$.
\end{lemma}
\begin{proof}
For positive integers $m \geq n$, we have $m-n < f_{\mathcal{M}_1,\mathcal{M}_2}$ if and only if the image of $\pi_{m,n}$ is infinite by definition. It is further equivalent to the image of $\mathbf{H}_m \to \mathbf{H}_n$ having positive dimension. This property is invariant under base change from $k$ to $k'$ and hence the lemma follows.
\end{proof}

\subsection{Level torsion} \label{subsection:leveltorsion}

In this subsection, we generalize the level torsion defined in \cite[Section 8.1]{Vasiu:traversosolved} for $p$-divisible groups.

Let $H_{12} $ be the set of all $W(k)$-linear homomorphisms from $M_1$ to $M_2$. We have a latticed $F$-isocrystal $(H_{12}, \varphi_{12})$ where $\varphi_{12}: H_{12} \otimes_{W(k)} B(k) \to H_{12} \otimes_{W(k)} B(k)$ is a $\sigma$-linear isomorphism defined by the rule $\varphi_{12}(h) = \varphi_2 h \varphi^{-1}_1$. By Dieudonn\'e-Manin's classification of $F$-isocrystals, we have finite direct sum decompositions 
\[(M_1 \otimes_{W(k)} B(k), \varphi_1) \cong \bigoplus_{\lambda_1 \in J_1} E^{m_{\lambda_1}}_{\lambda_1}, \qquad (M_2 \otimes_{W(k)} B(k), \varphi_2) \cong \bigoplus_{\lambda_2 \in J_2} E^{m_{\lambda_2}}_{\lambda_2}\] 
where the simple $F$-isocrystals $E_{\lambda_1}$ and $E_{\lambda_2}$ have Newton slopes equal to $\lambda_1$ and $\lambda_2$ respectively, the multiplicities $m_{\lambda_1}, m_{\lambda_2} \in \mathbb{Z}_{> 0}$ and the finite index sets $J_1, J_2 \subset \mathbb{Q}_{>0}$ are uniquely determined. From these decompositions, we obtain a direct sum decomposition
\[(H_{12} \otimes_{W(k)} B(k), \varphi_{12}) \cong L_{12}^+ \oplus L_{12}^0 \oplus L_{12}^-,\]
where
\begin{gather*}
L_{12}^+ = \bigoplus_{\lambda_1<\lambda_2}\mathrm{Hom}(E_{\lambda_1}^{m_{\lambda_1}}, E_{\lambda_2}^{m_{\lambda_2}}), \qquad  L_{12}^- =  \bigoplus_{\lambda_1 > \lambda_2} \mathrm{Hom}(E_{\lambda_1}^{m_{\lambda_1}}, E_{\lambda_2}^{m_{\lambda_2}}), \\ L_{12}^0 = \bigoplus_{\lambda_1=\lambda_2} \mathrm{Hom}(E_{\lambda_1}^{m_{\lambda_1}}, E_{\lambda_2}^{m_{\lambda_2}}).
\end{gather*}
Define	
\begin{gather*}
O_{12}^+ = \bigcap_{i=0}^{\infty} \varphi_{12}^{-i}(H_{12} \cap L_{12}^+),  \qquad O_{12}^- = \bigcap_{i=0}^{\infty} \varphi_{12}^i(H_{12} \cap L_{12}^-),\\
O_{12}^0 = \bigcap_{i=0}^{\infty} \varphi_{12}^{-i}(H_{12} \cap L_{12}^0) = \bigcap_{i=0}^{\infty} \varphi_{12}^i(H_{12} \cap L_{12}^0).
\end{gather*}
Let $A_{12}^0 = \{x \in H_{12} \; | \; \varphi_{12}(x) = x\}$ be the $\mathbb{Z}_p$-algebra that contains the elements fixed by $\varphi_{12}$. For $\dagger \in \{+, 0, -\}$, each $O_{12}^{\dagger}$ is a lattice of $L_{12}^{\dagger}$. We have the following relations:
\[\varphi(O_{12}^+) \subset O_{12}^+, \quad \varphi(O_{12}^0) = O_{12}^0 = A_{12}^0 \otimes_{\mathbb{Z}_p} W(k) = \varphi^{-1}(O_{12}^0), \quad \varphi^{-1}(O_{12}^-) \subset O_{12}^-.\]
Write $O_{12} := O_{12}^+ \oplus O_{12}^0 \oplus O_{12}^-$; it is a lattice of $H_{12} \otimes_{W(k)} B(k)$ inside $H_{12}$. The $W(k)$-module $O_{12}$ is called the \emph{level module} of $\mathcal{M}_1$ and $\mathcal{M}_2$.
\begin{definition}
The \emph{level torsion} of $\mathcal{M}_1$ and $\mathcal{M}_2$ is the smallest non-negative integer $\ell_{\mathcal{M}_1, \mathcal{M}_2}$ such that
\[p^{\ell_{\mathcal{M}_1, \mathcal{M}_2}} H_{12} \subset O_{12} \subset H_{12}.\]
If $\mathcal{M}_1 = \mathcal{M}_2 = \mathcal{M}$, then $\ell_{\mathcal{M}_1, \mathcal{M}_2}$ will be denoted by $\ell_{\mathcal{M}}$.
\end{definition}

\begin{remark} \label{remark:leveltorsiondiff}
The definition of level torsion in this paper is slightly different from the definition in \cite{Vasiu:reconstructing}. When $\mathcal{M}$ is a direct sum of two or more isoclinic ordinary $F$-crystals of different Newton polygons, its isomorphism number is $n_{\mathcal{M}} = 1$. According to the definition in \cite{Vasiu:reconstructing}, the level torsion $\ell_{\mathcal{M}} = 1$ but the definition in this paper gives $\ell_{\mathcal{M}} = 0$.
\end{remark}

For the duals $\mathcal{M}_1^*$ and $\mathcal{M}_2^*$ of $\mathcal{M}_1$ and $\mathcal{M}_2$ respectively, we can define $\ell_{\mathcal{M}_1^*, \mathcal{M}_2^*}$ in a similar way.

\begin{lemma} \label{lemma:leveltorsionequal}
We have $\ell_{\mathcal{M}_1, \mathcal{M}_2} = \ell_{\mathcal{M}_2, \mathcal{M}_1} = \ell_{\mathcal{M}_1^*, \mathcal{M}_2^*}$.
\end{lemma}
\begin{proof}
Write $H_{21} := \mathrm{Hom}(M_2, M_1) \cong \mathrm{Hom}(H_{12}, W(k)) =: H_{12}^*$. There is a direct sum decomposition 
\[H_{12}^* \otimes_{W(k)} B(k) \cong H_{21} \otimes_{W(k)} B(k) = L_{21}^+ \oplus L_{21}^0 \oplus L_{21}^-.\]
It is easy to see that 
\[L_{21}^+ \cong \mathrm{Hom}(L_{12}^-,B(k)) =: L_{12}^{-*}, \qquad  L_{21}^- \cong \mathrm{Hom}(L_{12}^+, B(k)) =: L_{12}^{+*},\]
\[\quad L_{21}^0 \cong \mathrm{Hom}(L_{12}^0,B(k)) =: L_{12}^{0*}\]
are isomorphic as $B(k)$-vector spaces. One can define $O_{21}$ in the same way:
\[O_{21} : =O_{21}^+ \oplus O_{21}^0 \oplus O_{21}^-  \cong O_{12}^{-*}\oplus  O_{12}^{0*} \oplus O_{12}^{+*}\]
and thus $O_{21} \cong O_{12}^*$. Therefore
\[p^{\ell_{\mathcal{M}_2,\mathcal{M}_1}}H_{21} \subset O_{21} \subset H_{21}\quad \textrm{if and only if} \quad p^{\ell_{\mathcal{M}_1,\mathcal{M}_2}}H_{12} \subset O_{12} \subset H_{12},\] 
whence $\ell_{\mathcal{M}_1, \mathcal{M}_2} = \ell_{\mathcal{M}_2,\mathcal{M}_1}$. As $H_{12}^* \cong H_{21}$, we get $\ell_{\mathcal{M}_2,\mathcal{M}_1} = \ell_{\mathcal{M}_1^*,\mathcal{M}_2^*} $.
\end{proof}

\begin{lemma} \label{lemma:leveltorsionrelation}
$\ell_{\mathcal{M}_1 \oplus \mathcal{M}_2} = \mathrm{max}\{\ell_{\mathcal{M}_1},\ell_{\mathcal{M}_2},\ell_{\mathcal{M}_1,\mathcal{M}_2}\}$.
\end{lemma}
\begin{proof}
The direct sum decomposition into $W(k)$-modules of
\[\mathrm{End}(M_1\oplus M_2) = \mathrm{End}(M_1) \oplus \mathrm{End}(M_2) \oplus \mathrm{Hom}(M_1,M_2) \oplus \mathrm{Hom}(M_2,M_1)\]
gives birth to the direct sum decomposition of the level module of $\mathcal{M}_1 \oplus \mathcal{M}_2$
\[O = O_{11} \oplus O_{22} \oplus O_{12} \oplus O_{21}.\]
Hence $\ell_{\mathcal{M}_1 \oplus \mathcal{M}_2} = \mathrm{max}\{\ell_{\mathcal{M}_1},\ell_{\mathcal{M}_2},\ell_{\mathcal{M}_1,\mathcal{M}_2}, \ell_{\mathcal{M}_2,\mathcal{M}_1}\} = \mathrm{max}\{\ell_{\mathcal{M}_1},\ell_{\mathcal{M}_2},\ell_{\mathcal{M}_1,\mathcal{M}_2}\}$ by Lemma \ref{lemma:leveltorsionequal}.
\end{proof}

\begin{lemma} \label{lemma:ellextension}
Let $k \subset k'$ be an extension of algebraically closed fields. We have $\ell_{\mathcal{M}_1,\mathcal{M}_2} = \ell_{\mathcal{M}_{1,k'},\mathcal{M}_{2,k'}}$.
\end{lemma}
\begin{proof}
For $\mathcal{M}_{1,k'}$ and $\mathcal{M}_{2,k'}$, we can define $H'_{12}$ and $O'_{12}$ in an analogous manner. One can check that 
\[H'_{12} = H_{12} \otimes_{W(k)} W(k'), \qquad O'_{12} = O_{12} \otimes_{W(k)} W(k').\]
The lemma follows easily.
\end{proof}

\section{Proof of the Main Theorem} \label{section:proofmaintheorem}

The proofs of this section follow closely the ones of \cite[Section 8]{Vasiu:traversosolved}.

\subsection{Notations} For this section, we denote by $H:=H_{12}$ the group of $W(k)$-linear homomorphisms from $M_1$ to $M_2$, and $H_{s}$ the group of homomorphisms from $F_s(\mathcal{M}_1)$ to $F_s(\mathcal{M}_2)$. For simplicity, we denote $O_{12}$ by $O$, $O_{12}^{\dagger}$ by $O^{\dagger}$ for $\dagger \in \{+,0,-\}$, and $A^0_{12}$ by $A^0$. 

\subsection{The inequality $e_{\mathcal{M}_1,\mathcal{M}_2} \leq \ell_{\mathcal{M}_1,\mathcal{M}_2}$} \label{subsection:e<f}
We will follow the ideas of \cite[Section 8.2]{Vasiu:traversosolved} and prove that $\mathrm{Im}(\pi_{\infty,1}) = \mathrm{Im}(\pi_{\ell_{\mathcal{M}_1,\mathcal{M}_2}+1,1})$. For any $\bar{h} \in \mathrm{Im}(\pi_{\ell_{\mathcal{M}_1,\mathcal{M}_2}+1,1})$, let $h \in H_{\ell_{\mathcal{M}_1,\mathcal{M}_2}+1}$ be a preimage of $\bar{h}$. Hence $ \varphi_2  h \varphi^{-1}_1 \equiv h$ modulo $p^{\ell_{\mathcal{M}_1,\mathcal{M}_2}+1}$, that is, $ \varphi_2  h \varphi^{-1}_1 - h \in p^{\ell_{\mathcal{M}_1,\mathcal{M}_2}+1} \mathrm{Hom}(M_1, M_2) \subset pO$. By Lemma \ref{lemma:solveequation} below, there exists $h'' \in pO$ such that
\[\varphi_2  h \varphi^{-1}_1 - h = \varphi_2  h'' \varphi^{-1}_1 - h''.\]
Thus $h' := h-h'' \in H_{\infty}$ is a homomorphism whose image in $H_1$ is exactly $\bar{h}$.

\begin{lemma} \label{lemma:solveequation}
For each $x \in O$, the equation $x = \varphi_{12}(X) - X$ in $X$ has a solution in $O$ that is unique up to the addition of elements in $A^0$. Moreover, if $x \in p^sO$, then there exists a solution $X \in p^sO$.
\end{lemma}
\begin{proof}
Writing $x = x^+ + x^0 + x^-$ with $x^{\dagger} \in O^{\dagger}$ for $\dagger \in \{+, 0, -\}$, we will find $y^{\dagger} \in O^{\dagger}$ such that $x^{\dagger} = \varphi_{12}(y^{\dagger}) - y^{\dagger}$ for each $\dagger \in \{+,0,-\}$. Therefore $y = y^++y^0+y^-$ is a solution of the given equation.

Let $y^+ = -\sum_{i=0}^{+\infty} \varphi_{12}^i(x^+)$, and $y^- = \sum_{i=1}^{+\infty} \varphi_{12}^{-i}(x^-)$. Because $\varphi_{12}(O^+) \subset O^+$ and $\varphi_{12}^{-1}(O^-) \subset O^-$, we have $y^+ \in O^+$ and $y^- \in O^-$. It is easy to check that $x^+ = \varphi_{12}(y^+)-y^+$ and $x^- = \varphi_{12}(y^-)-y^-$.

Let $\{v_1, v_2, \dots, v_r\}$ be a $\mathbb{Z}_p$-basis of $A^0$; it is also a $W(k)$-basis of $O^0$. We also write $x^0  = \sum_{i=1}^d x_iv_i$. For $1 \leq i \leq r$, let $z_i \in W(k)$ be a solution of $\sigma(z_i)-z_i = x_i$ and put $y^0 = \sum_{i=1}^d z_iv_i \in O^0$. Using the fact that $\varphi_{12}(v_i) = v_i$ for all $1 \leq i \leq r$, it is easy to check that $x^0 = \varphi_{12}(y^0)-y^0$.

If $y, y' \in O$ satisfy the equation $x = \varphi_{12}(X) - X$, we have $\varphi_{12}(y) - y = \varphi_{12}(y')-y'$, i.e. $\varphi_{12}(y-y') = y-y'$, whence $y-y' \in A^0$.

If $x  = p^sx' \in p^sO$, then $y = p^sy' \in p^sO$ will be a solution of $x = \varphi_{12}(X) - X$ where $y' \in O$ is a solution of $x' = \varphi_{12}(X) - X$.
\end{proof}
\subsection{The inequality $\ell_{\mathcal{M}_1,\mathcal{M}_2} \leq f_{\mathcal{M}_1,\mathcal{M}_2}$} \label{subsection:l<f}
We follow the ideas of \cite[Section 8.3]{Vasiu:traversosolved}. By Lemmas \ref{lemma:fextension} and \ref{lemma:ellextension}, we can assume that $k \supset k'[[\alpha]]=R$ where $k \supset k'$ is an extension of algebraically closed fields and for $i=1,2$, we have 
\[(M_i,\varphi_i) \cong (M'_{i} \otimes_{W(k')} W(k), \varphi'_{i} \otimes \sigma),\] 
where $(M'_{i},\varphi'_{i})$ are $F$-crystals over $k'$. Let $\mathfrak{m}$ be the ideal of $R$ generated by $\alpha$. Let $H'$ and $O'$ be the analogues of $H$ and $O$ obtained from $(M'_{i},\varphi'_{i})$ instead of $(M_i,\varphi_i)$. It is easy to check that 
\[H = H' \otimes_{W(k')} W(k), \qquad O = O' \otimes_{W(k')} W(k),\]
and $p^jO \cap O' = p^jO'$ for $j \in \mathbb{Z}_{\geq 0}$.

Let $x = x^+ + x^0 + x^- \in O'$ where $x^{\dagger} \in O'^{\dagger}$, $\dagger \in \{+,0,-\}$.

\begin{lemma} \label{lemma:mainproofsolveequation}
For each $\eta \in W(\mathfrak{m})$, the equation $\eta x = \varphi_{12}(X) - X$ has a solution $x_{\eta} \in O$, that is unique up to the addition of an element of $A^{0}$.
\end{lemma}
\begin{proof}
Put
\[x_{\eta}^{+} = -\sum_{i=0}^{\infty} \varphi_{12}^{i}(\eta x^+) \in O^+,  x^-_{\eta} = \sum_{i=1}^{\infty} \varphi_{12}^{-i}(\eta x^-) \in O^-,  x_{\eta}^{0} = -\sum_{i=0}^{\infty} \varphi_{12}^{i}(\eta x^0) \in O^0.\]
The elements $x_{\eta}^{\pm}$ are well-defined as $\{\varphi_{12}^i(x^+)\}_{i \in \mathbb{Z}_{\geq 0}}$ and $\{\varphi_{12}^{i}(x^-)\}_{i \in \mathbb{Z}_{\geq 0}}$ are $p$-adically convergent in $O'^+$ and $O'^-$ respectively. As $\{\sigma^i(\eta)\}_{i \in \mathbb{Z}_{\geq 0}}$ is a $\alpha$-adically convergent in $W(R)$, $x_{\eta}^0$ is convergent in $O^0$. One can check that 
\[x_{\eta} := x_{\eta}^++x_{\eta}^0+x_{\eta}^- \in O\]
satisfies $\eta x = \varphi_{12}(x_{\eta})-x_{\eta}$.

Suppose $\eta x = \varphi_{12}(x_{\eta}) - x_{\eta}$ and $\eta x = \varphi_{12}(x^{\circ}_{\eta}) - x^{\circ}_{\eta}$, we have $x_{\eta} - x^{\circ}_{\eta} = \varphi_{12}(x_{\eta} - x^{\circ}_{\eta})$, hence the lemma.
\end{proof}

We define a homomorphism of abelian groups $\Omega_x:W(\mathfrak{m}) \to H/A^0$ by the formula $\Omega_x (\eta) = x_{\eta} + A^0$ where $x_{\eta} \in O \subset H$ satisfies $\eta x = \varphi_{12}(x_{\eta}) - x_{\eta}$. By Lemma \ref{lemma:mainproofsolveequation}, it is well-defined.

Let $x \in p^{\ell_{\mathcal{M}_{1},\mathcal{M}_{2}}}H' \setminus pO'$. For all $\eta \in W(\mathfrak{m})$, $\varphi_{12}(x_{\eta}) - x_{\eta} = \eta x \in p^{\ell_{\mathcal{M}_{1},\mathcal{M}_{2}}}H'$, thus $\varphi_{12}(x_{\eta}) \equiv x_{\eta}$ modulo $p^{\ell_{\mathcal{M}_{1},\mathcal{M}_{2}}}$. This implies that $x_{\eta}$ is a homomorphism modulo $p^{\ell_{\mathcal{M}_{1},\mathcal{M}_{2}}}$ from $\mathcal{M}_1$ to $\mathcal{M}_2$. Hence $x_{\eta} \in H_{\ell_{\mathcal{M}_{1},\mathcal{M}_{2}}}$. Clearly, every homomorphism of $F$-crystals is a homomorphism modulo powers of $p$. Hence $A^0 \subset H_{\ell_{\mathcal{M}_1,\mathcal{M}_2}}$. Thus the image of $\Omega_x$ is in $H_{\ell_{\mathcal{M}_1,\mathcal{M}_2}}/A^0$. 

Suppose $f_{\mathcal{M}_{1},\mathcal{M}_{2}} < \ell_{\mathcal{M}_{1},\mathcal{M}_{2}}$, we will show that $x \in pO'$, which is a contradiction! Let $\bar{\pi}_{\ell_{M_1,M_2},1} : H_{\ell_{\mathcal{M}_1,\mathcal{M}_2}}/A^0 \to H_{1}/A^0$ be the homomorphism induced by ${\pi}_{\ell_{M_1,M_2},1}$. The image of 
\[\bar{\pi}_{\ell_{M_1,M_2},1} \circ \Omega_x : W(\mathfrak{m}) \to H_{\ell_{\mathcal{M}_1,\mathcal{M}_2}}/A^0 \to H_{1}/A^0\]
takes only finitely many values $H_{1}/A^0$ as $\mathrm{Im}(\pi_{\ell_{\mathcal{M}_{1},\mathcal{M}_{2}},1})$ is finite by the assumption that $f_{\mathcal{M}_{1},\mathcal{M}_{2}} < \ell_{\mathcal{M}_{1},\mathcal{M}_{2}}$.  Since $\mathfrak{m}$ is infinite (and thus $W(\mathfrak{m})$ is infinite), the kernel of $\bar{\pi}_{\ell_{M_1,M_2},1} \circ \Omega_x$ is infinite. There exists $\eta = (\eta_0,\eta_1,\dots) \in W(\mathfrak{m})$ with $\eta_0 \neq 0$ such that $x_{\eta} \in pH$. Thus $x_{\eta} \in O \cap pH =:N$. Let $N' = O' \cap pH'$, we have $N \cong N' \otimes_{W(k')} W(k)$.

\begin{lemma} \label{lemma:mainproofequivalencemap}
An element $\bar{z} \in O/pO$ lies in $N/pO$ if and only if for every $k'$-linear map $\rho:O'/pO' \to k'$ with $\rho(N'/pO') = 0$ we have $(\rho \otimes 1_{k})(\bar{z}) = 0$.
\end{lemma}
\begin{proof}
For every $k'$-linear map $\rho: O'/pO' \to k'$ with $\rho(N'/pO')=0$, 
\[\mathrm{Ker}(\rho \otimes 1_k) = \mathrm{Ker}(\rho) \otimes_{k'} k \supset N'/pO' \otimes_{k'}k = N/pO.\]
Set $S = \{\rho:O'/pO' \to k \; | \; \rho(N'/pO') = 0\}$. We have $\bigcap_{\rho \in S} \mathrm{Ker}(\rho \otimes 1_k) = N/pO$. This concludes the proof.
\end{proof}

By Lemma \ref{lemma:mainproofequivalencemap}, for every $k'$-linear map $\rho : O'/pO' \to k'$ such that $\rho(N'/pO') = 0$, we have $(\rho \otimes 1_k)(\bar{x}_{\eta}) = 0$. Therefore the following equality holds in $R$
\begin{equation} \label{equation:mainproofeq1}
\sum_{i=0}^{\infty} \rho(\varphi_{12}^i(\bar{x}^+)) \eta_0^{p^i} + \sum_{i=0}^{\infty}\rho(\varphi_{12}^i(\bar{x}^0)) \eta_0^{p^i} - \sum_{i=1}^{\infty}\rho(\varphi_{12}^{-i}(\bar{x}^-)) \eta_0^{p^{-i}}=(\rho \otimes 1_k)(\bar{x}_{\eta}) = 0.
\end{equation}
Because the Newton slopes of $(O^+,\varphi_{12})$ and $(O^-,\varphi_{12}^{-1})$ are positive, there exists a big enough $n$ such that 
$\varphi_{12}^i(x_+) \in pO'^+$ and $\varphi_{12}^{-i}(x_-) \in pO'^-$ for $i> n.$ As $\rho(N'/pO') = 0$,
\begin{equation} \label{equation:mainprooffact1}
\rho(\varphi_{12}^i(\bar{x}^+)) = 0, \quad \rho(\varphi_{12}^{-i}(\bar{x}^-)) = 0,\quad \forall \; i > n.
\end{equation}
Thus \eqref{equation:mainproofeq1} is reduced to
\begin{equation} \label{equation:mainproofeq2}
-\sum_{i=-n}^{-1} \rho(\varphi_{12}^{i}(\bar{x}^-))\eta_0^{p^i} + \sum_{i=0}^n (\rho(\varphi_{12}^i(\bar{x}^+)) + \rho(\varphi_{12}^i(\bar{x}^0)))\eta_0^{p^i}+\sum_{i=n+1}^{\infty}\rho(\varphi_{12}^i(\bar{x}^0))\eta_0^{p^i}=0.
\end{equation}
Write 
\begin{gather*}
\Phi(\beta) = -\sum_{i=0}^{n-1} \rho(\varphi_{12}^{i-n}(\bar{x}^-)) \beta^{p^i} + \sum_{i=n}^{2n} (\rho(\varphi_{12}^{i-n}(\bar{x}^+)) + \rho(\varphi_{12}^{i-n}(\bar{x}^0)))\beta^{p^i} \\+\sum_{i=2n+1}^{\infty}\rho(\varphi_{12}^{i-n}(\bar{x}^0))\beta^{p^i} \in k'[[\beta]].
\end{gather*}
Then \eqref{equation:mainproofeq2} is equivalent to  $\Phi(\eta_0^{p^{-n}}) = 0$ where $\eta_0^{p^{-n}} \in \alpha^{p^{-n}}k'[[\alpha^{p^{-n}}]]$. As $\eta_0^{p^{-n}} \neq 0$, we deduce that $\Phi(\beta) = 0$ by \cite[Lemma 8.9]{Vasiu:traversosolved}. Combining \eqref{equation:mainprooffact1}, we get
\begin{equation} \label{equation:mainprooffact2}
\begin{split}
\rho(\varphi_{12}^{-i}(\bar{x}^-)) & = 0,  \; \forall \; i \geq 1 ,\qquad \rho(\varphi_{12}^i(\bar{x}^0)) = 0, \quad \rho(\varphi_{12}^i(\bar{x}^+)) = 0, \; \forall \; i>n
 \\
& \rho(\varphi_{12}^i(\bar{x}^+)) + \rho(\varphi_{12}^i(\bar{x}^0)) = 0, \; \forall \; i = 0, \dots, n. \\
\end{split}
\end{equation}
As $\varphi_{12}$ is bijective on $O'^0$ and thus on $O'^0/pO'^0$, the subspace $V \subset O'^0/pO'^0$ generated by $\{\varphi_{12}^i(\bar{x}^0) \; | \; i \geq 0\}$ satisfies $\varphi_{12}(V) = V$ and thus $\varphi_{12}^{j}(V) = V$ for every $j \geq 0$. This implies that $V$ is generated by $\{\varphi_{12}^i(\bar{x}_0) \; | \; i > n\}$ and hence for $0 \leq i \leq n$, $\varphi_{12}^i(\bar{x}_0)$ is a linear combination of elements in $\{\varphi_{12}^i(\bar{x}_0) \; | \; i > n\}$ whence $\rho(\varphi_{12}^i(\bar{x}_0)) = 0$ for all $i = 0, \dots, n$. This allows us to extend \eqref{equation:mainprooffact2} to get
\begin{equation} \label{equation:mainprooffact3}
\rho(\varphi_{12}^{-i}(\bar{x}^-))  = 0,  \; \forall \; i \geq 1 ,\qquad \rho(\varphi_{12}^i(\bar{x}^0)) = 0, \quad \rho(\varphi_{12}^i(\bar{x}^+)) = 0, \; \forall \; i \geq 0.
\end{equation}
Finally, since $x \in p^{\ell_{\mathcal{M}_{1},\mathcal{M}_{2}}}H'$ and $\ell_{\mathcal{M}_{1},\mathcal{M}_{2}} > f_{\mathcal{M}_{1},\mathcal{M}_{2}} \geq 0$, we have $x \in pH'$ and thus $x \in pH' \cap O' =: N'$. As $\rho(N'/pO') = 0$, we have $0 = \rho(\bar{x}) = \rho(\bar{x}^++\bar{x}^0+\bar{x}^-) = \rho(\bar{x}^-)$. Thus \eqref{equation:mainprooffact3} can be further extended to
\begin{equation} \label{equation:mainprooffact4}
\rho(\varphi_{12}^{i}(\bar{x}^+))  = 0, \qquad \rho(\varphi_{12}^i(\bar{x}^0)) = 0, \qquad \rho(\varphi_{12}^{-i}(\bar{x}^-)) = 0, \qquad \forall \; i \geq 0.
\end{equation}
By Lemma \ref{lemma:mainproofequivalencemap} and \eqref{equation:mainprooffact4}, we have $\varphi_{12}^{i}(x^+), \varphi_{12}^i(x^0), \varphi_{12}^{-i}(x^-) \in pH$ and thus in $pH'$ for all $i \geq 0$. By the definition of $O'$, we have $x = x^++x^0+x^- \in pO'$. This reaches the desired contradiction.

\subsection{The equality $f_{\mathcal{M}} = n_{\mathcal{M}}$} \label{subsection:f=n} In this subsection, we show that $f_{\mathcal{M}} = n_{\mathcal{M}}$ when $\mathcal{M}$ is not an ordinary $F$-crystal. Thus in this case, $n_{\mathcal{M}}>0$. Recall $E_s$ is the set of all endomorphisms of $F_s(\mathcal{M})$ and $\mathbf{E}_s(k) = E_s$. The restriction homomorphism $\pi_{s,1}: E_s \to E_1$ has finite image if and only if the image of $\pi_{s,1}:\mathbf{E}_s \to \mathbf{E}_1$ has zero dimension, if and only if $s \geq 1 + f_{\mathcal{M}}$ by definition. The dimension of $\pi_{s,1}$ is $\gamma_{\mathcal{M}}(s) - \gamma_{\mathcal{M}}(s-1)$. It is zero if and only if $s > n_{\mathcal{M}}$ by Theorem \ref{theorem:monotonicitygamma}. As $s \geq 1 + f_{\mathcal{M}}$ if and only if $s > n_{\mathcal{M}}$, we conclude that $f_{\mathcal{M}} = n_{\mathcal{M}}$.

\subsection{Conclusion} By Subsections \ref{subsection:e<f}, \ref{subsection:l<f}, \ref{subsection:f=n} and Proposition \ref{proposition:f<e}, we have the following two theorems:
\begin{theorem} \label{theorem:maintheorem2}
We have equalities $f_{\mathcal{M}_1,\mathcal{M}_2} = e_{\mathcal{M}_1,\mathcal{M}_2} = \ell_{\mathcal{M}_1,\mathcal{M}_2}.$
\end{theorem}
\begin{theorem} \label{theorem:maintheorem3}
If $\mathcal{M}$ is not ordinary, then $n_{\mathcal{M}} = f_{\mathcal{M}} = e_{\mathcal{M}} = \ell_{\mathcal{M}}$.
\end{theorem}
\begin{corollary}
We have equalities $f_{\mathcal{M}_1,\mathcal{M}_2} = f_{\mathcal{M}_2,\mathcal{M}_1} = f_{\mathcal{M}^*_1,\mathcal{M}^*_2}$ and $e_{\mathcal{M}_1,\mathcal{M}_2} = e_{\mathcal{M}_2,\mathcal{M}_1} = e_{\mathcal{M}^*_1,\mathcal{M}^*_2}$.
\end{corollary}
\begin{proof}
This is clear by Theorem \ref{theorem:maintheorem2} and Lemma \ref{lemma:leveltorsionequal}.
\end{proof}

\section{Application to $F$-crystal of rank $2$}
In \cite[Theorem 1.4]{Xiao:computing}, we proved that if $\mathcal{M}$ is a non-isoclinic $F$-crystal of rank $2$, and is not a direct sum of two $F$-crystals of rank $1$, then $n_{\mathcal{M}} \leq 2 \lambda_1$ where $\lambda_1$ is the smallest Newton slope of $\mathcal{M}$. Now we show that the inequality is in fact an equality. For the sake of completeness, we state the theorem of isomorphism number of rank $2$ in all cases.
\begin{theorem} \label{theorem:rank2}
Let $\mathcal{M}$ be an $F$-crystal of rank $2$ with Hodge slopes $0$ and $e>0$. Let $\lambda_1$ be the smallest Newton slope of $\mathcal{M}$. Then we have the following three disjoint cases:
\begin{enumerate}[(i)]
\item if $\mathcal{M}$ is a direct sum of two $F$-crystals of rank $1$, then $n_{\mathcal{M}} = 1$;
\item if $\mathcal{M}$ is not a direct sum of two $F$-crystals of rank $1$ and is isoclinic, then $n_{\mathcal{M}} = e$;
\item if $\mathcal{M}$ is not a direct sum of two $F$-crystals of rank $1$ and is non-isoclinic, then $n_{\mathcal{M}} = 2\lambda_1$.
\end{enumerate}
\end{theorem}
\begin{proof}
Parts (i) and (ii) are proved in \cite[Theorem 1.4 (i) and (ii)]{Xiao:computing}. In the case of Part (iii), \cite[Theorem 1.4 (iii)]{Xiao:computing} proves only the inequality $n_{\mathcal{M}} \leq 2 \lambda_1$. The proof of \cite[Theorem 1.4 (iii)]{Xiao:computing} has a minor mistake that can be easily fixed. In this paper, we will only prove the equality $n_{\mathcal{M}} = 2\lambda_1$ in the case of Part (iii).

We show that $\lambda_1 >0$ by showing that the assumption that $\lambda_1= 0$ leads to a contradiction. If $\lambda_1 = 0$, then the Hodge polygon and the Newton polygon  of $\mathcal{M}$ coincide. By \cite[Theorem 1.6.1]{Katz:slopefiltration}, we can decompose $\mathcal{M}$ into a direct sum of two $F$-crystals of rank $1$. Hence $\lambda_1>0$. Let $\lambda_2$ be the other Newton slope of $\mathcal{M}$. As $\mathcal{M}$ is not isoclinic, $\lambda_1 < \lambda_2$. It is easy to see that $\lambda_1$ and $\lambda_2$ are two positive integers. Hence there is a $W(k)$-basis $\mathcal{B}_1 = \{x_1, x_2\}$ of $M$ such that $\varphi(x_1) = p^{\lambda_1}x_1$ and $\varphi(x_2) = ux_1+p^{\lambda_2}x_2$ where $u \in W(k)$. If $u$ is a non-unit and belongs to $pW(k)$, then $\varphi(M) \subset pM$ and thus the smallest Hodge slope of $\mathcal{M}$ must be positive. This contradicts the assumption of the proposition, hence $u$ is a unit. By solving equations of the form $\varphi(z) = p^{\lambda_1}z$ and $\varphi(z) = p^{\lambda_2}z$, we find a $B(k)$-basis $\mathcal{B}_2 = \{y_1=x_1, y_2 = vx_1+p^{\lambda_1}x_2\}$ of $M[1/p]$ with $v$ a unit in $W(k)$ such that $\sigma(v)+u=p^{\lambda_2-\lambda_1}v$. It is easy to see that there is a unique $v$ satisfying this equation. 

Let $\mathcal{B}_1 \otimes \mathcal{B}_1^*$ be the $W(k)$-basis of $\mathrm{End}(M)$ that contains $x_i \otimes x_j^*$ for all $1 \leq i, j \leq 2$, where $(x_i \otimes x_j^*)(x_j) = x_i$. It is a $B(k)$-basis of $\mathrm{End}(M[1/p])$. We compute the formula of $\varphi :\mathrm{End}(M[1/p]) \to \mathrm{End}(M[1/p])$ with respect to $\mathcal{B}_1$ as follows:
\begin{align*}
\varphi(x_1 \otimes x_1^*) &= x_1 \otimes x_1^* - p^{-\lambda_2}ux_1 \otimes x_2^*, \\
\varphi(x_2 \otimes x_1^*) &= p^{-\lambda_1}u x_1 \otimes x_1^* + p^{\lambda_2 - \lambda_1} x_2 \otimes x_1^* - p^{-\lambda_1-\lambda_2}u^2x_1 \otimes x_2^* - p^{-\lambda_1}u x_2 \otimes x_2^*, \\
\varphi(x_1 \otimes x_2^*) &= p^{\lambda_1-\lambda_2}x_1 \otimes x_2^*, \\
\varphi(x_2 \otimes x_2^*) &= p^{-\lambda_2}ux_1 \otimes x_2^* + x_2 \otimes x_2^*.
\end{align*}

Similarly the set $\mathcal{B}_2 \otimes \mathcal{B}_2^*$ is another $B(k)$-basis of $\mathrm{End}(M[1/p])$. As $\varphi(y_1)=p^{\lambda_1}y_1$ and $\varphi(y_2) = p^{\lambda_2}y_2$, we compute the formula of $\varphi :\mathrm{End}(M[1/p]) \to \mathrm{End}(M[1/p])$ with respect to $\mathcal{B}_2$ as follows: 
\begin{align*}
\varphi(y_2 \otimes y^*_1) &= p^{\lambda_2-\lambda_1} y_2 \otimes y^*_1, \quad & \varphi(y_1 \otimes y^*_1) &= y_1 \otimes y^*_1, \\
\varphi(y_2 \otimes y^*_2) &= y_2 \otimes y^*_2, \quad &\varphi(y_1 \otimes y^*_2) &= p^{\lambda_1-\lambda_2} y_1 \otimes y_2^*.
\end{align*}
Therefore, we have found $B(k)$-bases for
\[L^+ = \langle y_2 \otimes y^*_1 \rangle_{B(k)}, \quad L^0 = \langle y_1 \otimes y^*_1, y_2 \otimes y^*_2 \rangle_{B(k)}, \quad L^- = \langle y_1 \otimes y^*_2 \rangle_{B(k)}.\]
We compute the change of basis matrix from $\mathcal{B}_1 \otimes \mathcal{B}_1^*$ to $\mathcal{B}_2 \otimes \mathcal{B}_2^*$ as follows:
\begin{align*}
y_1 \otimes y_1^* &= x_1 \otimes x^*_1-\frac{v}{p^{\lambda_1}}x_1 \otimes x_2^*,\\
y_2 \otimes y_1^* &= vx_1 \otimes x_1^*+p^{\lambda_1}x_2 \otimes x_1^*-\frac{v^2}{p^{\lambda_1}}x_1 \otimes x_2^*-vx_2 \otimes x_2^*, \\
y_1 \otimes y^*_2 &= \frac{1}{p^{\lambda_1}}x_1 \otimes x^*_2,\\
y_2 \otimes y^*_2 &= \frac{v}{p^{\lambda_1}}x_1 \otimes x^*_2 + x_2 \otimes x^*_2.
\end{align*}
It is easy to see that $p^{\lambda_1} y_i \otimes y^*_j \in \mathrm{End}(M) \; \backslash \; p\mathrm{End}(M)$ for $i, j \in \{1,2\}$. We get that 
\begin{enumerate}[(a)]
\item $O^+ = \langle p^{\lambda_1}y_2 \otimes y_1 \rangle_{W(k)}$; 
\item $N := \langle y_1 \otimes y^*_1+y_2 \otimes y_2^*, p^{\lambda_1}y_2\otimes y^*_2\rangle_{W(k)} \subset O^0$ is a lattice; 
\item $O^- = \langle p^{\lambda_1}y_1 \otimes y^*_2 \rangle_{W(k)}$. 
\end{enumerate}
We now show that in fact $N=O^0$. As $O^0 = A^0 \otimes W(k)$, it is enough to show that $A^0 \subset N$. Suppose 
\[ax_1 \otimes x_1^* + bx_2 \otimes x_1^* + cx_1 \otimes x_2^* + dx_2 \otimes x_2^* \in A^0,\]
we have
\begin{align*}
& \varphi(ax_1 \otimes x_1^* + bx_2 \otimes x_1^* + cx_1 \otimes x_2^* + dx_2 \otimes x_2^*) \\
= & (\sigma(a)-\sigma(b)p^{-\lambda_1}u) x_1 \otimes x_1^* + \sigma(b)p^{\lambda_2-\lambda_1}x_2 \otimes x_1^* +(-\sigma(b)p^{-\lambda_1}u+\sigma(d))x_2 \otimes x_2^*\\
&+(-\sigma(a)p^{-\lambda_2}u-\sigma(b)p^{-\lambda_1-\lambda_2}u^2+\sigma(c)p^{\lambda_1-\lambda_2}+\sigma(d)p^{-\lambda_2}u)x_1\otimes x_2^*\\
=& ax_1 \otimes x_1^* + bx_2 \otimes x_1^* + cx_1 \otimes x_2^* + dx_2 \otimes x_2^* .
\end{align*}
Hence
\begin{align}
a &= \sigma(a)-\sigma(b)p^{-\lambda_1}u,  \label{equation:solvinga}\\
b &= \sigma(b)p^{\lambda_2-\lambda_1}, \label{equation:solvingb} \\
c &= -\sigma(a)p^{-\lambda_2}u-\sigma(b)p^{-\lambda_1-\lambda_2}u^2+\sigma(c)p^{\lambda_1-\lambda_2}+\sigma(d)p^{-\lambda_2}u, \label{equation:solvingc}\\
d &= -\sigma(b)p^{-\lambda_1}u + \sigma(d). \label{equation:solvingd}
\end{align}
By \eqref{equation:solvingb}, we know that $b=0$. Hence $a=\sigma(a)$, $d = \sigma(d)$ by \eqref{equation:solvinga} and \eqref{equation:solvingd}, and
\[c = -ap^{-\lambda_2}u+\sigma(c)p^{\lambda_1-\lambda_2}+dp^{-\lambda_2}u,\]
by \eqref{equation:solvingc}, namely
\[p^{\lambda_1}(p^{\lambda_2-\lambda_1}c-\sigma(c)) = (d-a)u.\]
In order to have a solution for $c$, we need $d-a \in p^{\lambda_1}W(k)$. Let $d-a=p^{\lambda_1}\alpha$ for some $\alpha \in \mathbb{Z}_p$ as $a, d \in \mathbb{Z}_p$. Then we have a unique solution $c$ such that  
\[p^{\lambda_2-\lambda_1}c-\sigma(c) = \alpha u.\]
As $u = p^{\lambda_2 - \lambda_1}v-\sigma(v)$, we get $c=\alpha v$. It is now easy to see that 
\begin{align*}
& ax_1 \otimes x_1^* + bx_2 \otimes x_1^* + c x_1 \otimes x_2^* + d x_2 \otimes x_2^* \\
= & a(x_1 \otimes x_1^* + x_2 \otimes x_2^*) + (\alpha v) x_1 \otimes x_2^*+ (d-a)x_2 \otimes x_2^* \\
= & a(y_1 \otimes y_1^* + y_2 \otimes y_2^*)  + \alpha p^{\lambda_1}y_2 \otimes y_2^* \in N.
\end{align*}
Hence $N = O^0$.

The change of basis matrix from $\{y_1 \otimes y_1^*+y_2 \otimes y_2^*, p^{\lambda_1}y_2 \otimes y_1^*, p^{\lambda_1}y_1 \otimes y_2^*, p^{\lambda_1}y_2 \otimes y_2^*\}$ to $\mathcal{B}_1 \otimes \mathcal{B}_1^*$ is
\[A = \begin{pmatrix} 1 & p^{\lambda_1}v & 0 & 0 \\ 0 & p^{2\lambda_1} & 0 & 0 \\ 0 & -v^2 & 1 & v \\ 1 & -p^{\lambda_1}v & 0 & p^{\lambda_1} \end{pmatrix} \]
To find an upper bound of $\ell_{\mathcal{M}}$, we compute the inverse of $A$:
\[A^{-1} = \frac{1}{p^{2\lambda_1}}\begin{pmatrix} p^{2\lambda_1} & -p^{\lambda_1}v & 0 & 0 \\ 0 & 1 & 0 & 0 \\ p^{\lambda_1}v & -v^2 & p^{2\lambda_1} & -p^{\lambda_1}v \\ -p^{\lambda_1} & 2v & 0 & p^{\lambda_1} \end{pmatrix} \]
Thus the smallest number $\ell$ such that all entries of $p^{\ell}A^{-1} \in W(k)$ is $2\lambda_1$. Hence $\ell_{\mathcal{M}} = 2\lambda_1$. By Theorem \ref{theorem:maintheorem}, we have $n_{\mathcal{M}} = 2\lambda_1$.
\end{proof}

\section*{Acknowledgement}
The author would like to thank Adrian Vasiu for several suggestions and conversations and the anonymous referee for numerous suggestions. 

\bibliographystyle{amsplain}
\bibliography{references}

\providecommand{\bysame}{\leavevmode\hbox to3em{\hrulefill}\thinspace}
\providecommand{\MR}{\relax\ifhmode\unskip\space\fi MR }
\providecommand{\MRhref}[2]{%
  \href{http://www.ams.org/mathscinet-getitem?mr=#1}{#2}
}
\providecommand{\href}[2]{#2}
\begin{thebibliography}{10}

\bibitem{Neronmodels}
Siegfried Bosch, L\"{u}kebohmert Werner, and Michel Raynaud, \emph{N\'{e}ron
  models}, Ergebnisse der Mathematik und ihrer Grenzgebiete (3), vol.~21,
  Springer-Verlag, 1990.

\bibitem{SGA3II}
Michel Demazure and Alexander Grothendieck~eds., \emph{S\'eminaire de
  {G}\'eom\'etrie {A}lg\'ebrique du {B}ois {M}arie - 1962/64 - {S}ch\'emas en
  groupes ({SGA} 3) - {V}ol. 2}, Lecture Notes in Math., vol. 152,
  Springer-Verlag, 1970.

\bibitem{Fontaine:padicrep1}
Jean-Marc Fontaine, \emph{Repr\'esentations $p$-adiques des corps locaux,
  1\`ere partie}, The Grothendieck Festschrift, vol. II (Pierre Cartier, Luc
  Illusie, Nicholas~Michael Katz, G\'{e}rard Laumon, Yu~Ivanovitch Manin, and
  Kenneth~Alan Ribet, eds.), Progr. Math., vol.~87, Birkh\"{a}user, Boston,
  1990, pp.~249--309.

\bibitem{Vasiu:dimensions}
Ofer Gabber and Adrian Vasiu, \emph{Dimensions of group schemes of
  automorphisms of truncated {B}arsotti--{T}ate groups}, Int. Math. Res. Not.
  \textbf{18} (2013), 4285--4333.

\bibitem{Greenberg:Schemata}
Marvin Greenberg, \emph{Schemata over local rings}, Ann. of Math. (2)
  \textbf{73}, no.~3, 624--648.

\bibitem{Katz:slopefiltration}
Nicholas~Michael Katz, \emph{Slope filtration of {$F$-crystals}}, Journ\'ees de
  G\'eom\'etrie Alg\'ebrique de Rennes (Rennes, 1978), Vol. I, Ast\'erisque
  (1979), no.~63, 113--163.

\bibitem{Vasiu:traversosolved}
Eike Lau, Marc-Hubert Nicole, and Adrian Vasiu, \emph{Stratifications of
  {Newton} polygon strata and {Traverso}'s conjectures for $p$-divisible
  groups}, Ann. of Math. \textbf{178} (2013), no.~3, 789--834.

\bibitem{Manin:formalgroups}
Yu~Ivanovitch Manin, \emph{The theory of commutative formal groups over fields
  of finite characteristic}, Uspekhi Mat. Nauk \textbf{18} (1963), no.~6,
  3--90.

\bibitem{Nie:Vasiuconj}
Nie Sian, \emph{On isomorphism numbers of ``{$F$}-crystals"}, Preprint, {\tt
  arXiv:1403.2095}, 2013.

\bibitem{Springer:linearalgebraicgroups}
Tonny~Albert Springer, \emph{Linear {A}lgebraic {G}roups}, Progress in
  Mathematics, vol.~9, Birkh\"auser, Boston, MA, 1998.

\bibitem{Traverso:pisa}
Carlo Traverso, \emph{Sulla classificazione dei gruppi analitici commutativi di
  caratteristica positiva}, Ann. Sc. Norm. Super. Pisa \textbf{23} (1969),
  no.~3, 481 -- 507.

\bibitem{Traverso:specializations}
\bysame, \emph{Specializations of {Barsotti--Tate} groups}, Symposia
  Mathematica, Vol. XXIV (Sympo., INDAM, Rome, 1979), Academic Press,
  London--New York, 1981, pp.~1--21.

\bibitem{Vasiu:CBP}
Adrian Vasiu, \emph{Crystalline boundedness principle}, Ann. Sci. \'{E}c. Norm.
  Sup. (4) \textbf{39} (2006), no.~2, 245--300.

\bibitem{Vasiu:levelm}
\bysame, \emph{Level $m$ stratifications of versal deformations of
  $p$-divisible groups}, J. Algebraic Geom. \textbf{17} (2008), no.~4,
  599--641.

\bibitem{Vasiu:modp}
\bysame, \emph{Mod $p$ classification of {S}himura {$F$}-crystals}, Math.
  Nachr. \textbf{283} (2010), no.~8, 1068--1113.

\bibitem{Vasiu:reconstructing}
\bysame, \emph{Reconstructing $p$-divisible groups from their truncations of
  small level}, Comment. Math. Helv. \textbf{85} (2010), no.~1, 165--202.

\bibitem{Viehmann:truncation1}
Eva Viehmann, \emph{Truncations of level $1$ of elements in the loop group of a
  reductive group}, Ann. of Math. \textbf{179} (2014), no.~3, 1009--1040.

\bibitem{Xiao:computing}
Xiao Xiao, \emph{Computing isomorphism numbers of {$F$}-crystals using level
  torsions}, J. Number Theory \textbf{132} (2012), no.~12, 2817--2835.

\end{thebibliography}

\end{document}